\documentclass[reqno]{amsart}
\usepackage{amssymb}
\usepackage[usenames, dvipsnames]{color}
\usepackage{hyperref}
\usepackage{marginnote}
\usepackage{todonotes}

\numberwithin{equation}{section}

\usepackage{verbatim}

\usepackage{esvect} 
\usepackage{esint} 
\usepackage{mathtools} 
\DeclarePairedDelimiter{\norm}{\lVert}{\rVert}

\newcommand\ep{\varepsilon}

\newcommand\nn{\mathbb{N}} 
\newcommand\zz{\mathbb{Z}}

\newcommand\rr{\mathbb{R}}
\newcommand\cc{\mathbb{C}}

\newcommand{\broy}{B_r ( \overline{y} ) }
\newcommand\hh{\mathbb H}
\newcommand\cmm{\mathcal{M}} 
\newcommand\cdd{\mathcal{D}} 
\newcommand\chh{\mathcal{H}} %
\newcommand\css{\mathcal{S}}
\newcommand\cbb{\mathcal{B}}
\newcommand\cxx{\mathcal{X}}

\DeclareMathOperator*{\supp}{supp}

\theoremstyle{plain}
\newtheorem{theorem}{Theorem}[section]
\newtheorem{lemma}[theorem]{Lemma}
\newtheorem{corollary}[theorem]{Corollary}
\newtheorem{proposition}[theorem]{Proposition}

\theoremstyle{definition}

\newtheorem{assumption}[theorem]{Assumption}

\theoremstyle{remark}
\newtheorem{remark}[theorem]{Remark}

\makeatletter
\@namedef{subjclassname@2020}{%
  \textup{2020} Mathematics Subject Classification}
\makeatother

\begin{document}
\title[Fractional wave equations with VMO coefficients]{Weighted mixed norm estimates for fractional wave equations with VMO coefficients}

\author[H. Dong]{Hongjie Dong}
\address[H. Dong]{Division of Applied Mathematics, Brown University, 182 George Street, Providence, RI 02912, USA}

\email{Hongjie\_Dong@brown.edu}
\thanks{H. Dong was partially supported by the Simons Foundation, grant \# 709545.}

\author[Y. Liu]{Yanze Liu}
\address[Y. Liu]{Department of Mathematics, Brown University, 151 Thayer Street, Providence, RI 02912, USA}

\email{Yanze\_Liu@brown.edu}


\begin{abstract}
This paper is a comprehensive study of $L_p$ estimates for time fractional wave equations of order $\alpha \in (1,2)$ in the whole space, a half space, or a cylindrical domain.
We obtain weighted mixed-norm estimates and solvability of the equations in both non-divergence form and divergence form when the leading coefficients have small mean oscillation in small cylinders.
\end{abstract}
\maketitle
\section{Introduction}

This paper is devoted to the study of $L_p$ (and $L_{p,q}$) estimates for non-divergence form fractional wave equations \eqref{nondiv} and divergence form fractional wave equations \eqref{div} with a non-local type time derivative term of the forms
\begin{align*}
   \partial_t^\alpha u - a^{ij}(t,x) D_{ij} u - b^i(t,x) D_i u - c(t,x) u = f(t,x), \stepcounter{equation}\tag{\theequation}\label{nondiv}
\end{align*}
\begin{align*}
    \partial_t^\alpha u
- D_i (a^{ij} D_j u + a^i u ) - b^i D_i u - cu = D_i g_i + f,\stepcounter{equation}\tag{\theequation}\label{div}
\end{align*}
with zero initial conditions $u(0,\cdot) = 0 $ and $\partial_t u(0,\cdot) = 0 $ in {a} cylindrical domain $(0,T) \times \Omega$, where $\partial_t^\alpha u$ is the Caputo fractional derivative of order $\alpha \in (1,2)$ and $\Omega$ is the whole space $\rr^d$, a half space $\rr_+^d:=\{(x^1,\ldots,x^d)\in \rr^d:x^1>0\}$, or a domain in $\rr^d$.
See the precise definition of $\partial_t^\alpha u$ in Section \ref{2}. When $\Omega = \rr^d_+$ or a domain, we also impose the zero Dirichlet boundary condition: $u(t,x) = 0$ for $t \in (0,T)$ and $x \in \partial \Omega$.

Fractional wave equations have many applications in mechanics and probability. For example, they govern the propagation of mechanical diffusive waves in viscoelastic media \cite{Mainardi}, and they also appear in the probability theory related to non-Markovian diffusion processes with a memory \cite{Metzler}. {There is a vast literature about} fractional parabolic and wave equations. Equation \eqref{nondiv} can be considered as a type of Volterra equation \cite{Pruss}
$$
u(t,x) + \int_0^t a(t-s)Au(s,x)\, ds = f(t,x),
$$
{where $A$ is a time-independent $\mathcal{R}$-sectorial operator.}
Using this interpretation, a purely operator-theoretic approach was applied in \cite{Zacher} to prove the existence of unique solutions under the assumptions that $a^{ij} = \delta^{ij}$, $\alpha \in (1,2)$, $p=q$, and that $\alpha$ and $p$ satisfy the algebraic condition $$
\alpha \not\in \big\{2/(2p-1), 2/(p-1),1/p+1,3/(2p-1)+1\big\}.
$$
In \cite{Kochubei}, {for} $\alpha\in(1,2)$, {assuming that} the coefficients depend only on $x$, and are bounded and uniformly H\"older continuous with a H\"older exponent $\gamma\in (2-2/\alpha,1]$, the author established the unique solvability of $\eqref{nondiv}$ by using Levi's method together with the fundamental solution of $\partial_t^{\alpha} - \Delta$ constructed earlier in \cite{Pskhu}.

There are also many work on $L_p$ ({or} $L_{p,q}$) estimates for fractional parabolic and wave equations. In \cite{P1}, under the assumptions that $\alpha\in (0,2)$, and $a^{ij}(t,x)$ are piecewise continuous in $t$ and uniformly continuous in $x$, the authors proved a priori $L_{p,q}$ estimates and the existence of unique solutions to \eqref{nondiv} {in the whole space}. Their proofs are based on estimates of the fundamental solution of $\partial_t^{\alpha} - \Delta$ and classical tools from the PDE theory, in particular the Fefferman-Stein sharp function theorem, the Hardy-Littlewood maximal function theorem, and the Calder\'on-Zygmund theorem. Quite recently, the results were extended in \cite{Han}, where mixed norms with Muckenhoupt weights were considered under the assumption that the leading coefficients are constants. In \cite{dong19}, the authors derived $L_p$ estimates and solvability of \eqref{nondiv} when $\alpha \in (0,1)$ under the weaker assumption that the leading coefficients have small mean oscillations with respect to the spatial variables only. See also \cite{dong20a} for results about divergence form equations with $\alpha\in (0,1)$. Later in \cite{dong20}, by using a mean oscillation argument, they further obtained weighted mixed-norm estimates for \eqref{nondiv} when $\alpha \in (0,1)$ under the same assumption on the coefficients. We note the proofs in \cite{dong19, dong20a} do not rely on the estimates of fundamental solutions. Instead, the argument is based on integration by parts, Sobolev type embeddings with fractional derivatives, a bootstrap argument, and a level set argument. However, {this} method fails for the case $\alpha\in(1,2)$ {{because} when $\alpha>1$, the integral
$$
\int_0^t (\partial_t^\alpha u(s,\cdot))|u(s,\cdot)|^{p-2}u(s,\cdot)\,ds
$$
may not always be nonnegative. Therefore, here} we use the results from \cite{P1}{, which are} derived from the fundamental solution mentioned above.

In this paper, we consider the equations \eqref{nondiv} and \eqref{div} when $\alpha\in(1,2)$, and generalize the previous results in \cite{P1,Han,Kochubei} by only assuming that the leading coefficients $a^{ij} =a^{ij}(t,x) $ locally have small mean oscillations as functions of $(t,x)$. See Assumption \ref{ass2} $(\gamma_0)$ for details. Furthermore, we obtain the results to equations in the half space and domains, and also consider the equations in the whole space with non-zero initial conditions. Our main theorem for non-divergence form equations reads that for any $p,q \in (1, \infty)$, Muckenhoupt weights $w_1\in A_p (\rr), w_2(x) \in A_q(\Omega)$, and $a^{ij}$'s with small mean oscillations, if $u$ satisfies (\ref{nondiv})  with the zero initial conditions (and the zero Dirichlet boundary condition in the cases of the half space and $C^{1,1}$ domains), then we have
$$
\sum_{i=0}^2\norm{D^i u}_{L_{p,q,w}(\Omega_T)} + \norm{ \partial_t^\alpha u}_{L_{p,q,w}(\Omega_T)}
\le N \norm{f}_{L_{p,q,w}(\Omega_T)},
$$
where $w(t,x) = w_1(t)w_2(x)$, $\Omega_T = (0,T) \times \Omega$, $N$ is independent of $u$ and $f$, and the definition of the $L_{p,q,w}$ norm is given in Section \ref{2}. Furthermore, for any $f \in L_{p,q,w}((0,T) \times \Omega)$, there exists a unique solution $u$ to (\ref{nondiv}) in the appropriate weighted mixed-norm Sobolev space defined in Section \ref{2}  with {the} zero initial (and boundary) conditions. Our main theorem for divergence form equations reads that for any $p,q \in (1, \infty)$, Muckenhoupt weights $w_1\in A_p (\rr), w_2(x) \in A_q(\Omega)$, and $a^{ij}$'s with small mean oscillations, if $u$ satisfies (\ref{div}) with the zero initial conditions (and the zero boundary condition in the cases of the half space and Lipschitz domains), then it holds that
$$
\sum_{i=0}^1\norm{D^i u}_{L_{p,q,w}(\Omega_T)} + \norm{ \partial_t^\alpha u}_{\hh^{-1}_{p,q,w}(\Omega_T)}
\le N \norm{f}_{L_{p,q,w}(\Omega_T)} + N \norm{g}_{L_{p,q,w}(\Omega_T)},
$$
where $N$ is independent of $u$, $f$ and $g$, and the definitions of the norms are given in Section \ref{2}. Furthermore, for any $f, g \in L_{p,q,w}((0,T) \times \Omega)$, there exists a unique solution $u$ to (\ref{div}) in the appropriate space defined in Section \ref{2} with the zero initial (and boundary) conditions.

For the proof of the a priori estimate {for} \eqref{nondiv} in the whole space, we adapt the mean oscillation argument in \cite{dong20} by establishing a H\"older estimate of $D^2u$, where $u$ {is a strong} solution to \eqref{nondiv}. For this, we first consider the case when $a^{ij} = \delta^{ij}$ without lower-order terms, and decompose $u = w + v$, where $w$ is constructed using Fourier series in rectangular {cylindrical} domains, and $v = u-w$ satisfies a homogeneous equation. To estimate $v$ and $w$, we consider the infinite cylinder $(-\infty,t_0)\times B_r(x_0)$ instead of the {usual} parabolic cylinder $Q_r(t_0,x_0)$. We also apply a delicate bootstrap argument with Sobolev embedding similar to that of \cite{dong20}. However, compared to the case {considered} in \cite{dong20}, where $\alpha\in(0,1)$, and only the first-order derivative in time is taken in the definition of $\partial_t^\alpha$, the estimate of $v$ for $\alpha\in(1,2)$ is more involved due to the presence of an extra first-order term in the fractional derivative when we take the cutoff function in time to apply the bootstrap argument. To overcome this difficulty, we take a sequence of cutoff functions in time and apply an interpolation inequality for the time derivatives. For non-divergence form equations in the half space, the results follow from taking extensions, and in smooth domains by a partition of unity argument together with S. Agmon's idea, adapted from the proofs in \cite[Ch. 8]{k}. It is also worth noting that the results for divergence form equations follow from the results for non-divergence form equations with a similar perturbation argument.

The remaining part of the paper is organized as follows. In Section \ref{2}, we introduce notation, definitions, and the main results of the paper. Equations in non-divergence form (\ref{nondiv}) are studied in Sections \ref{3}, \ref{4}, and \ref{5}. Particularly, in Section \ref{3}, we estimate the mean oscillation by using a decomposition mentioned above. In Section \ref{4}, we complete the proof {of} the theorem for non-divergence form equations in the whole space. In Section \ref{5}, we prove the results for non-divergence form equations in {the} half space and smooth domains. Finally, the results for divergence form equations (\ref{div}) are proved in Section \ref{6}. In Appendix \ref{A}, we prove miscellaneous lemmas used in the main proofs. In Appendix \ref{B}, we give some details about the proofs for equations in domains, such as the extension of weights. In Appendix \ref{C}, we prove the results for equations in the whole space with non-zero initial conditions.

\section{Notation and Main Results}\label{2}
We first introduce some notation used throughout the paper. For $\alpha \in (0,2)$ and $S \in \rr$, we denote
\[
I^\alpha_S u
= \frac{1}{\Gamma(\alpha)}
\int_S^t (t - s)^{\alpha - 1} u (s,x) \, ds,
\]
and for $\alpha \in (1,2)$
\[
\partial_t^\alpha u
= \frac{1}{\Gamma(2 - \alpha)}
\int_0^t (t-s)^{1-\alpha} \partial_t^2 u (s,x) \, ds.
\]
Note that we have $\partial_t^\alpha u = \partial^2_t I_0^{2-\alpha} u$ for a sufficiently smooth $u$ with $u(0,x) = 0$ and $\partial_t u(0,x) = 0$. In some places, without assuming $S=0$, we still use $\partial_t^\alpha u$ to indicate $\partial_t^2 I_S^{2-\alpha} u$ when $u(S,x) = 0$ and $\partial_t u(S,x) = 0$. Moreover, for $\beta \in (0,1)$, we use $\partial_t^\beta u $ to indicate $\partial_t I_S^{1-\beta} u$ when $u(S,x) = 0$.

For $\alpha \in (1,2)$, $r_1,r_2 >0$, and $(t,x)\in \rr^{d+1}$, we denote the parabolic cylinder by
\[
Q_{r_1, r_2} (t,x) = (t-r_1^{2/\alpha}, t) \times B_{r_2} (x)
\quad \text{and}\quad
Q_r(t,x) = Q_{r,r}(t,x),
\]
where $B_{r_2} (x) = \{ y \in \rr^d: |y - x| < r_2 \}$.
We often write $B_r$ and $Q_r$ for $B_r(0)$ and $Q_r(0,0)$. Furthermore, for $x\in\rr^d$ and $r>0$, we denote $x' = (x^2,\ldots,x^d)\in\rr^{d-1}$ and $B^{'}_{r} (x') := \{ y' \in \rr^{d-1} : |x'-y'| < {r}\}$.
Next, for $p \in (1, \infty)$, $k \in \{ 1,2,\ldots \}$ and a domain $\Omega \subset \rr^{k}$, let $A_p (\Omega, dx)$ be the set of all non-negative functions $w$ on $\Omega$ such that
\[
[w]_{A_p(\Omega)}
:= \sup_{x_0 \in \rr^k \cap \Omega, r> 0}
\biggl(
\fint_{B_r(x_0)\cap \Omega} w(x) \, dx
\biggr) \biggl(
\fint_{B_r(x_0)\cap \Omega} (w(x))^{\frac{-1}{p-1}} \, dx
\biggr)^{p-1}
< \infty,
\]
where $B_r(x_0) = \{ x \in \rr^k: |x - x_0| < r \}$. Recall that $[w]_{A_p} \geq 1$. Furthermore, for a constant $K_1 > 0$, we write $[w]_{p,q,\Omega} \leq K_1$ and $[w]_{p,q} \le K_1$ when $\Omega= \rr^d$ if $w = w_1(t) w_2(x)$ for some $w_1$ and $w_2$ satisfying
\[
w_1(t) \in A_p (\rr, dt),
\quad
w_2(x) \in A_q(\Omega, dx),
\quad
[w_1]_{A_p} \leq K_1,
\quad \text{and}\quad
[w_2]_{A_q(\Omega)} \leq K_1.
\]

For a domain $\Omega \subset \rr^d$,  let $(0,T) \times \Omega :=\Omega_T$ and particularly, $(0,T) \times \rr^d :=\rr^d_T $.
For a given $w(t,x) = w_1(t) w_2 (x)$, where $(t,x) \in \rr \times \Omega$, $w_1 \in A_p (\rr, dt)$, and $w_2 \in A_q (\Omega, dx)$, we {denote} $L_{p,q,w} ( \Omega_T )$ to be the set of all measurable functions $f$ defined on $\Omega_T$ satisfying
\[
\norm{f}_{L_{p,q,w} ( \Omega_T )}:=
\biggl( \int_0^T
\biggl( \int_{\Omega}
| f(t,x)|^q w_2(x) \, dx
\biggr)^{p/q}
w_1(t) \, dt
\biggr)^{1/p}
< \infty.
\]
When $p=q$ and $w=1$, $L_{p,q,w} ( \Omega_T )$ becomes the usual Lebesgue space $L_p ( \Omega_T )$. We write $u \in \hh_{p,q,w,0}^{\alpha, 2} ( \Omega_T ) $
if there exists a sequence of smooth functions $\{ u_n \}$ such that  $u_n \in C^\infty \bigl( [0,T] \times \Omega \bigr)$ vanishing for large $|x|$,  $u_n(0,x) = 0,  \partial_t u_n(0,x) = 0$, and
\[
\norm{u_n - u}_{\hh_{p,q,w}^{\alpha, 2} ( \Omega_T )}
:= \sum_{i=0}^2\norm{D^i u_n - D^i u}_{p,q,w}
+ \norm{\partial_t^\alpha u_n - \partial_t^\alpha u}_{p,q,w}
\to 0
\]
as $n \to \infty$,
where $\norm{\cdot}_{p,q,w} = \norm{\cdot}_{L_{p,q,w} ( \Omega_T )}$.
We often write $\hh_{p,w,0}^{\alpha, 2} ( \Omega_T )$ if $p=q$ and use the notation $u \in  \hh_{p,0,\mathrm{loc}}^{\alpha, 2} ( \rr^d_T )$ to indicate the function satisfying
\[
u \in  \hh_{p,0}^{\alpha, 2} ( (0,T) \times B_R )  := \hh_{p,1,0}^{\alpha, 2} ( (0,T) \times B_R )
\]
for any $R>0$. Furthermore, we denote
$u \in \mathring{\hh}^{\alpha, 2}_{p,q,w,0} (\Omega_T )$ if $u \in \hh^{\alpha, 2}_{p,q,w,0} (\Omega_T)$ satisfies $u=0$ on $(0,T) \times \partial \Omega$.
In other words, the defining sequence $u_n \to u$ in $\hh_{p,q,w}^{\alpha, 2}(\Omega_T )$ can be taken to satisfy the following conditions:
\[
u_n \in C_0^\infty ([0,T] \times \overline{\Omega}),
\quad
u_n (0, \cdot) = 0,
\quad
\partial_t u_n(0, \cdot) = 0,
\quad
u_n(\cdot, x) = 0 \quad \text{for}\quad  x \in \partial \Omega.
\]

In Section \ref{6}, we consider equations in divergence form. Therefore, we define the following. For $u\in L_{p,q,w}( (S,T) \times \Omega)$, we say $\partial_t^\alpha u \in \hh_{p,q,w}^{-1}( (S,T) \times \Omega)$ if there exist $f,g_i \in  L_{p,q,w}( (S,T) \times \Omega)$, $i= 1,\ldots,d$, such that for any $\phi \in C_0^\infty ([S,T) \times \Omega)$,
\[
\int_S^T\int_\Omega \partial_t^\alpha u  \phi
=  \int_S^T\int_\Omega f \phi-\int_S^T\int_\Omega g_i D_i \phi, \stepcounter{equation}\tag{\theequation}\label{6def}
\]
and we write
\begin{align*}
    &\norm{\partial_t^\alpha u}_{\hh^{-1}_{p,q,w}( (S,T) \times \Omega)} \\
    &= \inf \Big\{ \sum_{i=1}^d \norm{g_i}_{L_{p,q,w}( (S,T) \times \Omega)} + \norm{f}_{L_{p,q,w}( (S,T) \times \Omega)}: (\ref{6def}) \,\text{is satisfied} \Big\}.
\end{align*}
Let
$\chh_{p,q,w}^{\alpha, 1} ((S,T) \times \Omega) $ be the collection of functions $u \in L_{p,q,w}( (S,T) \times \Omega)$ such that
$Du \in L_{p,q,w}( (S,T) \times \Omega)$ and $\partial_t^\alpha u \in \hh_{p,q,w}^{-1}( (S,T) \times \Omega)$, and denote
$$
\norm{u}_{\chh_{p,q,w}^{\alpha, 1}( (S,T) \times \Omega)}=
\norm{\partial_t^\alpha u}_{\hh^{-1}_{p,q,w}( (S,T) \times \Omega)}
+ \sum_{i=0}^1\norm{D^i u}_{L_{p,q,w}( (S,T) \times \Omega)}.$$
Then, we define $u\in \chh_{p,q,w,0}^{\alpha, 1} ((S,T) \times \Omega)$ if $u\in\chh_{p,q,w}^{\alpha, 1} ((S,T) \times \Omega)$, and there exists a sequence of $\{u_n\} \subset C^\infty ([S,T] \times \Omega)$ vanishing for large $|x|$, $u_n(S,x) = 0$, $\partial_t u_n(S,x) = 0$ and
\[
\norm{u_n - u}_{\chh^{\alpha,1}_{p,q,w}( (S,T) \times \Omega)} \to 0 \quad \text{as}\quad  n \to \infty.
\]
We denote
$u\in\mathring{\chh}^{\alpha, 1}_{p,q,w,0} ((S, T) \times \Omega )$ if $u\in \chh^{\alpha, 1}_{p,q,w,0} ((S, T) \times \Omega )$ satisfies $u=0$ on $(0,T) \times \partial \Omega$.
In other words, the defining sequence $u_n \to u$ in $\chh_{p,q,w}^{\alpha, 1}((S, T) \times \Omega )$ can be taken to satisfy the following conditions:
\[
u_n \in C_0^\infty ([S,T] \times \overline{\Omega}),
\quad
u_n (S, \cdot) = 0,
\quad
\partial_t u_n(S, \cdot) = 0,
\quad
u_n(\cdot, x) = 0 \quad \text{for}\quad  x \in \partial \Omega.
\]
Moreover, we say that $u \in \chh_{p,q,w,0}^{\alpha, 1} ((S,T) \times \Omega) $
satisfies the divergence form equation
\[
\partial_t^\alpha u
- D_i (a^{ij} D_j u + a^i u ) - b^i D_i u - cu = D_i g_i + f
\]
for $f, g_i\in L_{p,q,w} ((S,T) \times \Omega)$ if
\begin{align*}
&\int_S^T \int_\Omega I_S^{2-\alpha}u
\partial_t^2\phi
+\int_S^T\int_\Omega a^{ij} D_j u D_i \phi
+a^i u D_i \phi-b^i D_i u \phi
- cu\phi \\
&= \int_S^T\int_\Omega f \phi - g_i D_i \phi
\end{align*}
for any $\phi \in C_0^\infty ([S,T) \times \Omega)$.

In this paper, we use maximal functions and strong maximal functions defined, respectively, by
\[
\cmm f (t,x) = \sup_{Q_r (s,y) \ni (t,x)}
\fint_{Q_r (s,y)}
|f(r,z) | \chi_\cdd \, dz \, dr
\]
and
\[
(\css \cmm f ) (t,x)= \sup_{Q_{r_1, r_2} (s,y) \ni (t,x)}
\fint_{Q_{r_1, r_2} (s,y)}
|f(r,z) | \chi_\cdd \, dz \, dr
\]
for $f\in L_{1,\text{loc}}$ defined on $\cdd \subset \rr^{d+1}$ and $(t,x) \in \cdd$.

In order to consider equations with non-zero initial conditions, we introduce the trace space as follows. For $p, q \in (1, \infty)$, a non-integer ${\theta} > 0$ with $k$ being the integer part of ${\theta}$, and $w_2\in A_q(\rr^d)$, we define
\[
\cbb^{\theta}_{p,q,w_2} (\rr^d)
:= \{ u \in L_{q, w_2} (\rr^d ) : \norm{u}_{ \cbb^{\theta}_{p,q,w_2} (\rr^d) } < \infty \},
\]
where
\[
\norm{u}_{ \cbb^{\theta}_{p,q,w_2} (\rr^d)}
= [u]_{\cbb^{\theta}_{p,q,w_2} (\rr^d)}
+ \sum_{|\alpha| \leq k} \norm{D^\alpha u}_{L_{q,w_2} (\rr^d)}
\]
and
\[
[u]_{\cbb^{\theta}_{p,q,w_2} (\rr^d)}
= \left(
\int_{\rr^d}
\norm{D^k u (\cdot + y) - D^k u (\cdot)}^p_{L_{q,w_2}}
|y|^{-d - p({\theta} -k )} \, dy
\right)^{1/p}.
\]
Note that when $p=q$ and $w_2 \equiv 1$, $\cbb^{\theta}_{p,q,w_2} (\rr^d) = \cbb^{\theta}_{p,p} (\rr^d)$, the usual Besov space.
For ${\theta_0} \in \rr$,
we denote $\cxx_0 = L_{q,w_2}$ when ${\theta_0} < 0, \cxx_0 =\cbb^{\theta_0}_{p,q,w_2} (\rr^d) $ when ${\theta_0} > 0$ is a non-integer, and $\cxx_0 = \cbb^{{\theta_0}+\ep}_{p,q,w_2} (\rr^d)$ for arbitrary $\ep > 0$ when ${\theta_0}$ is a nonnegative integer. For a different $\theta_1 \in \rr$, we define $\cxx_1$ similarly.

Finally, we write $N = N (\ldots)$ if the constant $N$ depends only on the parameters in the parentheses.

Next, we present the assumptions for the coefficients and domains. Assumption \ref{ass1} for a fixed $\delta$ is imposed throughout the paper.
\begin{assumption}\label{ass1}
There exists $\delta \in (0,1)$ such that for any $(t,x) \in \rr^{d+1}$ and $\xi \in \rr^d$,
\[
    a^{ij} (t,x) \xi^i \xi^j \geq \delta |\xi|^2 \quad \text{and}\quad  |a^{ij}| ,|b^i|, |c|  \le \delta^{-1}.
\]
\end{assumption}

\begin{assumption}[$\gamma_0$]\label{ass2}
There exists $R_0 \in (0, 1]$ such that for any $(t_0, x_0) \in \rr^{d+1}$ and $0<r\le R_0$,
\[
    \sup_{i,j} \fint_{Q_r (t_0, x_0)} |a^{ij} - \overline{a}^{ij} | \leq \gamma_0,
\]
where
  \[
    \overline{a}^{ij} = \fint_{Q_r (t_0, x_0)} a^{ij} \, dx dt.
  \]
\end{assumption}

\begin{remark}\label{ass2.3}
If we take $x_0\in \Omega = \rr^d_+$ and define
 \[
    \overline{a}^{ij} = \fint_{Q_r (t_0, x_0)\cap \Omega_T} a^{ij} \, dx dt,
\]
then by Assumption \ref{ass2} $(\gamma_0)$, we have
\[
    \sup_{i,j} \fint_{Q_r (t_0, x_0)\cap\Omega_T} |a^{ij} - \overline{a}^{ij} | \leq 4\gamma_0.
\]
\end{remark}
\begin{assumption}[$\theta$]\label{ass3}
    There exists $R_2 \in (0,1]$ such that for any $z \in \partial \Omega$, there exists a Lipschitz function $\phi: \rr^{d-1} \to \rr$ {satisfying}
    $$\Omega \cap B_{R_2}(z) = \{x\in B_{R_2}(z) : x^1 > \phi(x')\}$$
    and
    \begin{align*}
        \sup_{x',y'\in B^{'}_{R_2} (z'), x'\neq y'} \frac{|\phi(x') - \phi(y')|}{|x'-y'|} \le \theta
    \end{align*}
    in a new coordinate system.
\end{assumption}

We need a stronger assumption on the domains when non-divergence form equations are considered.
\begin{assumption}[$\theta$]\label{ass4}
    There exists $R_2 \in (0,1]$ such that for any $z \in \partial \Omega$, there exists a {$C^{1,1}$} function $\phi: \rr^{d-1} \to \rr$ satisfying
    $$\Omega \cap B_{R_2}(z) = \{x\in B_{R_2}(z) : x^1 > \phi(x')\},$$
    \begin{align*}
        \sup_{x',y'\in B^{'}_{R_2} (z'), x'\neq y'} \frac{|\phi(x') - \phi(y')|}{|x'-y'|} \le \theta,\quad\text{and}\,\,\sup_{B'_{R_2}(z')}|D^2 \phi|\le \delta^{-1}
    \end{align*}
in a new coordinate system.
\end{assumption}

With the assumptions {above}, we are ready to state the main theorems of this paper.
\begin{theorem}\label{main}
Let $\alpha \in (1,2)$, $T \in (0, \infty)$, $p,q \in (1, \infty)$, $ K_1 \in [1, \infty)$, and $[w]_{p,q} \le K_1$.
There exists $\gamma_0 = \gamma_0 (d, \delta, \alpha, p, q, K_1 ) \in (0,1)$ such that under Assumption \ref{ass2} $(\gamma_0)$, the following holds.
For any $u \in \hh_{p,q,w,0}^{\alpha, 2} ( \rr^d_T )$ satisfying
\[
\partial_t^\alpha u - a^{ij} D_{ij} u - b^i D_i u - cu = f \quad \text{in}\quad  \rr^d_T,
\stepcounter{equation}\tag{\theequation}\label{main estimate}
\]
we have
\[
\norm{\partial_t^\alpha u}_{p,q,w}
+ \norm{u}_{p,q,w}
+ \norm{Du}_{p,q,w}
+ \norm{D^2 u}_{p,q,w}
\leq N \norm{f}_{p,q,w},\stepcounter{equation}\tag{\theequation}\label{99999}
\]
where
$\norm{\cdot}_{p,q,w}
= \norm{\cdot}_{L_{p,q,w} (\rr^d_T) }$ and
$
N = N(d, \delta, \alpha, p, q, K_1, R_0, T).$
Moreover, for any $f \in L_{p,q,w} ( \rr^d_T )$, there exists a unique solution $u \in \hh_{p,q,w,0}^{\alpha, 2} ( \rr^d_T )$ to (\ref{main estimate}).
\end{theorem}

Next, we introduce the results for the half space and smooth domains under the zero initial conditions and the zero boundary condition.
\begin{theorem}\label{main2}
Let  $\Omega = \rr^d_+$, $\alpha \in (1,2)$, $T \in (0, \infty)$, $p,q \in (1, \infty)$, $K_1 \in [1, \infty)$ and  $[w]_{p,q, \Omega} \le K_1$.
There exists $\gamma_0 = \gamma_0 (d, \delta, \alpha, p, q, K_1 ) \in (0,1)$ such that under Assumption \ref{ass2} $(\gamma_0)$, the following holds.
For any $u \in \mathring{\hh}_{p,q,w,0}^{\alpha, 2} ( \Omega_T )$ satisfying
\[
\partial_t^\alpha u - a^{ij} D_{ij} u - b^i D_i u - cu = f \quad \text{in}\quad  \Omega_T,
\stepcounter{equation}\tag{\theequation}\label{main estimate2}
\]
 we have
\[
\norm{\partial_t^\alpha u}_{p,q,w}
+ \norm{u}_{p,q,w}
+ \norm{Du}_{p,q,w}
+ \norm{D^2 u}_{p,q,w}
\leq N \norm{f}_{p,q,w},
\]
where
$\norm{\cdot}_{p,q,w}
= \norm{\cdot}_{L_{p,q,w} (\Omega_T) }$ and $N = N(d, \delta, \alpha, p, q, K_1, R_0, T).$
Moreover, for any $f \in L_{p,q,w} ( \Omega_T)$, there exists a unique solution
$u \in \mathring{\hh}_{p,q,w,0}^{\alpha, 2} ( \Omega_T)$ to (\ref{main estimate2}).
\end{theorem}

\begin{theorem}\label{nondivdomain}
Let $\Omega $ be a domain, $\alpha \in (1,2)$, $T \in (0, \infty)$, $p,q \in (1, \infty)$, $K_1 \in [1, \infty)$ and $[w]_{p,q, \Omega} \le K_1$.
There exist constants ${\gamma_0} = \gamma_0 (d, \delta, \alpha, p, q, K_1 ) \in (0,1)$ and ${\theta=} \theta(d, \delta, \alpha, p, q, K_1) \in (0,{1/3})$ such that under Assumption \ref{ass2} $(\gamma_0)$ and Assumption {\ref{ass4}} $(\theta)$, the following holds.
For any $u \in \mathring{\hh}_{p,q,w,0}^{\alpha, 2} ( \Omega_T )$ satisfying
\[
\partial_t^\alpha u -a^{ij} D_{ij} u - b^i D_i u - cu = f \quad \text{in}\quad  \Omega_T,
\stepcounter{equation}\tag{\theequation}\label{10041}
\]
we have
\[
\norm{\partial_t^\alpha u}_{p,q,w}
+ \norm{u}_{p,q,w}
+ \norm{Du}_{p,q,w}
+ \norm{D^2 u}_{p,q,w}
\leq N \norm{f}_{p,q,w},\stepcounter{equation}\tag{\theequation}\label{jj1}
\]
where
$\norm{\cdot}_{p,q,w}
= \norm{\cdot}_{L_{p,q,w} (\Omega_T) }$ and $N = N(d, \delta, \alpha, p, q, K_1, R_0, R_2, T).$
Moreover, for any $f \in L_{p,q,w} ( \Omega_T)$, there exists a unique solution
$u \in \mathring{\hh}_{p,q,w,0}^{\alpha, 2} ( \Omega_T)$ to (\ref{10041}).
\end{theorem}

Next we state the results for equations in divergence form under the zero initial conditions and the zero boundary condition.
\begin{theorem}\label{6main}
Let $\alpha \in (1,2)$, $T \in (0, \infty)$, $p,q \in (1, \infty)$,  $K_1 \in [1, \infty)$, and $[w]_{p,q} \le K_1$.
There exists $\gamma_0 = \gamma_0 (d, \delta, \alpha, p, q, K_1 ) \in (0,1)$ such that under Assumption \ref{ass2} $(\gamma_0)$, the following holds.
For any $u \in \chh_{p,q,w,0}^{\alpha, 1} ( \rr^d_T )$ satisfying
\[
\partial_t^\alpha u
- D_i (a^{ij} D_j u + a^i u ) - b^i D_i u - cu = D_i g_i + f \quad \text{in}\quad \rr^d_T,
\stepcounter{equation}\tag{\theequation}\label{6main estimate}
\]
we have
\[
\norm{\partial_t^\alpha u}_{\hh^{-1}_{p,q,w}}
+ \norm{u}_{p,q,w}
+ \norm{Du}_{p,q,w}
\leq N (\norm{f}_{p,q,w}+\norm{g}_{p,q,w}), \stepcounter{equation}\tag{\theequation}\label{6gg}
\]
where $\norm{\cdot}_{p,q,w}
= \norm{\cdot}_{L_{p,q,w} (\rr^d_T) }$ and $N = N(d, \delta, \alpha, p, q, K_1, R_0, T).$
{Moreover}, for any $f,g_i \in L_{p,q,w} ( \rr^d_T )$, there exists a unique solution $u \in \chh_{p,q,w,0}^{\alpha, 1} ( \rr^d_T )$ to (\ref{6main estimate}).
\end{theorem}

\begin{theorem}\label{main4}
Let $\Omega = \rr^d_+$, $\alpha \in (1,2)$, $T \in (0, \infty)$, $p,q \in (1, \infty)$, $K_1 \in [1, \infty)$, and $[w]_{p,q,\Omega} \le K_1$.
There exists $\gamma_0 = \gamma_0 (d, \delta, \alpha, p, q, K_1 ) \in (0,1)$ such that under Assumption \ref{ass2} $(\gamma_0)$, the following holds.
For any $u \in \mathring{\chh}_{p,q,w,0}^{\alpha, 1} ( \Omega_T )$ satisfying
\[
\partial_t^\alpha u
- D_i (a^{ij} D_j u + a^i u ) - b^i D_i u -cu = D_i g_i + f \quad \text{in}\quad  \Omega_T,
\stepcounter{equation}\tag{\theequation}\label{915}
\]
we have
\[
\norm{\partial_t^\alpha u}_{\hh^{-1}_{p,q,w}}
+ \norm{u}_{p,q,w}
+ \norm{Du}_{p,q,w}
\leq N (\norm{f}_{p,q,w}+\norm{g}_{p,q,w}), \stepcounter{equation}\tag{\theequation}\label{9184}
\]
where
$\norm{\cdot}_{p,q,w}
= \norm{\cdot}_{L_{p,q,w} (\Omega_T) }$ and $N = N(d, \delta, \alpha, p, q, K_1, R_0, R_2, T).$
Moreover, for any $f, g_i \in L_{p,q,w} ( \Omega_T)$, there exists a unique solution
$u \in \mathring{\chh}_{p,q,w,0}^{\alpha, 1} ( \Omega_T)$ to (\ref{915}).
\end{theorem}

\begin{theorem}\label{divdomain}
Let $\Omega $ be a domain $\alpha \in (1,2)$, $T \in (0, \infty)$, $p,q \in (1, \infty)$, $K_1 \in [1, \infty)$, and $[w]_{p,q,\Omega} \le K_1$.
There exist $\gamma_0 = \gamma_0 (d, \delta, \alpha, p, q, K_1) \in (0,1)$ and $ \theta = \theta(d, \delta, \alpha, p, q, K_1) \in (0,{1/3})$ such that under Assumption  \ref{ass2} $(\gamma_0)$ and Assumption  \ref{ass3} $(\theta)$, the following holds.
For any $u \in \mathring{\chh}_{p,q,w,0}^{\alpha, 1} ( \Omega_T )$ satisfying
\[
\partial_t^\alpha u
- D_i (a^{ij} D_j u + a^i u ) - b^i D_i u - cu = D_i g_i + f \quad \text{in}\quad  \Omega_T,
\stepcounter{equation}\tag{\theequation}\label{10042}
\]
we have
\[
\norm{\partial_t^\alpha u}_{\hh^{-1}_{p,q,w}}
+ \norm{u}_{p,q,w}
+ \norm{Du}_{p,q,w}
\leq N (\norm{f}_{p,q,w}+\norm{g}_{p,q,w}),
\]
where
$\norm{\cdot}_{p,q,w}
= \norm{\cdot}_{L_{p,q,w} (\Omega_T) }$ and $N = N(d, \delta, \alpha, p, q, K_1, R_0,R_2,T).$
Also, for any $f, g_i \in L_{p,q,w} ( \Omega_T)$, there exists a unique solution
$u \in \mathring{\chh}_{p,q,w,0}^{\alpha, 1} ( \Omega_T)$ to (\ref{10042}).
\end{theorem}

Finally, we state the results for equations in the whole space with nonzero initial conditions,
\begin{align*}
    \begin{cases}
    \partial_t^\alpha u - (a^{ij} D_{ij} u + b^i D_i u + cu) = f &\quad \text{in}\quad  {\rr^d_T} \\
    u(0, \cdot) = u_0(\cdot) &\quad \text{in}\quad  \rr^d\\
    u_t(0, \cdot) = u_1(\cdot) &\quad \text{in}\quad  \rr^d
    \end{cases} \stepcounter{equation}\tag{\theequation}\label{cong}
\end{align*}
that is,
\[
D_t^\alpha (u - u_0- tu_1) - (a^{ij} D_{ij} u + b^i D_i u + cu) = f \quad \text{in}\quad  {\rr^d_T} .
\]

\begin{corollary}\label{initial}
Let $\alpha \in (1,2)$, $T \in (0, \infty)$, $p,q \in (1, \infty)$, $K_1 \in [1, \infty)$, $\mu \in (-1, p-1)$, and $w = w_1(t) w_2 (x)$, where $w_1(t) = t^\mu$ and $w_2(x) \in A_q (\rr^d, dx)$ with $[w_2]_{A_q} \leq K_1$. Let $\theta_0 = 2 - 2(1 + \mu) / (\alpha p)$, $\theta_1 = 2 - 2/\alpha - 2(1 + \mu) / (\alpha p)$, and recall the definition{s} of the spaces $\cxx_0$ and $\cxx_1$ above. There exists $\gamma_0 = \gamma_0 (d, \delta, \alpha, p, q, K_1, \mu) \in (0,1)$ such that, under Assumption \ref{ass2} $(\gamma_0)$, the following holds. For any $f \in L_{p,q,w} ({\rr^d_T})$, $u_0 \in \cxx_0 $, and  $u_1 \in \cxx_1 $, there exists a unique function $u$ in ${\rr^d_T}$ satisfying \eqref{cong} with the estimate
\begin{align}
&\norm{\partial_t^\alpha u}_{p,q,w} + \norm{u}_{p,q,w} + \norm{Du}_{p,q,w}
+ \norm{D^2 u}_{p,q,w}\notag\\
&\leq N \norm{f}_{p,q,w} + N \norm{u_0}_{\cxx_0} +  N \norm{u_1}_{\cxx_1}, \label{88998}
\end{align}
where $\norm{\cdot}_{p,q,w} = \norm{\cdot}_{L_{p,q,w} ({\rr^d_T})}$ and $N = N(d, \delta, \alpha, p, q, K_1, K_0, R_0, T, \mu)$.
\end{corollary}

For equations in divergence form
\begin{align*}
    \begin{cases}
    \partial_t^\alpha u
- D_i (a^{ij} D_j u + a^i u ) - b^i D_i u - cu = D_i g_i + f &\quad \text{in}\quad  {\rr^d_T} \\
    u(0, \cdot) = Du_{0,1}(\cdot) + u_{0,2}(\cdot) &\quad \text{in}\quad  \rr^d\\
    u_t(0, \cdot) = Du_{1,1}(\cdot) + u_{1,2}(\cdot) &\quad \text{in}\quad  \rr^d
    \end{cases} \stepcounter{equation}\tag{\theequation}\label{cong2}
\end{align*}
that is,
\[
D_t^\alpha (u - u_0- tu_1) - {D_i (a^{ij} D_j u + a^i u ) - b^i D_i u - cu} = f \quad \text{in}\quad  {\rr^d_T},
\]
we have the following results.
\begin{corollary}\label{initial2}
Let $\alpha \in (1,2)$, $T \in (0, \infty)$, $p,q \in (1, \infty)$, $K_1 \in [1, \infty)$, $\mu \in (-1, p-1)$, and $w = w_1(t) w_2 (x)$, where $w_1(t) = t^\mu$ and $w_2(x) \in A_q (\rr^d, dx)$ with $[w_2]_{A_q} \leq K_1$. Let $\theta_0 = 2 - 2(1 + \mu) / (\alpha p)$, $\theta_1 = 2 - 2/\alpha - 2(1 + \mu) / (\alpha p)$, and recall the definition{s} of the spaces $\cxx_0$ and $\cxx_1$ above. There exists $\gamma_0 = \gamma_0 (d, \delta, \alpha, p, q, K_1, \mu) \in (0,1)$ such that, under Assumption \ref{ass2} $(\gamma_0)$, the following holds. For any $f,g \in L_{p,q,w} ({\rr^d_T})$, $u_{0,1},u_{0,2} \in \cxx_0 $, and  $u_{1,1},u_{1,2} \in \cxx_1 $, there exists a unique function $u$ in ${\rr^d_T}$ satisfying \eqref{cong2} with the estimate
\begin{align*}
   & \norm{\partial_t^\alpha u}_{\hh^{-1}_{p,q,w}}
+ \norm{u}_{p,q,w}+ \norm{Du}_{p,q,w}\\
&\leq N \norm{|f|+|g|}_{p,q,w} + N \norm{|u_{0,1}|+ |u_{0,2}|}_{\cxx_0} +  N \norm{|u_{1,1}|+ |u_{1,2}|}_{\cxx_1},
\end{align*}
where $\norm{\cdot}_{p,q,w} = \norm{\cdot}_{L_{p,q,w} ({\rr^d_T})}$ and $N = N(d, \delta, \alpha, p, q, K_1, K_0, R_0, T, \mu)$.
\end{corollary}

At the end of this section, we derive an useful property of functions satisfying Assumption \ref{ass2} $(\gamma_0)$.
The following lemma implies that Assumption \ref{ass2} $(\gamma_0)$ gives a control over the mean oscillation on sets that are large in the time direction and small in the spatial directions.
\begin{lemma}\label{2.1}
If $a$ satisfies Assumption \ref{ass2} $(\gamma_0)$, $h \ge r^{2/\alpha} =  \widetilde{r}$, and $ r \leq R_1 \leq R_0$, where $R_1= 2^{-\alpha/2}R_0$.
Then
\[
\fint_{(t_0 - h, t_0)} \fint_{B_r(x_0)}
|a - \overline{a}|
< N  hr^{-2/\alpha}  \gamma_0,
\]
where $\overline{a} = \fint_{Q_r (t_0, x_0) }   a $ and $N=N(\alpha,d) $.
\end{lemma}

\begin{proof}
Suppose that $ (k-1) \widetilde{r} \leq h < k \widetilde{r}$ for some integer $k\ge 2$. Then,
\begin{align*}
\fint_{(t_0 - h, t_0)} \fint_{B_r(x_0)}
|a - \overline{a}|
&\leq \frac{1}{h}  \int^{t_0}_{t_0 - k \widetilde{r}}
\fint_{B_r (x_0)} | a - \overline{a}| \\
&\leq \frac{\widetilde{r} }{ h}
\sum_{j=0}^{k - 1} \fint_{t_0 - (j+1) \widetilde{r} }^{t_0 - j \widetilde{r} }
\fint_{B_r (x_0)} |a - \overline{a}|.
\stepcounter{equation}\tag{\theequation}\label{bb1}\end{align*}
For each term in the sum of the right-hand side of \eqref{bb1}, we have
\begin{align*}
\fint_{t_0 - (j+1) \widetilde{r} }^{t_0 - j \widetilde{r} }
\fint_{B_r (x_0)} |a - \overline{a}|
&\leq  \fint_{t_0 - (j+1) \widetilde{r} }^{t_0 - j \widetilde{r} }
\fint_{B_r (x_0)} |a - \overline{b}_j|
+ |\overline{a} - \overline{b}_j| \\
&\leq \gamma_0 +  |\overline{a} - \overline{b}_j|,\stepcounter{equation}\tag{\theequation}\label{bb4}
\end{align*}
where
\[
\overline{b}_j = \fint_{t_0 - (j+1) \widetilde{r} }^{t_0 - j \widetilde{r} }
\fint_{B_r (x_0)} a
\quad \text{for}\quad  j = 0, 1, 2, \ldots.
\]
Note that $\overline{b}_0 = \overline{a}$.
By the triangle inequality,
\[
|\overline{a} - \overline{b}_j |
\leq \sum_{i = 0}^{j-1} |\overline{b}_i - \overline{b}_{i+1} |.\stepcounter{equation}\tag{\theequation}\label{bb2}
\]
Also, we have
\begin{align*}
|\overline{b}_i - \overline{b}_{i+1} |
&=  \Bigg| \fint_{t_0 - (i+1) \widetilde{r} }^{t_0 - i \widetilde{r} }
\fint_{B_r (x_0)} a
- \fint_{t_0 - (i+2) \widetilde{r} }^{t_0 - (i+1) \widetilde{r} }
\fint_{B_r (x_0)} a \Bigg| \\
&\leq \Bigg| \fint_{t_0 - (i+1) \widetilde{r} }^{t_0 - i \widetilde{r} }
\fint_{B_r (x_0)} a
- \fint_{t_0 - (i+2) \widetilde{r} }^{t_0 - i \widetilde{r} }
\fint_{B_{2^{\alpha/2} r} (x_0)} a \Bigg| \\
&\quad +
\Bigg| \fint_{t_0 - (i+2) \widetilde{r} }^{t_0 - (i+1) \widetilde{r} }
\fint_{B_r (x_0)} a
- \fint_{t_0 - (i+2) \widetilde{r} }^{t_0 - i \widetilde{r} }
\fint_{B_{2^{\alpha/2} r} (x_0)} a \Bigg| \\
&\leq 2  \frac{2 \widetilde{r} \big( 2^{\alpha / 2}  r \big)^d }{ \widetilde{r}  r^d }  \gamma_0 = 4  2^{ d \alpha / 2} \gamma_0,\stepcounter{equation}\tag{\theequation}\label{bb3}
\end{align*}
where for the last inequality, we used that for $A \subset B$,
\[
\Bigg|  \fint_A f - \fint_B f \Bigg|
\leq \frac{|B|}{|A|} \fint_B | f - (f)_B|,\stepcounter{equation}\tag{\theequation}\label{simpl}
\]
and the oscillation of $a$ on $Q_{2^{\alpha/2} r}(t_0-i\widetilde{r},x_0)$ can be controlled by $\gamma_0$.
Therefore, by \eqref{bb2} and \eqref{bb3},
\[
|\overline{a} - \overline{b}_j |
 \leq
\sum_{i=0}^{j-1} |\overline{b}_i - \overline{b}_{i+1} |
\leq
4j  2^{ d \alpha / 2}   \gamma_0,
\]
which together with \eqref{bb4} implies that
\begin{align*}
&\fint_{(t_0 - h, t_0)} \fint_{B_r(x_0)}
|a - \overline{a}|
\leq \frac{1}{k-1} \sum_{j=0}^{k-1} [r_0 + |\overline{a} - \overline{b}_j |] \\
&\quad \leq \frac{1}{k-1} \sum_{j=0}^{k-1} [4 2^{d \alpha / 2} j + 1 ] \gamma_0
\leq N(\alpha, d)  k \gamma_0 \leq N  hr^{-2/\alpha}  \gamma_0.
\end{align*}
The lemma is proved.
\end{proof}

\section{Mean Oscillation Estimates}\label{3}

Throughout the section, we assume that $a^{ij}$'s are constants, $b^i\equiv c \equiv 0$, $T \in (0, \infty)$, and $p=q$, i.e., the space with unmixed norm.
The main goal of this section is to derive a mean oscillation estimate of $u \in \hh_{p_0, 0, \mathrm{loc}}^{\alpha, 2} ( \rr^d_T )$ satisfying
\[
\partial_t^\alpha u - a^{ij} D_{ij} u = f
\quad \text{in}\quad  \rr^d_T
\]
when $p_0 \in (1, 2)$.
For $t_0 \in (0, T]$, $r >0$, and $\Omega = (-r,r)^d \supset B_r$, we {are going to} decompose
$u = v + w$ in $(0, t_0) \times \Omega$,
where $w,v \in \hh_{p_0, 0}^{\alpha, 2} ( (0,t_0) \times \Omega ) $ satisfy
\begin{align*}
\partial_t^\alpha w - a^{ij} D_{ij} w = f \stepcounter{equation}\tag{\theequation}\label{99991}
\end{align*}
in $(0,t_0) \times \Omega$,
$w=0$ on $(0,t_0) \times \partial \Omega,$
and
\[
\partial_t^\alpha v - a^{ij} D_{ij} v = 0 \stepcounter{equation}\tag{\theequation}\label{9999}
\]
in $(0,t_0) \times  \Omega$.

In Subsections \ref{33.11} and \ref{33.22}, we estimate $v$ and $w$ separately. Then, in Subsection \ref{33.33}, we combine the results to get the estimate of $u$.

\subsection{Estimates of \texorpdfstring{$v$}{v}}\label{33.11}
We begin with a formula for computing the fractional derivative of the product with a cutoff function in time.
\begin{lemma}\label{3.1}
Let $p \in [1, \infty)$, $\alpha \in (1,2)$, $k \in \{ 1,2, \ldots \}$,
$-\infty < S < t_0 < T < \infty$, and $v \in \hh_{p,0}^{\alpha, k} ( (S,T) \times \Omega )$. Then, for any infinitely differentiable function $\eta$ defined on $\rr$ such that $\eta(t) = 0$ for $t \leq t_0$,
we have $\eta v \in \hh_{p,0}^{\alpha,2} ( (t_0, T) \times \Omega )$ and
\[
\partial_t^\alpha (\eta v)(t,x)
= \partial^2_t I_{t_0}^{2 - \alpha} (\eta v) (t,x)
= \eta (t) \partial^2_t I_{S}^{2 - \alpha } v (t,x) + g(t,x)
\]
for $(t,x) \in (t_0, T) \times \Omega$, where
\begin{align*}
g(t,x) &=  \frac{\alpha (\alpha - 1)}{\Gamma (2 - \alpha)}  \int_S^t (t-s)^{-\alpha - 1}
[ \eta(s) - \eta(t) - \eta'(t) (s-t) ] v (s,x) \, ds  \\
&\quad+ \frac{\alpha}{\Gamma (2 - \alpha)}  \eta'(t)  \partial_t
\int_S^t (t-s)^{1- \alpha} v (s,x) \, ds. \stepcounter{equation}\tag{\theequation}\label{914}
\end{align*}
\end{lemma}
\begin{proof}
The fact that $\eta v \in \hh_{p,0}^{\alpha,2} ( (t_0, T) \times \Omega )$ follows from the density of smooth functions in the space. The proof {for} the case of $\alpha \in (0,1)$ is given in \cite[Lemma 3.6]{dong19}, and it works here with minor modification. We skip the details. It remains to show \eqref{914}. Without loss of generality, we assume $t_0 = 0$. Also, recall that $\partial_t^\alpha = \partial_t^2 I^{2 - \alpha}_S$. Therefore,
\begin{align*}
&g(t,x)   \Gamma (2 - \alpha) \\
&= \partial_t^2
\Big[
\int_0^t (t-s)^{1 - \alpha}
\eta(s) v (s,x) \, ds
\Big]
- \partial_t^2
\Big[
\eta(t)
\int_S^t (t-s)^{1 - \alpha} v (s,x) \, ds
\Big] \\
&\quad+  2 \eta'(t)  \partial_t
\int_S^t (t-s)^{1 - \alpha} v (s,x) \, ds
+ \eta''(t)  \int_S^t (t-s)^{1 - \alpha} v (s,x) \, ds. \stepcounter{equation}\tag{\theequation}\label{g}
\end{align*}
The first two terms of the right-hand side of (\ref{g}) equal
\begin{align*}
&\partial_t^2
\Big[
\int_S^t (t-s)^{1- \alpha}
[\eta(s) - \eta(t)] v (s,x) \, ds
\Big] \\
&= \partial_t
\Big[
\int_S^t (1-\alpha)
(t-s)^{- \alpha}
[\eta(s) - \eta(t)] v (s,x) \, ds
- \int_S^t (t-s)^{1 - \alpha}
\eta'(t) v (s,x) \, ds
\Big] \\
&= \partial_t
\Big[
\int_S^t (1-\alpha)
(t-s)^{- \alpha}
[\eta(s) - \eta(t) -\eta'(t) (s-t)] v (s,x) \, ds  \\
&\quad+ \int_S^t (\alpha - 2)
(t-s)^{1- \alpha}
\eta'(t) v (s,x) \, ds \Big] \\
&= \int_S^t (-\alpha) (1- \alpha) (t-s)^{- \alpha - 1}
[\eta(s) - \eta(t) -\eta'(t) (s-t)] v (s,x) \, ds  \\
&\quad+ \int_S^t (1 - \alpha ) (t-s)^{1- \alpha} \eta''(t) v (s,x) \, ds +  (\alpha - 2) \eta''(t)
\int_S^t (t-s)^{1- \alpha} v (s,x) \, ds \\
&\quad+   (\alpha - 2) \eta'(t)  \partial_t \Big[
\int_S^t (t-s)^{1- \alpha} v (s,x) \, ds \Big].
\end{align*}
Thus, by (\ref{g}),
\begin{align*}
g(t,x)\Gamma (2 - \alpha)
&= \alpha (\alpha - 1) \int_S^t (t-s)^{-\alpha - 1}
[ \eta(s) - \eta(t) - \eta'(t) (s-t) ] v (s,x) \, ds  \\
&\quad+ \alpha  \eta'(t)  \partial_t
\int_S^t (t-s)^{1- \alpha} v (s,x) \, ds.
\end{align*}
The lemma is proved.
\end{proof}
\begin{remark}\label{r3.2}
If there exists a function $\xi \in C^\infty(\rr)$ such that $\xi(t) = 0$ for $t \leq t_0'$ for some $ t_0' \in (S, t_0)$, and  $\xi \equiv 1 $ in $ \mathrm{supp } \, \eta $, then we can rewrite the second term of $g$ in (\ref{914}) as follows. With $\beta = \alpha -1$,
\begin{align*}
&\eta'(t) \partial_t
\int_{S}^t (t-s)^{1-\alpha } v(s,x) \, ds
= \frac{\eta'(t)}{ \xi(t)  } \Bigg( \Gamma (1 - \beta) \xi(t,x) \partial_t^\beta  v(t,x)  \Bigg)\\
&=\frac{\eta'(t)}{ \xi(t)  }  \Bigg(  \Gamma (1 - \beta)
\Big[
\xi(t,x) \partial_t^\beta v(t,x) - \partial_t^\beta (\xi v)(t,x) + \partial_t^\beta (\xi v)(t,x)
\Big]  \Bigg)\\
&= \frac{\eta'(t)}{ \xi(t)  } \Bigg(
 \frac{\Gamma (1 - \beta) \beta}{\Gamma (1- \beta)}
\int_{S}^t (t-s)^{-\beta - 1} [\xi(s) - \xi(t)] v(s,x) \, ds +  \Gamma (1 - \beta)  \partial_t^\beta (\xi v)(t,x) \Bigg)\\
&= \frac{\eta'(t)}{ \xi(t)  }  \Bigg(\Big[  (\alpha - 1) \int_{S}^t (t-s)^{- \alpha}
[\xi(s) - \xi(t)] v(s,x) \, ds \Big] +  \Gamma (2 - \alpha) \partial_t^\beta (\xi v)(t,x)  \Bigg),
\end{align*}
where we used \cite[Lemma 3.6]{dong19} in the second last equality.
\end{remark}

The following interpolation inequality {for} the time derivatives will be used for the estimate of $g$ in Lemma \ref{3.1}.
\begin{lemma}\label{3.2}
For any $\ep > 0 $, $\beta = \alpha - 1$, $p_0 \in (1,\infty)$, any open set $G \subset \rr^d $,
and $v\in  \hh_{p_0, 0}^{\alpha, 2} ( (0,T) \times G ) $,
we have
\[
\norm{\partial_t^\beta v}_{L_{p_0} (  (0,T) \times G ) }
\leq \ep \norm{\partial_t^\alpha v}_{L_{p_0}      (  (0,T) \times G )        }
+ N(\alpha,p_0) T^{\alpha-2}  \ep^{-1} \norm{v}_{L_{p_0}(  (0,T) \times G ).}
\]
\end{lemma}
\begin{proof}
Using \ref{A2} and \cite[Lemma 5.5]{dong20}, we conclude
\begin{align*}
\norm{\partial_t^\beta v}_{L_{p_0}(  (0,T) \times G )}
&\leq \ep \norm{\partial_t^\alpha v}_{L_{p_0}(  (0,T) \times G )}
+ N \ep^{-1} \norm{I^{2-\alpha} v}_{L_{p_0}(  (0,T) \times G )} \\
&\leq \ep \norm{\partial_t^\alpha v}_{L_{p_0}(  (0,T) \times G )}
+ N(\alpha,p_0) T^{2-\alpha} \ep^{-1} \norm{v}_{L_{p_0}(  (0,T) \times G )}.
\end{align*}
Note that in \cite[Lemma 5.5]{dong20}, the author assumes $\alpha\in (0,1)$, but the proof works for all $\alpha\in (0,\infty)$.
\end{proof}

Next, we estimate $v$ by using cutoff functions together with the Sobolev embedding.
\begin{proposition}\label{3.3}
Let $p_0 \in (1, \infty)$, $t_0 \in (0,\infty)$, $r>0$, and $v \in \hh_{p_0,0}^{\alpha, 2} ( (0, t_0) \times B_r )$ satisfy (\ref{9999}) in $(0, t_0) \times B_r$. Then, for any
$p_1 \in (p_0, \infty)$
satisfying
\[
{1/p_1 \le 1/p_0 - 2/(d+2)},
\]
{we have $v \in \hh_{p_1, 0}^{\alpha, 2} ( (0,t_0) \times B_{r/2} )$ and}
\[
( | D^2 v|^{p_1} )^{1 /p_1}_{Q_{r/2} (t_1, 0) }
\leq N \sum_{j=1}^\infty j^{-(1 + \alpha )}
( | D^2 v|^{p_0} )^{1 /p_0}_{Q_{r} (t_1 - (j-1)r^{2/\alpha}, 0) }\stepcounter{equation}\tag{\theequation}\label{1100}
\]
for any $t_1 \leq t_0$, where $N = N(d, \delta, \alpha, p_1,p_0)$.
If $p_0 > d/2 + 1$, then
\[
[D^2 v]_{C^{\sigma \alpha/2, \sigma} ( Q_{r/2} (t_1, 0) ) }
\leq Nr^{-\sigma}
\sum_{j=1}^\infty j^{-(1 + \alpha )}
( | D^2 v|^{p_0} )^{1 /p_0}_{Q_{r} (t_1 - (j-1)r^{2/\alpha}, 0) } \stepcounter{equation}\tag{\theequation}\label{1101}
\]
for any $t_1 \leq t_0$, where $\sigma = \sigma (d, \alpha, p_0) \in (0,1)$ and $N = N(d, \delta, \alpha, p_0)$.
Moreover, for any $p_1\in (p_0,\infty)$,
we have
\[
( | D^2 v|^{p_1} )^{1 /p_1}_{Q_{r/2} (t_1, 0) }
\leq N \sum_{j=1}^\infty j^{-(1 + \alpha )}
( | D^2 v|^{p_0} )^{1 /p_0}_{Q_{r} (t_1 - (j-1)r^{2/\alpha}, 0) } \stepcounter{equation}\tag{\theequation}\label{homo2}
\]
for any $t_1 \leq t_0$, where $N = N(d, \delta, \alpha, p_1, p_0)$. Furthermore,
\[
[D^2 v]_{C^{\sigma \alpha/2, \sigma} ( Q_{r/2} (t_1, 0) ) }
\leq Nr^{-\sigma}
\sum_{j=1}^\infty j^{-(1 + \alpha )}
( | D^2 v|^{p_0} )^{1 /p_0}_{Q_{r} (t_1 - (j-1)r^{2/\alpha}, 0) } \stepcounter{equation}\tag{\theequation}\label{homo}
\]
for any $t_1 \leq t_0$, where $\sigma = \sigma (d, \alpha, p_0) \in (0,1)$ and $N = N(d, \delta, \alpha, p_0)$.
\end{proposition}

\begin{proof}
{\em Step 0.}
We first note that \eqref{homo2} and \eqref{homo} follow from \eqref{1100} and \eqref{1101} by a finite iteration argument. See \cite[Proposition 4.3]{dong20} for details.

{\em Step 1.}
By scaling, we assume $r=1$. For $k = 0,1, \ldots,$
let $\delta = \delta_0 = (1/2)^{2/\alpha}$ and
\[
\delta_k = \delta + (1-\delta) \sum_{j = 1}^k 2^{-j},
\]
and take $\eta_k \in C^{\infty}(\rr)$ such that
\[
\eta_k(t) = \begin{cases}
1 \quad  t \in (t_1 - \delta_k, t_1),
\\
0 \quad  t \in (t_1 - 1, t_1 - \delta_{k+1} ),
\end{cases}
\]
\begin{align*}
     |\eta'_k| \leq (1 - \delta)^{-1}  2^{k+2},\quad  |\eta''_k| \leq (1 - \delta)^{-2} 2^{2k + 4}. \stepcounter{equation}\tag{\theequation}\label{010406}
\end{align*}
Note that $\eta_{k+1} \equiv 1$ in $ \mathrm{supp } \, \eta_k$, and $v$ can be extended as zero for $t\leq 0 $. Therefore, by Lemma \ref{3.1},
for any $k$, we have
\[
\eta_k  v \in \hh_{p_0, 0}^{\alpha, 2} ( (t_1-1,t_1) \times B_{3/4} ).
\]
Furthermore, let
\begin{align*}
    &g_k = \partial_t^\alpha (\eta_k v) - \eta_k \partial_t^\alpha v  \quad \text{in}\quad  (t_1-1,t_1) \times B_{3/4}, \\
    &r_k = \frac{3}{4} - \frac{1}{4}  \frac{1}{2^k}, \quad \text{and}\quad  \Omega_k :=  (t_1-1, t_1) \times B_{r_k} \stepcounter{equation}\tag{\theequation}\label{010409}
\end{align*}
with $r_0 = 1/2$ {and} $r_\infty = 3/4$. Then,
\begin{align*}
\norm{\eta_k v}_{\hh_{p_0}^{\alpha, 2} ( \Omega_k) }
&\leq N  \norm{\eta_k v}_{L_{p_0} (\Omega_k)}
+ \norm{ D^2 (\eta_k v) }_{L_{p_0} (\Omega_k) }
+ \norm{\partial_t^\alpha (\eta_k v) }_{L_{p_0} (\Omega_k) }  \\
&\leq N  \norm{\eta_k v}_{L_{p_0} (\Omega_k)}
+ N \norm{g_k }_{L_{p_0} (\Omega_{k+1}) }
+  \frac{N}{(r_{k+1} - r_k)^2  }   \norm{ \eta_k v }_{L_{p_0} (\Omega_{k+1}) }   \\
&\leq N 2^{k+2} \norm{ v}_{L_{p_0} (\Omega_\infty)}
+ N \norm{g_k}_{L_{p_0} (\Omega_{k+1})},
\stepcounter{equation}\tag{\theequation}\label{*}
\end{align*}
where $N = N(d, \delta, \alpha, p_0)$, and the second inequality is by a local estimate which can be proved using the estimate in whole space in \cite{P1} together with a localization argument as in \cite[Lemma 4.2]{dong19}.

To estimate $g_k$,
we take $\zeta \in C^{\infty}(\rr)$ such that
\[
|\zeta|  \leq 1 \quad \text{and}\quad  \zeta(t) = \begin{cases}
1 \quad \quad \text{when}\quad  t \in (t_1 - 2, t_1), \\
0 \quad \quad \text{when}\quad  t \in (-\infty, t_1 - 3 ).
\end{cases} \stepcounter{equation}\tag{\theequation}\label{rr}
\]
Furthermore, let $v_1 = \zeta v$ and $v_2 = v-v_1$. Lemma \ref{3.1} with $S = -\infty$ yields
\begin{align*}
|g_k | &\leq N(\alpha) \big[  |g_{k,1}| + |g_{k,2}| \big],
\end{align*}
where
\begin{align*}
g_{k,1}(t,x)
&= \alpha (\alpha - 1) \int_{t_1-3}^t (t-s)^{-\alpha - 1}
[ \eta_k(s) - \eta_k(t) - \eta_k'(t) (s-t) ] v_1(s,x) \, ds  \\
&\quad+ \alpha  \eta_k'(t)  \partial_t
\int_{t_1-3}^t (t-s)^{1- \alpha} v_1(s,x) \, ds,\stepcounter{equation}\tag{\theequation}\label{rr1}
\end{align*}
and
\begin{align*}
 g_{k,2}(t,x)
&= \alpha (\alpha - 1) \int_{-\infty}^{t_1-2} (t-s)^{-\alpha - 1}
[ \eta_k(s) - \eta_k(t) - \eta_k'(t) (s-t) ] v_2(s,x) \, ds  \\
&\quad+ \alpha  \eta_k'(t)  \partial_t
\int_{-\infty}^{t_1-2} (t-s)^{1- \alpha} v_2(s,x) \, ds \\
&= \alpha (\alpha - 1)  \int_{-\infty}^{t_1-2} (t-s)^{-\alpha - 1} [ \eta_k (s) - \eta_k (t) ] v_2(s,x) \, ds.
\stepcounter{equation}\tag{\theequation}\label{rr2}\end{align*}

{\em Step 2.}
We further decompose $g_{k,1}$ as
$$g_{k,  1} = \alpha (\alpha - 1) g_{k,1,1} + \alpha g_{k, 1, 2},$$
where
\begin{align*}
 g_{k,1,1}(t,x)
= \int_{t_1-3}^{t} (t-s)^{-\alpha - 1}
[ \eta_k(s) - \eta_k(t)   - \eta_k' (t)  (s-t)  ]v_1(s,x) \, ds.
\end{align*}
For $t\in (t_1-1, t_1)$,
\begin{align*}
|g_{k,1,1}(t,x)|
&\leq
\norm{\eta_k''}_{L_\infty}
\cdot \int_{t_1 - 3}^t | (t-s)^{-\alpha +1} v(s,x) | \, ds \\
&\leq \norm{\eta_k''}_{L_\infty}
\int_{t - 3}^t |(t-s)^{-\alpha + 1} v(s,x) | \, ds \\
&= \norm{\eta_k''}_{L_\infty}
\int_{0}^3 \tau^{-\alpha + 1} |v(t-\tau,x) | \, d\tau,
\end{align*}
and by the Minkowski inequality,
\begin{align*}
\norm{g_{k,1,1}}_{L_{p_0} (\Omega_{k+1} )  }
&\leq N \norm{  \eta_k'' }_{L_\infty} \norm{v}_{L_{p_0} ( (t_1 - 4, t_1) \times B_{3/4} ) }. \stepcounter{equation}\tag{\theequation}\label{k11}
\end{align*}
For $\beta := \alpha - 1\in (0,1)$, using the facts that $v_1 \in \hh_{p_0, 0}^{\alpha, 2} ( (t_1-3,t_1) \times B_{3/4} $ and
$\eta_{k+1} \equiv 1 $ in $\mathrm{supp }\, \eta_k$ together with Remark \ref{r3.2}, we have
\begin{align*}
g_{k,1,2}(t,x)
&= \frac{\eta_k'(t)}{ \eta_{k+1}(t)  } \Big[  (\alpha - 1) \int_{t_1 - 3}^t (t-s)^{- \alpha}
[\eta_{k+1}(s) - \eta_{k+1}(t)] v_1(s,x) \, ds  \\
& \quad \quad+  \Gamma (2 - \alpha) \partial_t^\beta (\eta_{k+1} v_1)  \Big],
\end{align*}
which implies that
\begin{align*}
|g_{k,1,2}(t,x) |
&\leq N(\alpha)\norm{\eta_k'}_{L_\infty}
\int_{t_1 - 3}^t
(t-s)^{-\alpha}
|\eta_{k+1} (s) - \eta_{k+1} (t) | |v(s,x)| \, ds \\
&\quad+ N(\alpha)\norm{\eta_k'}_{L_\infty} |\partial_t^\beta (\eta_{k+1} v_1)|\\
&\leq N(\alpha)\norm{\eta_k'}_{L_\infty}\norm{\eta_{k+1}'}_{L_\infty}
\int_{t_1 - 3}^t
(t-s)^{1-\alpha}|v(s,x)| \, ds \\
&\quad+ N(\alpha)\norm{\eta_k'}_{L_\infty}|\partial_t^\beta (\eta_{k+1} v_1)|.
\end{align*}
By applying Lemma \ref{3.2} to $\eta_{k+1} v_1 = \eta_{k+1} v \in \hh_{p_0, 0}^{\alpha, 2}(\Omega_{k+1})$ and Young's inequality,  for any $\ep > 0$, we have
\begin{align*}
&\norm{g_{k,1,2}}_{L_{p_0} (\Omega_{k+1} )}
\leq
N\norm{\eta_k'}_{L_\infty}
 \norm{\eta_{k+1}'}_{L_\infty}
\norm{v}_{L_{p_0} ( t_1 - 4, t_1 ) \times B_{3/4} (x_0) )  } \\
&\quad+
\ep \norm{\partial_t^\alpha  (\eta_{k+1} v) }_{L_{p_0} (\Omega_{k+1})}
+ N \ep^{-1} \norm{\eta_k'}^2_{L_\infty}
\norm{\eta_{k+1} v}_{L_p (\Omega_{k+1})}, \stepcounter{equation}\tag{\theequation}\label{k12}
\end{align*}
where $N = N(\alpha,d,p_0)$.

Recall the definition of $g_{k,2}$ in \eqref{rr2}. For $j \ge 2$, $t \in (t_1 - 1, t_1)$, and $s\in (t_1 - j - 1,t_1 - j)$, it is easily seen that $t-s > (j-1) >  (j + 1)/4$.
Thus,
\begin{align*}
|g_{k,2}(t,x)|
&\leq \int_{-\infty}^{t_1 - 2}
(t-s)^{-\alpha - 1} | v (s,x)| \, ds \leq \sum_{j=2}^\infty \int_{t_1 - j - 1}^{t_1 - j}
(t-s)^{-\alpha - 1} |v (s,x)| \, ds \\
&\leq N(\alpha)\sum_{j=2}^\infty \int_{t_1 - j - 1}^{t_1 - j}
(j+1)^{-\alpha - 1} |v (s,x)| \, ds.
\end{align*}
Therefore, by H\"older's inequality,
\begin{align*}
\norm{ g_{k,2} }_{L_{p_0} (\Omega_{k+1})  }
&\leq  N \sum_{j=2}^\infty (j+1)^{-(\alpha + 1)}
\norm{v}_{ L_{p_0} ((t_1 - j - 1, t_1 - j ) \times B_{3/4}) } \\
&\leq  N \sum_{j=2}^\infty (j+1)^{-(\alpha + 1)}
\norm{v}_{ L_{p_0} ((t_1 - j - 1, t_1 - j ) \times B_{1}) }, \stepcounter{equation}\tag{\theequation}\label{k2}
\end{align*}
where $N = N(\alpha,d,p_0)  $.
Combining  (\ref{k11}), (\ref{k12}), and (\ref{k2}), we get
\begin{align*}
&\norm{g_k}_{L_{p_0} (\Omega_{k+1}) } \leq N(\alpha)\norm{g_{k,2}}_{L_{p_0} (\Omega_{k+1}) }
+ N(\alpha)\norm{g_{k,1}}_{L_{p_0} (\Omega_{k+1}) } \\
&\leq N \sum_{j=2}^\infty
(j+1)^{-\alpha- 1}
\norm{v}_{L_{p_0} ((t_1 - j - 1, t_1 -j )\times B_1)} \\
&\quad+
N\norm{\eta_k''}_{L_\infty} \sum_{j=0}^{3} \norm{v}_{L_{p_0} ((t_1 - j - 1, t_1 -j ) \times B_1)} \\
&\quad+
N \norm{\eta_k'}_{L_\infty} \norm{\eta_{k+1}'}_{L_\infty}
\sum_{j=0}^{3} \norm{v}_{L_{p_0} ((t_1 - j - 1, t_1 -j )\times B_1)} \\
&\quad+
\ep \norm{\partial_t^\alpha (\eta_{k+1} v)}_{L_{p_0} (\Omega_{k+1})}
+ N \ep^{-1} \norm{\eta_k'}_{L_\infty}^2 \norm{\eta_{k+1} v}_{L_{p_0} (\Omega_{k+1})},\stepcounter{equation}\tag{\theequation}\label{yy1}
\end{align*}
where $N = N(\alpha,d,p_0)  $.

{\em Step 3.}
Multiplying both sides of (\ref{*}) by $\ep^k$, summing over $ k = 0, 1, \ldots$, and using \eqref{yy1}, we get
\begin{align*}
& \sum_{k=0}^\infty \varepsilon^k\norm{\eta_k v}_{\hh_{p_0}^{\alpha, 2} ( \Omega_k) } \leq N \norm{v}_{L_{p_0} (\Omega_\infty)} \sum_{k=0}^\infty 2^k\ep^{k} + \sum_{k=0}^\infty \ep^{k} \norm{g_k}_{L_{p_0} (\Omega_{k+1})}\\
&\leq N \norm{v}_{L_{p_0} (\Omega_\infty)} \sum_{k=0}^\infty 2^k\ep^{k}
+ N \sum_{j=2}^\infty
(j+1)^{-\alpha- 1}
\norm{v}_{L_{p_0} (t_1 - j - 1, t_1 -j )\times B_1} \sum_{k=0}^\infty \ep^{k}\\
&\quad+
N \sum_{j=0}^{3} \norm{v}_{L_{p_0} (t_1 - j - 1, t_1 -j ) \times B_1} \sum_{k=0}^\infty \ep^{k}2^{2k}\\
&\quad+
\sum_{k=0}^\infty \ep^{k+1} \norm{\partial_t^\alpha (\eta_{k+1} v)}_{L_{p_0} (\Omega_{k+1})}
+ N \norm{ v}_{L_{p_0} (\Omega_\infty)} \ep^{-1} \sum_{k=0}^\infty \ep^{k}2^{2k}, \stepcounter{equation}\tag{\theequation}\label{k88}
\end{align*}
where $N = N(d, \delta, \alpha, p_0)$. Picking $\ep < 1/8$ such that
$\sum_{k=0}^\infty \ep^k 2^{2k} < \infty$, and absorbing the second last summation to the left{-hand side} of \eqref{k88}, we derive
\begin{align*}
\norm{\eta_0 v}_{\hh_{p_0}^{\alpha, 2} ( \Omega_0 ) }
\leq N\sum_{j=1}^\infty j^{-\alpha -1}
\norm{v}_{ L_{p_0} ( (t_1 - j, t_1 - j +1 ) \times B_1 )}.
\end{align*}
Therefore, by using the Sobolev embedding
$$
 \hh_{p_0, 0}^{\alpha, 2} ( (t_1-1, t_1) \times B_{1/2} ) \hookrightarrow \mathbf{W}_{p_0, 0}^{1, 2}( (t_1-1, t_1) \times B_{1/2} ) \hookrightarrow L_{p_1}( (t_1-1, t_1) \times B_{1/2} ),
$$
we conclude
\begin{align*}
& \norm{v}_{ L_{p_1} ( Q_{1/2} (t_1, 0) )   }
= \norm{\eta_0 v}_{  L_{p_1} (Q_{1/2}( (t_1, 0) ) } \leq \norm{\eta_0 v}_{  L_{p_1}( (t_1 - 1, t_1) \times B_{1/2} ) } \\
&\leq N \norm{\eta_0 v}_{\hh_{p_0}^{\alpha, 2} ( (t_1-1, t_1) \times B_{1/2}  ) } \leq N \sum_{j=1}^\infty j^{-\alpha -1}
\norm{v}_{ L_{p_0} ( (t_1 - j, t_1 - j +1 ) \times B_1 )},  \stepcounter{equation}\tag{\theequation}\label{homo estimate}
\end{align*}
where $N = N(d,\delta,\alpha,p_0,p_1)$.

For the proof of \eqref{1101}, we pick $\widetilde{p}_0\le p_0$ such that $\widetilde{p}_0 \in (d/2 + 1,d + 2)$, and let $\sigma = 2-(d+2)/\widetilde{p}_0 \in (0,1)$. By using the Sobolev embedding
$$ \hh_{\widetilde{p}_0 , 0}^{\alpha, 2} ( (t_1-1, t_1) \times B_{1/2} ) \hookrightarrow \mathbf{W}_{\widetilde{p}_0 , 0}^{1, 2}( (t_1-1, t_1) \times B_{1/2} ) \hookrightarrow C^{\sigma/2,\sigma}( (t_1-1, t_1) \times B_{1/2} ),$$
we replace $\norm{v}_{ L_{p_1} ( Q_{1/2} (t_1, 0) )}$ with $[ v]_{C^{\sigma \alpha/2, \sigma} ( Q_{r/2} (t_1, 0) ) }$ in \eqref{homo estimate} to conclude
\[
[v]_{C^{\sigma \alpha/2, \sigma} ( Q_{1/2} (t_1, 0) ) }
\leq N \sum_{j=1}^\infty j^{-\alpha -1}
\norm{v}_{ L_{p_0} ( (t_1 - j, t_1 - j +1 ) \times B_1 )}, \stepcounter{equation}\tag{\theequation}\label{1102}
\]
where $N= N(d,\delta,\alpha,p_0)$.

{\em Step 4.}
It remains to prove $
D^2v\in \hh_{p_0, 0}^{\alpha, 2} ( (0,t_{1}) \times B_{3/4} )$.
If so, \eqref{1100} and \eqref{1101} follow from replacing $v$ with $D^2v$ into (\ref{homo estimate}) and \eqref{1102}.
Note that by \cite[Lemma 4.3]{dong19}, $(D  v)^\delta = D(v^\delta) \in
\hh_{p_0, 0}^{\alpha, 2} ( (0,t_1) \times B_{5/6} )
$
for $\delta$ sufficiently small, where $(D v)^\delta$ and $v^\delta$ denote the mollification in the spatial variables. It follows that
$$
(D  v)^\delta
\to D v
\quad \text{in}\quad
L_{p_0} ( (0,t_1) \times B_{5/6} ).
$$
By repeating the argument in Step 3, {we have}
\begin{align*}
\norm{ (Dv)^\delta}_{\hh_{p_0,0}^{\alpha,2} ( (0,t_1) \times B_{5/6} )}
&\leq  N \sum_{j=1}^\infty j^{-\alpha -1}
\norm{(Dv)^\delta}_{ L_{p_0} ( (t_1 - j, t_1 - j +1 ) \times B_{7/8} )} \\
&\leq N \sum_{j=1}^\infty j^{-\alpha -1}
\norm{(Dv)^\delta}_{ L_{p_0} ( (0,t_1) \times B_{7/8} )},
\end{align*}
where $N = N(d,\delta,\alpha,p_0)$. Therefore,
$ (D v)^\delta$ is Cauchy in $\hh_{p_0,0}^{\alpha,2} ((0,t_1)\times B_{5/6} )$, which implies
\[
D v \in \hh_{p_0, 0}^{\alpha, 2} ( (0,t_1) \times B_{5/6} ).
\]
Replacing $Dv$ with $D^2v$ and repeating the above process, we conclude
\[
D^2  v \in \hh_{p_0, 0}^{\alpha, 2} ( (0,t_1) \times B_{3/4} ).
\]
Finally, the statement $v \in \hh_{p_1, 0}^{\alpha, 2} ( (0,t_1) \times B_{1/2} )$ follows from a mollification argument. See \cite[Lemma 4.1]{dong20} for details. The proposition is proved.
\end{proof}

\begin{remark}\label{go}
By slightly modifying the proof of \eqref{homo estimate}, we derive an estimate that will be used later in Proposition \ref{3.8}. For $w \in \hh_{p_0, 0}^{\alpha, 2} ( (0, t_0) \times B_1 )$ {satisfying} \eqref{99991}, if we take cutoff functions $\eta_k$ defined in \eqref{010406} and $g_k{,\Omega_k}$ defined in \eqref{010409} with $w$ in place of $v$, then $\partial_t^{\alpha} (w \eta_k ) - \Delta (w \eta_k ) = f\eta_k {+} g_k $. Similar to \eqref{*}, we have
\begin{align*}
\norm{\eta_k w}_{\hh_{p_0}^{\alpha, 2} ( \Omega_k) }
\leq N 2^{k+2} \norm{ w}_{L_{p_0} (\Omega_\infty)} + N \norm{ f}_{L_{p_0} (\Omega_\infty)}
+ N \norm{g_k}_{L_{p_0} (\Omega_{k+1})}.
\stepcounter{equation}\tag{\theequation}\label{**}
\end{align*}
Furthermore, by redefining $\zeta$ in \eqref{rr} by \[
\zeta(t) = \begin{cases}
1 \quad \quad \text{when}\quad  t \in (t_1 - 3, t_1),  \\
0 \quad \quad \text{when}\quad  t \in (-\infty, t_1 - 4 ),
\end{cases}
\]
the lower bounds of the integrals in \eqref{rr1} become $t_1 - 4$, and the upper bound of the integral in \eqref{rr2} becomes $t_1 - 3= s_2$, where $s_j = t_1 - 2^{j} + 1$ for $j=0,1,\ldots$. In this case, we estimate $g_{k,1}$ similarly as in Step 2 above. However, we estimate $g_{k,2}$ differently:
\begin{align*}
|g_{k,2}(t,x)|
&\leq \int_{-\infty}^{t_1 - 3}
(t-s)^{-\alpha - 1} | v (s,x)| \, ds \leq \sum_{j=2}^\infty \int^{s_j}_{s_{j+1}}
(t-s)^{-\alpha - 1} |v (s,x)| \, ds \\
&\leq N(\alpha)\sum_{j=2}^\infty \int^{s_j}_{s_{j+1}}
2^{j(-\alpha - 1)} |v (s,x)| \, ds,
\end{align*}
where the last inequality is because $t\in (t_1-1,t_1)$. Therefore, by repeating the argument in Step 3 above with \eqref{**}, and applying the Minkowski inequality and H\"older’s inequality, we conclude
\begin{align*}
(  |D^2 w|^{p_0} )^{1 / p_0}_{Q_{1/2} (t_0, 0)}
&\leq N(|f|^{p_0})^{1/p_0}_{Q_1 (t_0, 0 )}
+ N \sum_{k=0}^\infty 2^{-k \alpha}
(|w|^{p_0})^{1/p_0}_{ ( s_{k+1},s_k ) \times B_1},
\end{align*}
where $N = N(d, \delta, \alpha, p_0)$.
\end{remark}

\subsection{Estimates of \texorpdfstring{$w$}{w}}\label{33.22}
In this subsection, we prove the existence and the estimates of $w$ satisfying \eqref{99991} when $a^{ij} = \delta^{ij}$.
\begin{lemma}\label{3.6}
For $\Omega=(0,1)^d$ and $f \in L_2 ( \Omega_T )$, there exists a unique $w \in \hh_{2,0}^{\alpha, 2} (\Omega_T )$ to
\[
\partial_t^\alpha w - 4\Delta w = f \quad \text{in}\quad
\Omega_T,\quad
w=0
\quad \text{on}\quad
(0,T) \times \partial \Omega,
\]
and $w$ satisfies
$$
\norm{w}_{\hh_{2}^{\alpha, 2} (\Omega_T )} \le N \norm{f}_{L_2(\Omega_T )}.
$$
\end{lemma}
\begin{proof}
We use the method of separation of variables to construct $w$ explicitly. Since $f\in L_2 ((0,t_0) \times \Omega )$, we write
\[
f (t,x^1,\ldots,x^d) = \sum^\infty_{n_1, n_2, \ldots, n_d =1}
f_{n_1, \ldots, n_d } (t)\sin (n_1 \pi x^1) \cdots \sin (n_d \pi x^d).
\]
By orthogonality,
\[
f_{n_1, \ldots, n_d } (t)  = 2^d \int_0^1 \cdots \int_0^1 \sin (n_1 \pi x^1) \cdots \sin (n_d \pi x^d) f(t, x ) \, dx.
\]
We take
\[
w (t,x^1,\ldots,x^d) = \sum^\infty_{n_1, n_2, \ldots, n_d =1}
w_{n_1, \ldots, n_d } (t) \sin (n_1 \pi x^1) \cdots \sin (n_d \pi {x^d}) \stepcounter{equation} \tag{\theequation}\label{10291}.
\]
Then,
\begin{align*}
\partial_t^\alpha w - 4\Delta w &= \sum^\infty_{n_1,\ldots, n_d =1}  \big[  \partial_t^\alpha {w}_{n_1, \ldots, n_d }(t) + 4(n_1\ldots n_d)^2 \pi^{2d} {w}_{n_1, \ldots, n_d }(t) \big]\\
& \quad \quad\quad\quad \quad \cdot \sin (n_1 \pi x^1) \cdots \sin (n_d \pi x^d) \\
&= \sum^\infty_{n_1, n_2, \ldots, n_d =1}f_{n_1, \ldots, n_d } (t)  \sin (n_1 \pi x^1) \cdots \sin (n_d \pi {x^d}).
\end{align*}
To determine $w_{n_1, \ldots, n_d }$, it is sufficient to solve the following ODE: for $\lambda > 0 $,
\begin{align*}
\partial_t^\alpha \phi(t) + \lambda \phi(t) = f(t),
\quad
\phi(0) = 0,
\quad
\phi'(0) = 0.
\end{align*}
Let
\[
\phi(t) = \int_0^t H_{\alpha, \lambda}(t-s) f(s) \, ds,
\]
where
$$
H_{\alpha, \lambda}(t) = I^{\alpha - 1} E_\alpha (- \lambda t^\alpha ), \quad \text{and}\quad
 E_\alpha(z) = \sum^\infty_{n=0} \frac{z^n}{\Gamma(\alpha n +1)}
 $$
is the Mittag-Leffler function. We check that this $\phi$ solves the differential equation. It suffices to check that
\[
L \phi
= L(H_{\alpha, \lambda}) \cdot L(f)
=  \frac{Lf(s)}{s^\alpha + \lambda}, \stepcounter{equation} \tag{\theequation}\label{opop}
\]
where  $L$ denotes the Laplace transformation, and in the second equality of \eqref{opop} we used the following proved in \cite[Lemma 3.4]{P1}
$$
L(H_{\alpha, \lambda}) (s) = \frac{1}{s^\alpha + \lambda}.
$$
Recall that
$\partial_t^\alpha \phi = \partial_t I^{2 - \alpha} (\phi ' ).$
Therefore,
$L[\partial_t^\alpha \phi] (s)
= s^\alpha L \phi, $
which implies
$$
s^\alpha L \phi (s) + \lambda L \phi (s) = L f(s),
$$
and we arrive at \eqref{opop}.

By \cite{ML}, $E_\alpha(z) \in (0,1)$ for $z\in(-\infty,0)$, and thus,
\begin{align*}
&\norm{H_{\alpha, \lambda}  \cdot  \chi_{(0,T) } }_{L_1 (\rr)} = I^{\alpha} E_\alpha (- \lambda \cdot^\alpha )(T) = \frac{1}{\Gamma(\alpha)} \int_0^T (T-s)^{\alpha -1} E_\alpha (- \lambda s^\alpha ) ds\\
&= \frac{1}{\Gamma(\alpha)} \int_0^T (T-s)^{\alpha -1} \sum_{k=0}^\infty \frac{(-\lambda)^k s^{\alpha k}}{\Gamma(\alpha k +1)}  ds\\
&= \frac{1}{\Gamma(\alpha)} \sum_{k=0}^\infty \frac{(-\lambda)^k}{\Gamma(\alpha k +1)} \int_0^T (T-s)^{\alpha -1} s^{\alpha k} ds\\
&= \frac{1}{\Gamma(\alpha)} \sum_{k=0}^\infty \frac{(-\lambda)^k}{\Gamma(\alpha k +1)} T^{\alpha(k+1)} \frac{\Gamma(\alpha) \Gamma(\alpha k +1)}{\Gamma(\alpha (k+1) +1)}= \lambda^{-1}(1 - E_\alpha (- \lambda T^\alpha )).
\end{align*}
Also, note that
\[
|\phi(t)| \leq \int_{-\infty}^\infty
|H_{\alpha, \lambda} (t-s) \cdot \chi_{t-s \in (0,T)} |
\cdot |f (s) \cdot \chi_{s \in (0,T)} |
\, ds.
\]
By Young's inequality,
\begin{align*}
\norm{\phi}_{L_2 (0,T)}
&\leq
\norm{f \cdot  \chi_{s \in (0,T)}}_{L_2 (\rr)}
\cdot
\norm{\hh_{\alpha, \lambda}  \cdot  \chi_{(0,T) } }_{L_1 (\rr)} \\
&\leq
\lambda^{-1} \norm{f}_{L_2 (0,T)},
\end{align*}
and
$$\norm{\partial_t^\alpha \phi}_{L_2 (0,T)} \le \norm{\lambda \phi}_{L_2 (0,T)} + \norm{f}_{L_2 (0,T)} \le {2}\norm{f}_{L_2 (0,T)}.$$
Thus, we have
$$\norm{w_{n_1, \ldots, n_d }}_{L_2 (0,T)} \le (n_1\pi{\cdots} n_d\pi)^{-2}\norm{f_{n_1, \ldots, n_d }}_{L_2 (0,T)}, $$
and
$$\norm{\partial_t^\alpha w_{n_1, \ldots, n_d }}_{L_2 (0,T)} \le 2\norm{f_{n_1, \ldots, n_d }}_{L_2 (0,T)}. $$
It follows that  $w\in {\hh_{2,0}^{\alpha, 2} (\Omega_T )}$ and
$$
\norm{w}_{\hh_{2}^{\alpha, 2} (\Omega_T )} \le N \norm{f}_{L_2(\Omega_T )}.
$$
Thus, the lemma is proved.
\end{proof}

\begin{lemma}\label{3.7}
Let $p_0\in (1,2)$, $\Omega=(-1,1)^d $, and $f\in L_2(\Omega_T)$.  There exists a unique $w \in \hh_{2,0}^{\alpha, 2} (\Omega_T )$ to
\[
\partial_t^\alpha w - \Delta w = f \quad \text{in}\quad
\Omega_T, \quad
w=0
\quad \text{on}\quad
(0,T) \times \partial \Omega, \stepcounter{equation} \tag{\theequation}\label{1029}
\]
and $w$ satisfies
\[
\norm{ w}_{\hh_{p_0}^{\alpha, 2} (\Omega_T ) }
\leq N \norm{f}_{L_{p_0} (\Omega_T ) },
\stepcounter{equation} \tag{\theequation}\label{local}
\]
where $N = N(d,p_0,\alpha)$.
\end{lemma}

\begin{proof}
 Let $\widetilde{f}(t,y^1,\ldots,y^d) = f(t, 2y^1-1,\ldots,2y^d-1)$ for $(y^1,\ldots,y^d) \in (0,1)^d$. Clearly, $\widetilde{f}\in L_2((0,T)\times (0,1)^d) $, and by Lemma \ref{3.6}, there exists a unique $\widetilde{w} \in \hh_{2,0}^{\alpha, 2} ((0,T)\times (0,1)^d )$ such that
\[
\partial_t^\alpha \widetilde{w}  - 4\Delta \widetilde{w}  = \widetilde{f}  \quad \text{in}\quad
(0,T)\times (0,1)^d,\quad
\widetilde{w}=0
\quad \text{on}\quad
(0,T) \times \partial (0,1)^d.
\]
Then, let $w(t,x^1,\ldots, x^d) = \widetilde{w} (t,(x^1+1)/2,\ldots,(x^d+1)/2 )$ for $(x^1,\ldots,x^d) \in \Omega$. It follows that $w$ satisfies \eqref{1029}. Next, we show \eqref{local} by extending $w$ to the whole space. First, we take the odd extensions of $w$ and $f$ to $(-1,3)^d$, i.e., for $x^i \in (1, 3)$,
$$f(t,x^1,\ldots,x^d) := (-1)^d f(t,2-x^1,\ldots,2-x^d).$$
Then, we take the periodic extensions of $w$ and $f$ to the whole space by extending along each spatial direction with a period of $4$.
Note that in order to smoothly extend the function {up to} the boundary, it is necessary to take the odd extension first.
Let {$M$ be a positive integer and} $\eta\in C^\infty(\rr^d), \eta \in [0,1],$ be a cutoff function in the spatial variables so that
\[
\eta(x) =
\begin{cases}
1 \quad \quad \text{when}\quad  x \in (
{-M,M})^d, \\
0 \quad \quad \text{when}\quad  x \notin  (
{-2M,2M})^d,
\end{cases}
\]
$|D\eta| \leq ({2M})^{-1}$, and $ |D^2\eta|  \leq (2M)^{-2}${.}
Then, it follows from \eqref{10291} that $\eta w \in \hh_{p_0,0}^{\alpha, 2} ( \rr^d_T )$ and
$$ \partial_t^\alpha (\eta w) - \Delta(\eta w) = \eta f - \Delta\eta w -  2D_i\eta D_iw.$$
Using the result for equations with constant leading coefficients in the whole space \cite[Theorem 2.9]{P1}, we have
\begin{align*}
\norm{D^2(\eta w)}_{L_{p_0} (\rr^d_T ) }\leq N \Big[
\norm{\eta f}_{L_{p_0} (\rr^d_T ) }
+ \norm{w D_{ij} \eta}_{L_{p_0} (\rr^d_T ) }
+ \norm{D_i \eta D_j w}_{L_{p_0} (\rr^d_T ) }
\Big],
\end{align*}
where $N = N (d,p_0,\alpha)$.
Therefore, since $w$ is periodic,
\begin{align*}
& {M}^{d/p_0}  \norm{D^2 w}_{L_{p_0} (\Omega_T ) }
\leq N \Big[
({2}M)^{d/p_0}\norm{f}_{L_{p_0} (\Omega_T ) } \\
&\quad\quad+ ({2M})^{-2}  ({2}M)^{d/p_0} \norm{w }_{L_{p_0} (\Omega_T ) } + ({2M})^{-1} ({2}M)^{d/p_0} \norm{Dw}_{L_{p_0} (\Omega_T ) }
\Big],
\end{align*}
which implies
\begin{align*}
\norm{D^2 w}_{L_{p_0} (\Omega_T ) }
\leq N \Big[
\norm{f}_{L_{p_0} (\Omega_T ) }
&+  {M}^{-2} \norm{w }_{L_{p_0} (\Omega_T ) } +  {M}^{-1} \norm{Dw}_{L_{p_0} (\Omega_T ) }
\Big].
\end{align*}
By sending $M \to \infty$, {we get}
\[
\norm{D^2 w}_{L_{p_0} (\Omega_T ) }
\leq N \norm{f}_{L_{p_0} (\Omega_T ) }.
\]
Moreover, we have
\[
\norm{w}_{L_{p_0} ( \Omega_T )}
\leq N \norm{Dw}_{L_{p_0} ( \Omega_T )}
\leq N \norm{D^2 w}_{L_{p_0} ( \Omega_T )},
\]
where the first inequality is by the boundary Poincar\'e inequality, and the second inequality is also by a version of the Poincar\'e inequality using
$$
\int_\Omega Dw(t,x)\, dx = 0.
$$
Therefore, (\ref{local}) is proved.
\end{proof}

\begin{proposition}\label{3.8}
Let $\Omega = (-1,1)^d$, $p_0 \in (1,2)$, $f \in L_2 ( (0,t_0) \times B_{\sqrt{d}} )$, and $w \in \hh_{p_0, 0}^{\alpha, 2} ( (0, t_0) \times \Omega )$ with $w=0$ on $(0, t_0) \times \partial  \Omega$ such that
\[
\partial_t^\alpha w - \Delta w = f \quad \text{in}\quad  (0,t_0) \times \Omega.
\]
Then, for any $t_1 \leq t_0$, we have
\begin{align*}
( |D^2 w|^{p_0} )^{1 / p_0}_{Q_{1/2} (t_1, 0)}
&\leq
N \sum_{k=0}^\infty 2^{-\alpha k} (|f|^{p_0})^{1/p_0}_{ ( s_{k+1},s_k ) \times \Omega} \\
&\leq
N\sum_{k=0}^\infty 2^{-\alpha k} (|f|^{p_0})^{1/p_0}_{ ( s_{k+1},s_k ) \times B_{\sqrt{d}}}, \stepcounter{equation}\tag{\theequation}\label{nonhomoes}
\end{align*}
where $s_k = t_1 - 2^{k} + 1$ and $N =N(\alpha,d,p_0,\delta)$.
\end{proposition}

\begin{proof}
Without loss of generality, we assume that $t_1 = t_0$. By Remark \ref{go}, we have
\begin{align*}
(  |D^2 w|^{p_0} )^{1 / p_0}_{Q_{1/2} (t_0, 0)}
&\leq N(|f|^{p_0})^{1/p_0}_{Q_1 (t_0, 0 )}
+ N \sum_{k=0}^\infty 2^{-k \alpha}
(|w|^{p_0})^{1/p_0}_{ ( s_{k+1},s_k ) \times B_1}
\\
&\leq N(|f|^{p_0})^{1/p_0}_{Q_1 (t_0, 0 )}
+ N \sum_{k=0}^\infty 2^{-k \alpha}
(|w|^{p_0})^{1/p_0}_{ ( s_{k+1},s_k ) \times \Omega}  \stepcounter{equation} \tag{\theequation}\label{nonhomo},
\end{align*}
where $N = N(d, \delta, \alpha, p_0)$.
It remains to estimate
\[
A_k := ( |w|^{p_0} )^{1/p_0}_{(s_{k+1}, s_k) \times \Omega }.
\]
Take cutoff functions $\eta_k, \zeta \in C^{\infty}(\rr)$ such that
\[
\eta_k(t) =
\begin{cases}
1 \quad \quad \text{when}\quad  t \geq s_{k+1}, \\
0 \quad \quad \text{when}\quad  t \leq s_{k+2}
,\end{cases}\quad
\zeta(t) = \begin{cases}
1 \quad \quad \text{when}\quad  t \in (s_{k+3}, t_0), \\
0 \quad \quad \text{when}\quad  t \in (-\infty, s_{k+4} ),
\end{cases}
\]
{$|\eta_k'|  \leq 2^{-k}$, $|\eta_k''|\leq 2^{-2k}$, and $|\zeta| \leq 1$}. Furthermore, let $w_1 = \zeta w$ and $w_2 = w-w_1$. By Lemma \ref{3.1}, it follows that $w\eta_k \in \hh_{p_0, 0}^{\alpha, 2} ( (s_{k+2}, t_0) \times \Omega )$, and
{\[
\partial_t^{\alpha} (w \eta_k ) - \Delta (w \eta_k ) = g_k \quad \text{in}\quad  (s_{k+2}, t_0) \times \Omega, \stepcounter{equation} \tag{\theequation}\label{01045}
\]
where
\begin{align*}
g_k &= f \eta_k + g_{1, k} + g_{2, k}, \\
g_{k,1}(t,x)  &= \frac{ \alpha (\alpha - 1)}{\Gamma (2 - \alpha)}  \int_{s_{k+4}}^t (t-s)^{-\alpha - 1}
[ \eta_k(s) - \eta_k(t) - \eta_k'(t) (s-t) ] w_1(s,x) \, ds \\
&\quad+ \frac{ \alpha}{\Gamma (2 - \alpha)}  \eta_k'(t)  \partial_t
\int_{s_{k+4}}^t (t-s)^{1- \alpha} w_1(s,x) \, ds,
\end{align*}}
and
\begin{align*}
g_{k,2}(t,x)&=\frac{ \alpha (\alpha - 1)}{\Gamma (2 - \alpha)} \int_{-\infty}^{s_{k+3}} (t-s)^{-\alpha - 1}
[ \eta_k(s) - \eta_k(t) - \eta_k'(t) (s-t) ] w_2(s,x) \, ds  \\
&\quad+ \frac{ \alpha}{\Gamma (2 - \alpha)}  \eta_k'(t) \partial_t
\int_{-\infty}^{s_{k+3}} (t-s)^{1- \alpha} w_2(s,x) \, ds \\
&= N
\int_{-\infty}^{s_{k+3}} (t-s)^{-\alpha - 1}
[ \eta_k(s) - \eta_k(t) ]w_2(s,x) \, ds.
\end{align*}
Using the fact that $|\eta_k (s) - \eta_k (t) | = |\eta_k (t) |$
for $s < s_{k + 3}$, we have
\[
|g_{k,2}(t,x)|
\leq \sum_{j = k+3}^\infty \int_{s_{j + 1} }^{s_j}
(t-s)^{-\alpha - 1} |w(s,x)| \, ds.
\]
Moreover, note that for $t \in (s_{k+2}, s_k )$, $s \in (s_{j+1}, s_j )$, and $ j\ge k+3$, we have
\[
|t-s| \geq s_{k+2} - s_j
= 2^j - 2^{k+2}
\geq 2^j - 2^{j-1}
= 2^{j-1},
\]
and thus,
\[
|t-s|^{-\alpha - 1} \leq  N (\alpha) 2^{-j(\alpha + 1)}.
\]
Therefore, by the Minkowski inequality and H\"older’s inequality
\[
\norm{ g_{k,2} }_{L_{p_0} ( ( s_{k+2}, s_k ) \times \Omega  )}
\leq
N
\sum_{j=k+3}^\infty
2^{-\alpha j} (3 \cdot 2^k)^{1 / p_0}
\Bigg( \fint_{s_{j+1}}^{s_j} \fint_\Omega |w|^{p_0} \Bigg)^{1/ p_0},\stepcounter{equation} \tag{\theequation}\label{mm1}
\]
where $N = N(\alpha, d) $.

For $g_{k,1}$, we decompose
\[
g_{k, 1}
= g_{k, 1,1}
+ g_{k,1,2} + g_{k,1,3},
\]
where
\begin{align*}
g_{k, 1,1}(t,x)
&= N_1(\alpha) \int_{s_{k+4}}^{t} (t-s)^{-\alpha - 1}
[ \eta_k(s) - \eta_k(t)   - \eta_k' (t)  (s-t)  ]w_1(s,x) \, ds, \\
g_{k,1,2}(t,x)
&= N_1(\alpha) \frac{\eta_k'(t)}{\eta_{k+1}(t)}\int_{s_{k+4}}^{t} (t-s)^{-\alpha }[ \eta_{k+1}(s) - \eta_{k+1}(t) ]w_1(s,x) \, ds,
\end{align*}
and
\begin{align*}
g_{k,1,3}(t,x)
&= N_2(\alpha)  \frac{\eta_k'(t)}{\eta_{k+1}(t)}  \partial_t^{\alpha-1} (\eta_{k+1} w_1)(t,x)
= N_2(\alpha)  \frac{\eta_k'(t)}{\eta_{k+1}(t)}  \partial_t^{\alpha-1} (\eta_{k+1} w)(t,x).
\end{align*}
When
$t\ge s > s_{k+1}$, we have $\eta_k (s) = \eta_k (t) = 1$ and $\eta_k'(t) = 0$, which imply
\[
\eta_k (s) - \eta_k(t) - \eta_k ' (t) (s-t) = 0.
\]
Therefore, it follows
\begin{align*}
|g_{k,1,1} (t,x)|
&\leq N  \norm{\eta_k''}_{L_\infty}
\int^t_{s_{k+4}} (t-s)^{-\alpha+1} |w(s,x)| 1_{s \leq s_{k+1}} \, ds \\
&\leq N  2^{-2k}
\int_0^{15 2^k}
s^{-\alpha+1} |w (t-s, x) | 1_{s \geq t - s_{k+1}} \, ds,
\end{align*}
which implies
\begin{align*}
&\norm{g_{k,1,1}}_{L_{p_0}  ( ( s_{k+2}, s_k ) \times \Omega  )}\\
&\leq N 2^{-2k} \int_0^{15 \cdot 2^k}
s^{-\alpha+1} \norm{w (\cdot - s, \cdot)  1_{\cdot-s \leq s_{k+1} }   }_{L_{p_0} (        (s_{k+2}, s_k) \times \Omega  )     }  \, ds \\
&\leq N 2^{-\alpha k}
\norm{w}_{L_{p_0} ( ( s_{k+5}, s_{k+1}  ) \times \Omega ) } \leq N 2^{-\alpha k}
\sum_{j=1}^4 \norm{w}_{L_{p_0}  (  (s_{k + j + 1}, s_{k + j} ) \times \Omega ) },\stepcounter{equation} \tag{\theequation}\label{mm2}
\end{align*}
where $N = N(\alpha, d, p_0)$. Furthermore,
\begin{align*}
|g_{k,1,2}(t,x) |
&\leq N  \norm{\eta_k'}_{L_\infty} 1_{t \in (s_{k+2}, s_{k+1} )} \norm{\eta_{k+1}'}_{L_\infty}
\int^t_{s_{k+4}} (t-s)^{-\alpha+1} |w(s,x)| 1_{s \leq s_{k+2}} \, ds \\
&\leq N  2^{-k}  2^{-k} 1_{t \in (s_{k+2}, s_{k+1} )}
\int^t_{s_{k+4}} (t-s)^{-\alpha+1} |w(s,x)| 1_{s \leq s_{k+1}} \, ds \\
&\leq N  2^{-2k} 1_{t \in (s_{k+2}, s_{k+1} )}
\int_0^{15 2^k}
s^{-\alpha +1} |w (t-s, x) | 1_{s \geq t - s_{k+1}} \, ds,
\end{align*}
which implies that
\begin{align*}
\norm{g_{k,1,2}}_{L_{p_0}  ( ( s_{k+2}, s_k ) \times \Omega  )}
&\leq N 2^{-\alpha k}
\sum_{j=1}^4 \norm{w}_{L_{p_0}  (  (s_{k + j + 1}, s_{k + j} ) \times \Omega ) },\stepcounter{equation} \tag{\theequation}\label{mm3}
\end{align*}
where $N = N(\alpha, d, p_0)$. For any $\ep>0$, using Lemma \ref{3.2}, we have
\begin{align*}
&\norm{g_{k,1,3} }_{L_{p_0}  ( ( s_{k+2}, s_k ) \times \Omega  )}= N \norm{ \eta_k' (t) \partial_t^\beta (\eta_{k+1} w )}_{L_{p_0} ( (s_{k+2}, s_{k}) \times \Omega )} \\
&= N \norm{ \eta_k' (t) \partial_t^\beta (\eta_{k+1} w )}_{L_{p_0} ( (s_{k+2}, s_{k+1}) \times \Omega )} \\
&\leq N \norm{ \eta_k'}_{L_\infty}  \norm{\partial_t^\beta (\eta_{k+1} w )}_{L_{p_0} ( (s_{k+3}, s_{k+1}) \times \Omega )} \\
&\leq \ep \ \norm{\partial_t^\alpha (\eta_{k+1} w)}_{L_{p_0} ( (s_{k+3}, s_{k+1}) \times \Omega )}\\
&\quad + N\ep^{-1}   \norm{\eta_k'}_{L_\infty}^2  \norm{w}_{L_{p_0} ( (s_{k+3}, s_{k+1}) \times \Omega )}  (2^k)^{2 - \alpha} \\
&\leq \ep \norm{\eta_{k+1} w}_{ \hh_{p_0}^{\alpha, 2} ( (s_{k+3}, s_{k+1}) \times \Omega ) }+ N\ep^{-1}  2^{-\alpha k } \norm{w}_{L_{p_0} ( (s_{k+3}, s_{k+1}) \times \Omega )}. \stepcounter{equation} \tag{\theequation}\label{mm4}
\end{align*}
Thus, using Lemma \ref{3.7}, and combining \eqref{mm1}, \eqref{mm2}, \eqref{mm3}, and \eqref{mm4}, we derive
\begin{align*}
&\norm{w \eta_k}_{\hh_{p_0}^{\alpha, 2} ( (s_{k+2}, s_k ) \times \Omega )}\leq N \norm{g_k}_{L_{p_0} ( (s_{k+2},  s_k ) \times \Omega ) } \\
&\leq N \norm{f}_{L_{p_0} ( (s_{k+2},  s_k ) \times \Omega ) }
+ N \norm{g_{k,1}}_{L_{p_0} ( (s_{k+2},  s_k ) \times \Omega ) }
+ N \norm{g_{k, 2}}_{L_{p_0} ( (s_{k+2},  s_k ) \times \Omega ) } \\
&\leq N \norm{f}_{L_{p_0} ( (s_{k+2},  s_k ) \times \Omega ) } + N 2^{-\alpha k}
\sum_{j=1}^4 \norm{w}_{L_{p_0}  (  (s_{k + j + 1}, s_{k + j} ) \times \Omega ) } \\
&\quad + \ep \norm{\eta_{k+1} w}_{ \hh_{p_0}^{\alpha, 2} ( (s_{k+3}, s_{k+1}) \times \Omega ) }+ N\ep^{-1}  2^{-\alpha k } \norm{w}_{L_{p_0} ( (s_{k+3}, s_{k+1}) \times \Omega )} \\
&\quad +     N
\sum_{j=k+4}^\infty
2^{-\alpha j} (3 \cdot 2^k)^{1 / p_0}
\Bigg( \fint_{s_{j+1}}^{s_j} \fint_\Omega |w|^{p_0} \Bigg)^{1/ p_0},
\end{align*}
where $N = N(\alpha,d,p_0, \delta) $.
Dividing both sides by the measure, we have
\begin{align*}
&\norm{\eta_k w }_{\hh_{p_0}^{\alpha, 2} ( (s_{k+2}, s_k ) \times \Omega  )}
\big( { |\Omega|  |s_k - s_{k+2}  | } \big)^{-1/p_0}\\
&\leq N  ( |f|^{p_0} )^{1/p_0}_{(s_{k+2}, s_k ) \times \Omega} + N \sum_{j = k+1}^\infty 2^{-\alpha j} A_j+  N\ep^{-1} \sum_{j=k+1}^{k+2}  2^{-\alpha j}    A_j \\
&\quad+ 6^{1/p_0} \ep \norm{\eta_{k+1} w}_{\hh_{p_0}^{\alpha, 2} ( (s_{k+3}, s_{k+1} ) \times \Omega  )}
 \big( { |\Omega|  |s_{k+3} - s_{k+1}  | }  \big)^{-1/p_0}.\stepcounter{equation}\tag{\theequation}\label{tt2}
\end{align*}
For any fixed $k_0$, multiplying both sides of \eqref{tt2} by $2^{-\alpha k}$, and summing over $k= k_0, k_0+1, \ldots$, we get
\begin{align*}
 & \sum_{k = k_0}^\infty  2^{-\alpha k}
\norm{\eta_k w }_{\hh_{p_0}^{\alpha, 2} ( (s_{k+2}, s_k ) \times \Omega  )}
 \big( { |\Omega|  |s_k - s_{k+2}  | }  \big)^{-1/p_0} \\
&\leq N  \sum_{k = k_0}^\infty  2^{-\alpha k}
 ( |f|^{p_0} )^{1/p_0}_{(s_{k+2}, s_k ) \times \Omega} + N \sum_{k = k_0}^\infty  2^{-\alpha k}
 \sum_{j = k+1}^\infty 2^{-\alpha j} A_j \\
&\quad+ 6^{1/p_0} \ep \sum_{k = k_0}^\infty  2^{-\alpha k}
 \norm{\eta_{k+1} w}_{\hh_{p_0}^{\alpha, 2} ( (s_{k+3}, s_{k+1} ) \times \Omega  )}
 \big( { |\Omega|  |s_{k+3} - s_{k+1}  | }  \big)^{-1/p_0} \\
&\quad+  N\ep^{-1} \sum_{k = k_0}^\infty  2^{-\alpha k}
 \sum_{j=k+1}^{k+2}  2^{-\alpha j}  A_j. \stepcounter{equation} \tag{\theequation}\label{uui}
\end{align*}
By picking $\ep$ sufficiently small so that $6^{1/p_0}  \ep 2^\alpha \leq 1/2$, \eqref{uui} becomes
\begin{align*}
 & \sum_{k = k_0}^\infty  2^{-\alpha k}
\norm{\eta_k w}_{\hh_{p_0}^{\alpha, 2} ( (s_{k+2}, s_k ) \times \Omega  )}
 \big( { |\Omega|  |s_k - s_{k+2}  | }  \big)^{-1/p_0} \\
&\leq N  \sum_{k = k_0}^\infty  2^{-\alpha k}
 ( |f|^{p_0} )^{1/p_0}_{(s_{k+2}, s_k ) \times \Omega} + N \sum_{k = k_0}^\infty  2^{-\alpha k}
 \sum_{j = k+1}^\infty 2^{-\alpha j} A_j \\
&\quad+ \frac{1}{2} \sum_{k = k_0}^\infty  2^{-\alpha(k+1)}
 \norm{\eta_{k+1} w}_{\hh_{p_0}^{\alpha, 2} ( (s_{k+3}, s_{k+1} ) \times \Omega  )}
 \big( { |\Omega| |s_{k+3} - s_{k+1}  | }  \big)^{-1/p_0}. \stepcounter{equation} \tag{\theequation}\label{yy2}
\end{align*}
By absorbing the third term to the left-hand side of \eqref{yy2}, we conclude
\begin{align*}
\sum_{k = k_0}^\infty 2^{-\alpha k} A_k
&\leq \sum_{k = k_0}^\infty  2^{-\alpha k}
\norm{\eta_k w }_{\hh_{p_0}^{\alpha, 2} ( (s_{k+2}, s_k ) \times \Omega  )} \cdot
\big( { |\Omega| \cdot |s_k - s_{k+2}  | }  \big)^{-1/p_0} \\
&\leq N  \sum_{k = k_0}^\infty  2^{-\alpha k}
 ( |f|^{p_0} )^{1/p_0}_{(s_{k+2}, s_k ) \times \Omega}
 + N \sum_{k = k_0}^\infty  2^{-\alpha k}
 \sum_{j = k+1}^\infty 2^{-\alpha j} A_j \\
 &\leq N \sum_{k = k_0}^\infty  2^{-\alpha k}
 ( |f|^{p_0} )^{1/p_0}_{(s_{k+2}, s_k ) \times \Omega}
 + N \sum_{j = k_0 + 1}^\infty  2^{-\alpha j} A_j
 \sum_{k = k_0}^{j-1} 2^{-\alpha k} \\
 &\leq  N \sum_{k = k_0}^\infty  2^{-\alpha k}
 ( |f|^{p_0} )^{1/p_0}_{(s_{k+2}, s_k ) \times \Omega}
 +      \frac{N 2^{-\alpha k_0}}{1 - 2^{-\alpha}}
 \sum_{j = k_0 + 1}^{\infty} 2^{-\alpha j} A_j, \stepcounter{equation}\tag{\theequation}\label{ll}
\end{align*}
where $N = N(\alpha,d,p_0, \delta)$.
By picking $k_0$ sufficiently large so that $ \frac{N  2^{-\alpha k_0}}{1 - 2^{-\alpha}} \leq 1/2$, we have
\[
\sum_{k = k_0}^{\infty} 2^{-\alpha k} A_k
\leq  N \sum_{k = k_0}^\infty  2^{-\alpha k}
 ( |f|^{p_0} )^{1/p_0}_{(s_{k+2}, s_k ) \times \Omega}.
\]
Therefore, by induction,
\[\sum_{k = 0}^{\infty} 2^{-\alpha k} A_k
\leq  N(\alpha,d,p_0,\delta) \sum_{k = 0}^\infty  2^{-\alpha k}
 ( |f|^{p_0} )^{1/p_0}_{(s_{k+2}, s_k ) \times \Omega}. \stepcounter{equation}\tag{\theequation}\label{mm5}\]
Indeed, if there exists $N = N (\alpha,d,p_0,\delta)$ such that
$$\sum_{k = j}^{\infty} 2^{-\alpha k} A_k
\leq  N \sum_{k = j}^\infty  2^{-\alpha k}
 ( |f|^{p_0} )^{1/p_0}_{(s_{k+2}, s_k ) \times \Omega},$$
for $j = 1,2,\ldots, k_0$, then, by (\ref{ll}),
\begin{align*}
\sum_{k = j-1}^{\infty} 2^{-\alpha k} A_k &\le
 N \sum_{k = j-1}^\infty  2^{-\alpha k}
 ( |f|^{p_0} )^{1/p_0}_{(s_{k+2}, s_k ) \times \Omega}
 +    \big(  \frac{N 2^{-\alpha k_0}}{1 - 2^{-\alpha}} \big)
 \sum_{k = j}^{\infty} 2^{-\alpha k} A_k \\
 &\le  N
 \sum_{k = j-1}^\infty  2^{-\alpha k}
 ( |f|^{p_0} )^{1/p_0}_{(s_{k+2}, s_k ) \times \Omega} = N
 \sum_{k = j-1}^\infty  2^{-\alpha k}
 ( |f|^{p_0} )^{1/p_0}_{(s_{k+2}, s_k ) \times \Omega}.
\end{align*}
Finally, by (\ref{nonhomo}) and \eqref{mm5},
\begin{align*}
( |D^2 w|^{p_0} )^{1 / p_0}_{Q_{1/2} (t_1, 0)}
&\leq
N \sum_{k=0}^\infty 2^{-\alpha k} (|f|^{p_0})^{1/p_0}_{ ( s_{k+1},s_k ) \times \Omega} \\
&\leq
N\sum_{k=0}^\infty 2^{-\alpha k} (|f|^{p_0})^{1/p_0}_{ ( s_{k+1},s_k ) \times B_{\sqrt{d}}},
\end{align*}
where $N = N (\alpha,d,p_0,\delta)$.
The proposition is proved.
\end{proof}

\subsection{Estimates of \texorpdfstring{$u$}{u}}\label{33.33}
By combining the estimates of $v$ and $w$, we derive the mean oscillation estimates for $u$ in this subsection.
\begin{proposition}\label{3.9}
Let $p_0 \in (1, 2)$, $T \in (0, \infty)$, and $a^{ij}$ be constant. If $u \in \hh_{p_0, 0, \mathrm{loc}}^{\alpha, 2} ( \rr^d_T )$ satisfies
\[
\partial_t^\alpha u - a^{ij} D_{ij} u = f
\quad \text{in}\quad  \rr^d_T.
\]
Then, for any $(t_0, x_0) \in (0,T] \times \rr^d$,  $r \in (0, \infty)$, and $\kappa \in (0, 1/4)$, we have
\begin{align*}
& ( |D^2 u - (D^2 u )_{(t_0-(\kappa r)^{2/\alpha},t_0) \times B_{\delta\kappa r}(x_0))} | )_{
(t_0-(\kappa r)^{2/\alpha},t_0) \times B_{\delta \kappa r}(x_0)
} \\
&\quad \leq N \kappa^\sigma   \sum_{k=0}^\infty {2^{-k\alpha}}( |D^2 u |^{p_0} )^{1 / p_0}_{(t_0-2^{k}(r/2)^{2/\alpha}r,t_0)\times B_{\delta^{-1}r/2}(x_0)}\\
&\quad \quad+ N \kappa^{- (d+2 /\alpha)/p_0}
\sum_{k=0}^\infty 2^{-\alpha k} (|f|^{p_0})^{1/p_0}_{
( t_0 - 2^k r^{2/\alpha}, t_0 )
\times B_{\delta^{-1} \sqrt{d} r}{(x_0)}}
\stepcounter{equation} \tag{\theequation}\label{u},
\end{align*}
where $\sigma = \sigma (d, \alpha, p_0)$ and $N = N(d, \delta, \alpha, p_0)$.
\end{proposition}

\begin{proof}
The proof is similar to that of  \cite[Proposition 4.7]{dong20}, where $\alpha$ is taken to be smaller than $1$.

{\em Step 1.} We start with the case when $a^{ij} = \delta^{ij}$. By shifting and scaling, we also assume $x_0 = 0$ and $r=1$.
Let $\Omega = (-1,1)^d$ be a cube in $\rr^d$ such that $B_1 \subseteq \Omega \subseteq B_{\sqrt d}$.
Then, by Lemma \ref{3.7}, there exists $w \in \hh_{p_0, 0}^{\alpha, 2} ( (0,t_0) \times \Omega )$ with $w=0$ on $(0,t_0) \times \partial \Omega$ satisfying
\[
\partial_t^\alpha w - \Delta w = f
\quad \text{in}\quad  (0,t_0) \times \Omega.
\]
By taking $v= u - w \in \hh_{p_0, 0}^{\alpha, 2} ( (0,t_0) \times \Omega)$, we have
\begin{align*}
\partial_t^\alpha v &- \Delta v = 0 \quad \text{in}\quad  (0,t_0) \times \Omega.
\end{align*}
Next, by H\"older's inequality,
\begin{align*}
& ( |D^2 u - (D^2 u )_{Q_{\kappa} (t_0, 0)} | )_{
Q_{\kappa} (t_0, 0) } \\
&\quad \leq
( |D^2 v - (D^2 v )_{Q_{\kappa} (t_0, 0)} | )_{
Q_{\kappa} (t_0, 0)}
+ N \kappa^{- (d+2 /\alpha)/p_0}
(| D^2 w|^{p_0} )^{1/p_0}_{Q_{1/2} (t_0, 0)} \stepcounter{equation} \tag{\theequation}\label{12055}.
\end{align*}

{\em Step 2.} We focus on the estimate of the first term in the right-hand side of \eqref{12055} since the estimate of the second term in given by \eqref{nonhomoes}.

It is easily seen that
\[
(|D^2 v - (D^2 v)_{Q_\kappa ( t_0,0) } | )_{Q_\kappa ( t_0,0)}
\leq N \kappa^\sigma [D^2 v]_{C^{\sigma \alpha/2, \sigma} (Q_{1/4} (t_0, 0)  )  } \stepcounter{equation} \tag{\theequation}\label{10304},
\]
and by applying Proposition \ref{homo} with $r =1/2$ together with the triangle inequality, we have
\begin{align*}
[D^2 v]_{C^{\sigma \alpha/2, \sigma} (Q_{1/4} (t_0, 0)  )  }
&\leq N
\sum_{j=1}^\infty
j^{- (1 + \alpha)}
( |D^2 v|^{p_0} )^{1/p_0}_{Q_{1/2}
(t_0 - (j-1) 2^{-2/\alpha}, 0) } \\
&\leq N
\sum_{j=1}^\infty
j^{- (1 + \alpha)}
( |D^2 u|^{p_0} )^{1/p_0}_{Q_{1/2}
(t_0 - (j-1) 2^{-2/\alpha}, 0) } \\
&\quad+ N
\sum_{j=1}^\infty
j^{- (1 + \alpha)}
( |D^2 w|^{p_0} )^{1/p_0}_{Q_{1/2}
(t_0 - (j-1) 2^{-2/\alpha}, 0) }. \stepcounter{equation} \tag{\theequation}\label{10303}
\end{align*}
It remains to estimate the summation involving $D^2w$ in \eqref{10303}. For $j=1,2,\ldots$ and $s_k^j = t_0 - (j-1) 2^{-2 /\alpha } - 2^k + 1$, by (\ref{nonhomoes}), we have
\[
( |D^2 w|^{p_0} )^{1/p_0}_{Q_{1/2}
(t_0 - (j-1) 2^{-2/\alpha}, 0) }
\leq N\sum_{k=0}^\infty
2^{-\alpha k} (|f|^{p_0})^{1/p_0}_{ (s^j_{k+1}, s^j_k) \times B_{\sqrt{d}}  }.
\]
It follows that
\begin{align*}
&\sum_{j=1}^\infty j^{-(1 + \alpha)}
(|D^2 w|^{p_0})^{1/ p_0}_{Q_{1/2}
(t_0 - (j-1) 2^{-2/\alpha}, 0) } \\
&\leq N \sum_{j=1}^\infty j^{-(1 + \alpha)}
\sum_{k=0}^\infty 2^{-\alpha k}
(|f|^{p_0})^{1/p_0}_{(s^j_{k+1}, s^j_k)\times {B_{\sqrt{d} }}} \\
&= N \sum_{m=1}^\infty
\sum_{\substack{
j \in \nn \\
m-1 \leq (j-1) 2^{-2 /\alpha} < m
}}
j^{-(1 + \alpha)}
\sum_{k=0}^\infty
2^{-\alpha k}
(|f|^{p_0})^{1/p_0}_{(s^j_{k+1}, s^j_k)\times {B_{\sqrt{d} }}}.
\end{align*}
Moreover, observe that for  $m-1 \leq (j-1) 2^{-2 /\alpha} < m$,
$$
(s^j_{k+1}, s^j_k)
\subset (t_0 - 2^{k+1} + 1 - m, t_0 - m + 1).
$$
Then, for $s_k = t_0 - 2^k + 1$,
\[
(|f|^{p_0})^{1/p_0}_{(s^j_{k+1}, s^j_k)\times {B_{\sqrt{d} }}}
\leq
2^{1/p_0}
(|f|^{p_0})^{1/p_0}_{(s_{k+1} - m, t_0 - m+1) \times {B_{\sqrt{d} }}}.\stepcounter{equation} \tag{\theequation}\label{1205}
\]
Thus, {by using \eqref{1205} and H\"older's inequality,}
\begin{align*}
&\sum_{j=1}^\infty j^{-(1 + \alpha)}
(|D^2 w|^{p_0})^{1/ p_0}_{Q_{1/2}
(t_0 - (j-1) 2^{-2/\alpha}, 0) }  \\
&\leq N
\sum_{m=1}^\infty
\sum_{\substack{
j \in \nn \\
m-1 \leq (j-1) 2^{-2 /\alpha} < m
}}
j^{-(1 + \alpha)}
\sum_{k=0}^\infty
2^{-\alpha k}
(|f|^{p_0})^{1/p_0}_{(s_{k+1} - m, t_0 - m+1) \times {B_{\sqrt{d} }}} \\
&\leq N
\sum_{m=1}^\infty
m^{-(1 + \alpha)}
\sum_{k=0}^\infty
2^{-\alpha k}
(|f|^{p_0})^{1/p_0}_{(s_{k+1} - m, t_0 - m+1) \times {B_{\sqrt{d} }}} \\
&\leq N
\sum_{l=0}^\infty
\sum_{m = 2^l}^{2^{l+1} - 1}
2^{- l (1 + \alpha)}
\sum_{k=0}^\infty
2^{-\alpha k}
(|f|^{p_0})^{1/p_0}_{(s_{k+1} - m, t_0 - m+1) \times {B_{\sqrt{d} }}} \\
&= N
\sum_{l=0}^\infty
2^{- l (1 + \alpha)}
\sum_{k=0}^\infty
2^{-\alpha k}
\sum_{m = 2^l}^{2^{l+1} - 1}
(|f|^{p_0})^{1/p_0}_{(s_{k+1} - m, t_0 - m+1) \times {B_{\sqrt{d} }}} \\
&\leq N
\sum_{l=0}^\infty
2^{- l\alpha}
\sum_{k=0}^\infty
2^{-\alpha k}
\Bigg[
\sum_{m = 2^l}^{2^{l+1} - 1}
2^{-l}
(|f|^{p_0})_{(s_{k+1} - m, t_0 - m+1) \times {B_{\sqrt{d} }}}
\Bigg]^{1/p_0}. \stepcounter{equation} \tag{\theequation}\label{12051}
\end{align*}
Furthermore, for the most inner sum of \eqref{12051}, we have
\begin{align*}
&\sum_{m = 2^l}^{2^{l+1} - 1}
2^{-l}
(|f|^{p_0})_{(s_{k+1} - m, t_0 - m+1) \times {B_{\sqrt{d} }}} \\
&= 2^{-l} 2^{-(k+1)}
\sum_{m = 2^l}^{2^{l+1} - 1}
\int_{s_{k+1} - m}^{t_0 - m + 1}
\fint_{{B_{\sqrt{d} }}} |f|^{p_0} \, dx \, dt \\
&= 2^{-l} 2^{-(k+1)}
\sum_{m = 2^l}^{2^{l+1} - 1}
\sum_{i=0}^{2^{k+1} - 1}
\int_{t_0 - m - i}^{t_0 - m - i + 1}
\fint_{{B_{\sqrt{d} }}} |f|^{p_0} \, dx \, dt \\
&\leq 2^{-l-k} 2^{\min \{ l, k \} }
\sum_{m = 2^l}^{2^{l+1} + 2^{k+1} - 2}
\int_{t_0 - m }^{t_0 - m  + 1}
\fint_{{B_{\sqrt{d} }}} |f|^{p_0} \, dx \, dt  \\
&\leq \begin{cases}
4 (|f|^{p_0})_{(t_0 - 2^{l+2} + 2, t_0)  \times {B_{\sqrt{d} }}} &\quad \text{for}\quad  l \geq k, \\
4 (|f|^{p_0})_{(t_0 - 2^{k+2} + 2, t_0)  \times {B_{\sqrt{d} }}} &\quad \text{for}\quad  k > l.
\end{cases} \stepcounter{equation} \tag{\theequation}\label{12052}
\end{align*}
From \eqref{12051} and \eqref{12052}, it follows that
\begin{align*}
&\sum_{j=1}^\infty j^{-(1 + \alpha)}
(|D^2 w|^{p_0})_{Q_{1/2}
(t_0 - (j-1) 2^{-2/\alpha}, 0) }\\
&\leq
N \sum_{k=0}^\infty 2^{-\alpha k}
(|f|^{p_0})^{1/p_0}_{(t_0 - 2^{k+2} + 2, t_0)  \times {B_{\sqrt{d} }}} \stepcounter{equation} \tag{\theequation}\label{10302}.
\end{align*}
Thus, by \eqref{10304}, \eqref{10303}, and \eqref{10302},
\begin{align*}
( |D^2 v - (D^2 v )_{Q_{\kappa} (t_0, 0)} | )_{
Q_{\kappa (t_0, 0)} }
&\leq N \kappa^\sigma
\sum_{j=1}^\infty j^{-(1 + \alpha)}
(|D^2 u |^{p_0})^{1/p_0}_{Q_{1/2} (t_0 - (j-1)2^{-2/\alpha}, 0 ) } \\
&\quad+
 N \kappa^\sigma
\sum_{k=1}^\infty
2^{-\alpha k} (|f|^{p_0})^{1/p_0}_{ (t_0 - 2^{k+2} + 2, t_0 ) \times {B_{\sqrt{d} }} }. \stepcounter{equation} \tag{\theequation}\label{yy3}
\end{align*}

{\em Step 3.} By \eqref{12055}, \eqref{yy3}, and \eqref{nonhomoes}, we conclude
\begin{align*}
& ( |D^2 u - (D^2 u )_{Q_{\kappa} (t_0, 0)} | )_{
Q_{\kappa} (t_0, 0) } \\
& \leq
( |D^2 v - (D^2 v )_{Q_{\kappa} (t_0, 0)} | )_{
Q_{\kappa} (t_0, 0)}
+ N \kappa^{- (d+2 /\alpha)/p_0}
(| D^2 w|^{p_0} )^{1/p_0}_{Q_{1/2} (t_0, 0)} \\
& \leq N\kappa^\sigma \sum_{j=1}^\infty j^{-(1 + \alpha)}
( |D^2 u |^{p_0} )^{1 / p_0}_{Q_{1/2} (t_0 - (j-1)2^{-2/\alpha}, 0 ) } \\
& \quad+ N \kappa^{- (d+2/\alpha)/p_0}
\sum_{k=0}^\infty 2^{-\alpha k}
(|f|^{p_0})^{1/p_0}_{(t_0 -2^k,t_0 ) \times {B_{\sqrt{d} }}},\stepcounter{equation} \tag{\theequation}\label{iyy3}
\end{align*}
where $N = N(d, \delta, \alpha, p_0)$. For the first term on the right-hand side of \eqref{iyy3}, by H\"older's inequality for $l_{p_0}$,
\begin{align*}
&\sum_{j=1}^\infty j^{-(1 + \alpha)}( |D^2 u |^{p_0} )^{1 / p_0}_{Q_{1/2} (t_0 - (j-1)2^{-2/\alpha}, 0 ) }\\
&\le \sum_{k=0}^\infty \sum_{j= 2^k}^{2^{k+1}-1}2^{-k(1 + \alpha)}( |D^2 u |^{p_0} )^{1 / p_0}_{Q_{1/2} (t_0 - (j-1)2^{-2/\alpha}, 0 ) }\\
&\le \sum_{k=0}^\infty 2^{-k\alpha} \bigg( 2^{-k}\sum_{j=2^k}^{2^{k+1}-1} ( |D^2 u |^{p_0} )_{Q_{1/2} (t_0 - (j-1)2^{-2/\alpha}, 0 ) }\bigg)^{1 / p_0}\\
&\le N  \sum_{k=0}^\infty {2^{-k\alpha}}( |D^2 u |^{p_0} )^{1 / p_0}_{(t_0-2^{k}2^{-2/\alpha},t_0)\times B_{1/2}}
.\stepcounter{equation} \tag{\theequation}\label{iyy4}
\end{align*}
Therefore, (\ref{u}) follows from \eqref{iyy3} and \eqref{iyy4} when $a^{ij}=\delta^{ij}$ and $r = 1 $.

{\em Step 4.} For general $a^{ij}$'s, we apply a change of variables. Replacing $a^{ij}$ with $(a^{ij}+a^{ji})/2$, without loss of generality, we may assume that $A = (a^{ij}) $ is symmetric and positive definite. Then, there exists a symmetric and invertible $A^{1/2}$. Let $$\widetilde{u}(t,y) = u(t, A^{1/2}y) \quad \text{in}\quad  \rr^d_T \quad \text{and}\quad  y_0 = A^{-1/2}x_0.$$
Therefore, \[
\partial_t^\alpha \widetilde{u} - \Delta \widetilde{u} = \widetilde{f}
\quad \text{in}\quad  \rr^d_T
\]
and
\begin{align*}
& ( |D^2 \widetilde{u} - (D^2 \widetilde{u} )_{Q_{\kappa r} (t_0, y_0)} | )_{
Q_{\kappa r (t_0, y_0)}
} \\
& \leq N \kappa^\sigma  \sum_{k=0}^\infty {2^{-k\alpha}}( |D^2 \widetilde{u} |^{p_0} )^{1 / p_0}_{(t_0-2^{k}(r/2)^{2/\alpha}r,t_0)\times B_{r/2}(y_0)}\\
&\quad+ N \kappa^{- (d+2 /\alpha)/p_0}
\sum_{k=0}^\infty 2^{-\alpha k} (|\widetilde{f}|^{p_0})^{1/p_0}_{
( t_0 -2^k r^{2/\alpha}, t_0 )
\times B_{ \sqrt{d} r}(y_0).
}
\end{align*}
It follows that
\begin{align*}
& ( |D^2 u - (D^2 u )_{(t_0-(\kappa r)^{2/\alpha},t_0) \times B_{\delta\kappa r}(x_0))} | )_{
(t_0-(\kappa r)^{2/\alpha},t_0) \times B_{\delta \kappa r}(x_0)
} \\
& \leq N \kappa^\sigma    \sum_{k=0}^\infty {2^{-k\alpha}} ( |D^2 u |^{p_0} )^{1 / p_0}_{(t_0-2^{k}(r/2)^{2/\alpha}r,t_0)\times B_{\delta^{-1}r/2}(x_0)} \\
& \quad+ N \kappa^{- (d+2 /\alpha)/p_0}
\sum_{k=0}^\infty 2^{-\alpha k} (|f|^{p_0})^{1/p_0}_{
( t_0 - 2^kr^{2/\alpha}, t_0 )
\times B_{\delta^{-1} \sqrt{d} r}{(x_0)}},
\end{align*}
where we used the ellipticity condition of $a^{ij}$'s to derive
$$B_{\delta r}(x_0) \subset A^{1/2}(B_{r}(y_0)) \subset B_{\delta^{-1} r}(x_0),$$
and
$$|B_{\delta^{-1} r}(x_0)| \le N(d,\delta) |A^{1/2}(B_{r}(y_0))| \le N(d,\delta) |B_{\delta r}(x_0)| $$
for any $r>0$. The proposition is proved.
\end{proof}

\section{Proof of Theorem \ref{main}}\label{4}
In this section, with the mean oscillation estimate (\ref{u}), we will prove Theorem \ref{main} by using the sharp function theorem and the maximal function theorem.

\begin{proposition}\label{4.1}
Let $\alpha \in (1,2)$, $T \in (0, \infty)$, $p,q \in (1, \infty)$, $K_1 \in [1, \infty)$, and $[w]_{p,q} \le K_1$.
There exist $p_0 = p_0 (d, p, q, K_1) \in (1, 2)$ and $\mu = \mu (d, p, q, K_1) \in (1, \infty)$ such that
\[
p_0 < p_0 \mu < \min \{ p, q\}
\]
and the following holds. If $u \in \hh_{p,q,w,0}^{\alpha, 2} ( \rr^d_T )$ is supported on  $(0,T) \times B_{\xi R_0}(x_1)$ for some $x_1\in \rr^d$ and $\xi >0 $, and satisfies
\[
\partial_t^\alpha u - a^{ij} (t,x) D_{ij} u = f \quad \text{in}\quad  \rr^d_T, \stepcounter{equation}\tag{\theequation}\label{nol}
\]
where the coefficients $a^{ij}(t,x)$'s satisfy Assumption \ref{ass2} $(\gamma_0)$, then for any $  r\in (0, \infty)$, $\kappa \in (0, 1/4)$, and $(t_0, x_0) \in (0,T] \times \rr^d$, we have
\begin{align*}
& ( |D^2 u - (D^2 u )_{(t_0-(\kappa r)^{2/\alpha},t_0) \times B_{\delta \kappa r}(x_0)} | )_{
(t_0-(\kappa r)^{2/\alpha},t_0) \times B_{\delta \kappa r}(x_0)
} \\
&\leq N \xi^{d/q_0} \kappa^{-d/q_0} ( | D^2 u |^{p_0} )^{1/p_0}_{(t_0-(\kappa r)^{2/\alpha},t_0) \times B_{\delta \kappa r}(x_0)} \\
&\quad + N \kappa^\sigma \sum_{k=0}^\infty {2^{-k\alpha}}( |D^2 u |^{p_0} )^{1 / p_0}_{(t_0-2^{k}(r/2)^{2/\alpha}r,t_0)\times B_{\delta^{-1}r/2}(x_0)}\\
&\quad + N \kappa^{- (d+2 /\alpha)/p_0} \gamma_0^{1/(\nu p_0)}
\sum_{k=0}^\infty  2^{-\alpha k} 2^{k/(\nu p_0)}(|D^2 u |^{\mu p_0})^{1/(\mu p_0)}_{
( t_0 - 2^k r^{2/\alpha}, t_0 )
\times B_{  \delta^{-1} \sqrt{d} r}(x_0) }\\
&\quad+ N \kappa^{- (d+2 /\alpha)/p_0}
 \sum_{k=0}^\infty  2^{-\alpha k} (| f  |^{p_0})^{1/p_0}_{
( t_0 - {2^k} r^{2/\alpha}, t_0 ) \times B_{  \delta^{-1} \sqrt{d} r}(x_0) },  \stepcounter{equation}\tag{\theequation}\label{meanoc}
\end{align*}
where $\nu = \mu/(\mu - 1)$, $q_0 = p_0/(p_0 - 1)$, and $N = N(d, \delta, \alpha, p, q, K_1)$. The functions $u$ and $f$ are defined to be zero whenever $t \leq 0$.
\end{proposition}

\begin{proof}
For the given $w_1 \in A_p(\rr, dt)$ and $w_2 \in A_q (\rr^d, dx)$, using reverse H\"older's inequality for $A_p$ weights, we pick
$
\sigma_1 = \sigma_1 (d, p, K_1)$ and $
\sigma_2 = \sigma_2 (d, q, K_1)
$
such that $p - \sigma_1 > 1$, $q - \sigma_2 > 1$, and
\[
w_1 \in A_{p- \sigma_1} (\rr, dt),
\quad w_2 \in A_{q - \sigma_2} (\rr^d, dx).
\]
Take $p_0, \mu \in (1, \infty)$ so that
\[
p_0 \mu = \min \Bigg\{
\frac{p}{p- \sigma_1}, \frac{q}{q - \sigma_2}
\Bigg\} > 1.
\]
Note that
\[
w_1 \in A_{p - \sigma_1}
\subset A_{\frac{p}{p_0 \mu}}
\subset A_{\frac{p}{p_0}} (\rr, dt)
\]
and
\[
w_2 \in A_{q - \sigma_2}
\subset A_{\frac{q}{p_0 \mu}}
\subset A_{\frac{q}{p_0}} (\rr^d, dx).
\]
From these inclusions and the fact that $u \in \hh_{p,q,w,0}^{\alpha, 2} ( \rr^d_T )$, it follows that
$u \in \hh_{p_0\mu,0,\mathrm{loc}}^{\alpha, 2} ( \rr^d_T )$. See the proof of \cite[Lemma 5.10]{Ap}.

We first consider the case when $r \ge {R_1\delta}/\sqrt{d}$, where $R_1 = 2^{-\alpha/{2}}R_0$.
By H\"older's inequality,
\begin{align*}
& (|D^2 u - (D^2 u )_{(t_0-(\kappa r)^{2/\alpha},t_0) \times B_{\delta \kappa r}(x_0)} | )_{
(t_0-(\kappa r)^{2/\alpha},t_0) \times B_{\delta \kappa r}(x_0)
} \\
&\le 2 (|D^2 u| )_{(t_0-(\kappa r)^{2/\alpha},t_0) \times B_{\delta \kappa r}(x_0)}\\
&=  2 (|D^2 u| \chi_{(0,T) \times B_{\xi R_0}(x_1)})_{(t_0-(\kappa r)^{2/\alpha},t_0) \times B_{\delta \kappa r}(x_0)}\\
&\le 2 ( | D^2 u |^{p_0} )^{1/p_0}_{(t_0-(\kappa r)^{2/\alpha},t_0) \times B_{\delta \kappa r}(x_0)}  \Big(\frac{|B_{\xi R_0}|}{|B_{\delta \kappa r}|}\Big)^{1/q_0}\\
&\le N(d,\delta,p,q)\xi^{d/q_0} \kappa^{-d/q_0} ( | D^2 u |^{p_0} )^{1/p_0}_{(t_0-(\kappa r)^{2/\alpha},t_0) \times B_{\delta \kappa r}(x_0)} .\stepcounter{equation}\tag{\theequation}\label{nn3}
\end{align*}

Then, we consider the case when $r < {R_1\delta}/ \sqrt{d}$. Let
\[
\overline{a}^{ij}
= \fint_{Q_{\delta^{-1} 2 r} (t_0, x_0) }  a^{ij} \, dt dy
\]
and
\begin{align*}
 \partial_t^\alpha u - \overline{a}^{ij} D_{ij} u = f + ( a^{ij}- \overline{a}^{ij}) D_{ij} u:=\widetilde{f}.
\end{align*}
By Proposition \ref{3.9},
\begin{align*}
& ( |D^2 u - (D^2 u )_{(t_0-(\kappa r)^{2/\alpha},t_0) \times B_{\delta \kappa r}(x_0)} | )_{
(t_0-(\kappa r)^{2/\alpha},t_0) \times B_{\delta \kappa r}(x_0)
} \\
& \leq N \kappa^\sigma   \sum_{k=0}^\infty {2^{-k\alpha}}( |D^2 u |^{p_0} )^{1 / p_0}_{(t_0-2^{k}(r/2)^{2/\alpha}r,t_0)\times B_{\delta^{-1}r/2}(x_0)} \\
&\quad+ N \kappa^{- (d+2 /\alpha)/p_0}
\sum_{k=0}^\infty 2^{-\alpha k} (| \widetilde{f}  |^{p_0})^{1/p_0}_{
( t_0 - 2^k r^{2/\alpha}, t_0 )
\times B_{  \delta^{-1} \sqrt{d} r}(x_0) },
\end{align*}
where $N = N( d, \delta, \alpha, p, q, K_1)$.
By the definition of $\widetilde{f}$ {and H\"older's inequality},
\begin{align*}
& (| \widetilde{f}  |^{p_0})^{1/p_0}_{
( t_0 - 2^k r^{2/\alpha}, t_0 )
\times B_{  \delta^{-1} \sqrt{d} r}(x_0) } \\&\leq (| f  |^{p_0})^{1/p_0}_{
( t_0 - 2^k r^{2/\alpha}, t_0 )
\times B_{  \delta^{-1} \sqrt{d} r}(x_0) } +
(|\overline{a}^{ij} - a^{ij} |^{\nu p_0}
)^{1/ (\nu p_0)}_{
( t_0 - 2^k r^{2/\alpha}, t_0 )
\times B_{  \delta^{-1} \sqrt{d} r}(x_0) } \\
&\quad \quad\quad\cdot (|D^2 u |^{\mu p_0})^{1/(\mu p_0)}_{
( t_0 - 2^k r^{2/\alpha}, t_0 )
\times B_{  \delta^{-1} \sqrt{d} r}(x_0) }. \stepcounter{equation}\tag{\theequation}\label{nn1}
\end{align*}
Moreover, by Lemma \ref{2.1},
\begin{align*}
(|\overline{a}^{ij} - a^{ij} | )_{
( t_0 - 2^k r^{2/\alpha}, t_0 )
\times B_{  \delta^{-1} \sqrt{d} r}(x_0) } \leq N(\alpha, d,\delta) \gamma_0 2^k.\stepcounter{equation}\tag{\theequation}\label{nn2}
\end{align*}
Therefore, by \eqref{nn1} and \eqref{nn2},
\begin{align*}
& \sum_{k=0}^\infty  2^{-\alpha k} (| \widetilde{f}  |^{p_0})^{1/p_0}_{
( t_0 - 2^k r^{2/\alpha}, t_0 )
\times B_{  \delta^{-1} \sqrt{d} r}(x_0) }\\
&\leq \sum_{k=0}^\infty  2^{-\alpha k} (| f  |^{p_0})^{1/p_0}_{
( t_0 - 2^k r^{2/\alpha}, t_0 ) \times B_{  \delta^{-1} \sqrt{d} r}(x_0) }  \\
&\quad+ N\gamma_0^{1/(\nu p_0)}\sum_{k=0}^\infty  2^{-\alpha k} 2^{k/(\nu p_0)}(|D^2 u |^{\mu p_0})^{1/(\mu p_0)}_{
( t_0 - 2^k r^{2/\alpha}, t_0 )
\times B_{  \delta^{-1} \sqrt{d} r}(x_0) }.
\stepcounter{equation}\tag{\theequation}\label{nn4}
\end{align*}
Therefore, by combining both cases \eqref{nn3} and \eqref{nn4}, we arrive at the estimate (\ref{meanoc}). The proposition is proved.
\end{proof}

Next, we introduce the dyadic cubes as follows:
for each $n \in \zz$, pick $k(n)\in \zz$ such that
$$ k(n) \le \frac{2n}{\alpha} < k(n) +1,$$ and let
\begin{align*}
     Q_{\vec{i}}^n = \big[ \frac{i_0}{2^{k(n)}} + T, \frac{i_0+1}{2^{k(n)}} + T \big) \times \big[ \frac{i_1}{2^n}, \frac{i_1+1}{2^n} \big) \times \cdots \times \big[ \frac{i_d}{2^n}, \frac{i_d+1}{2^n} \big) \stepcounter{equation}\tag{\theequation}\label{cubes}
\end{align*} and
$$
\cc_n := \big\{ Q_{\vec{i}}^n = Q_{(i_0,\ldots,i_d)}^n : \vec{i} = (i_0,\ldots,i_d)\in \zz^{d+1}{, i_0\le -1}\big\}.
$$
Furthermore, denote the dyadic sharp function of $g$ by
$$  g^{\#}_{dy}(t,x) = \sup_{n< \infty} \fint_{Q_{\vec{i}}^n \ni (t,x) } |g(s,y) - g_{|n}(t,x)| \, dy \, ds,$$
where
$$g_{|n}(t,x) = \fint_{Q_{\vec{i}}^n} g(s,y) \, dy \, ds \quad \text{for}\quad  (t,x) \in Q_{\vec{i}}^n.
$$

\begin{proposition}\label{4.2}
Let $\alpha \in (1,2)$, $T \in (0, \infty)$, $p,q \in (1, \infty)$, $K_1 \in [1, \infty)$, and $[w]_{p,q} \le K_1$.
There exist $\gamma_0 = \gamma_0 (d, \delta, \alpha, p, q, K_1) > 0$  and $\xi(d, \delta, \alpha, p, q, K_1) >0$ such that under Assumption \ref{ass2} $(\gamma_0)$, for any $u \in \hh_{p,q,w,0}^{\alpha, 2} ( \rr^d_T )$ supported on $(0,T) \times B_{\xi R_0}(x_1)$ for some $x_1\in \rr^d$, satisfying (\ref{nol}), we have
\[
\norm{\partial_t^\alpha u}_{L_{p,q,w} }
+ \norm{D^2 u}_{L_{p,q,w} }
\leq
N \norm{f}_{L_{p,q,w} \stepcounter{equation}\tag{\theequation}\label{nn7} },
\]
where $\norm{\cdot}_{p,q,w} = \norm{\cdot}_{L_{p,q,w} ( \rr^d_T )} $ and $N = N(d, \delta, \alpha, p, q, K_1 )$.
\end{proposition}
\begin{proof}
For any $(t_0,x_0)\in \rr^d_T$, $n\in \zz$, and $ Q_{\vec{i}}^n$ containing $(t_0,x_0)$, there exist $r =r(n, \alpha, d,\delta)>0$ and $t_1 = \min(T,t_0+r^{2/\alpha}/2)\in(-\infty,T]$ such that
$$
Q_{\vec{i}}^n \subset (t_1- r^{2/\alpha},t_1) \times B_{\delta r}(x_0) \quad \text{and}\quad
|(t_1- r^{2/\alpha},t_1) \times B_{\delta r}(x_0)|\le N(\alpha,d,\delta)|Q_{\vec{i}}^n|.
$$
Thus, by (\ref{meanoc}), we have
\begin{align*}
&\fint_{Q_{\vec{i}}^n \ni (t_0,x_0) } |D^2 u(s,y) - {(D^2 u)}_{|n}(t_0,x_0)| \, dy \, ds\\
&\le ( |D^2 u - (D^2 u )_{(t_1- r^{2/\alpha},t_1) \times B_{\delta r}(x_0)} | )_{
(t_1-r^{2/\alpha},t_1) \times B_{\delta  r}(x_0)
} \\
&\leq N \xi^{d/q_0} \kappa^{-d/q_0} ( | D^2 u |^{p_0} )^{1/p_0}_{(t_1- r^{2/\alpha},t_1) \times B_{\delta r}(x_0)} \\
& + N \kappa^\sigma \sum_{k=0}^\infty {2^{-k\alpha}}( |D^2 u |^{p_0} )^{1 / p_0}_{(t_1-2^{k}(r/(2\kappa))^{2/\alpha}r,t_1)\times B_{\delta^{-1}r/(2\kappa)}(x_0)}\\
& + N \kappa^{- (d+2 /\alpha)/p_0} \gamma_0^{1/(\nu p_0)}
\sum_{k=0}^\infty  2^{-\alpha k} 2^{k/(\nu p_0)}(|D^2 u |^{\mu p_0})^{1/(\mu p_0)}_{
( t_1 - 2^k (r/\kappa)^{2/\alpha}, t_1 )
\times B_{  \delta^{-1} \sqrt{d} r/\kappa}(x_0) }\\
& + N \kappa^{- (d+2 /\alpha)/p_0}
 \sum_{k=0}^\infty  2^{-\alpha k} (| f  |^{p_0})^{1/p_0}_{
( t_1 - {2^k} (r/\kappa)^{2/\alpha}, t_1 ) \times B_{  \delta^{-1} \sqrt{d} r/\kappa}(x_0) },
\end{align*}
where $N$ is independent of $n$ and $\kappa$, and for the first inequality, we used \eqref{simpl}.
Therefore,  since $\alpha>1>1/(\nu p_0)$, we conclude that
\begin{align*}
(D^2 u)^{\#}_{dy}(t_0,x_0)
&\leq N \xi^{d/q_0} \kappa^{-d/q_0} (\css \cmm | D^2 u |^{p_0} )^{1/p_0} (t_0, x_0)\\
&\quad + N \kappa^\sigma (\css \cmm | D^2 u |^{p_0} )^{1/p_0} (t_0, x_0) \\
& \quad+ N \kappa^{- (d+2 /\alpha)/p_0} \gamma_0^{1/(\nu p_0)}
 (\css \cmm | D^2 u |^{\mu p_0} )^{1/(\mu p_0)} (t_0, x_0)\\
& \quad+ N \kappa^{- (d+2 /\alpha)/p_0}
 (\css \cmm | f |^{p_0} )^{1/p_0} (t_0, x_0).
\end{align*}
Then, by the weighted sharp function theorem \cite[Corollary 2.7]{Ap} and the weighted maximal function theorem for strong maximal functions \cite[Theorem 5.2]{dong20},
\begin{align*}
\norm{D^2 u}_{p,q,w} &\le N (\xi^{d/q_0} \kappa^{-d/q_0} + \kappa^\sigma + \kappa^{- (d+2 /\alpha)/p_0} \gamma_0^{1/(\nu p_0)} ) \norm{D^2 u}_{p,q,w} \\
&\quad + N \kappa^{- (d+2 /\alpha)/p_0}\norm{f}_{p,q,w},
\end{align*}
where $N = N (d,\delta,\alpha,p,q,K_1)$. Therefore, by first taking a sufficiently small $\kappa < 1/4$, and then taking sufficiently small $\xi,\gamma_0 > 0 $ such that
$$
N (\xi^{d/q_0} \kappa^{-d/q_0} + \kappa^\sigma + \kappa^{- (d+2 /\alpha)/p_0} \gamma_0^{1/(\nu p_0)} ) <1/2,
$$
we arrive at
\[\norm{D^2 u}_{L_{p,q,w} }
\leq
N \norm{f}_{L_{p,q,w} }. \stepcounter{equation}\tag{\theequation}\label{nn6}\]
Finally, the equation \eqref{nol} together with \eqref{nn6} yields \eqref{nn7}. The proposition is proved.
\end{proof}

Next, we get rid of the restriction on the function in Proposition \ref{4.2} by a partition of unity argument.
\begin{corollary}\label{withl}
Let $\alpha \in (1,2)$, $T \in (0, \infty)$, $p,q \in (1, \infty)$, $K_1 \in [1, \infty)$, and $[w]_{p,q} \le K_1$.
There exists $\gamma_0 = \gamma_0 (d, \delta, \alpha, p, q, K_1) > 0$ such that under Assumption \ref{ass2} $(\gamma_0)$, for any $u \in \hh_{p,q,w,0}^{\alpha, 2} ( \rr^d_T )$ satisfying (\ref{main estimate}) in $\rr^d_T$, we have
\[
\norm{\partial_t^\alpha u}_{p,q,w}
+ \norm{D^2 u}_{p,q,w}
\leq
N_0 \norm{f}_{p,q,w} + {N_0R_0^{-2}}\norm{u}_{p,q,w},\stepcounter{equation}\tag{\theequation}\label{wlower}
\]
where $N_0 = N_0(d, \delta, \alpha, p, q, K_1 )$ and $\norm{\cdot}_{p,q,w} = \norm{\cdot}_{L_{p,q,w} ( \rr^d_T )} $.
\end{corollary}

\begin{proof}
{\em Case 1.} We first consider the case when $p=q$ and $b^i=c=0$. We {take $\gamma_0$ and $\xi$ from Proposition \ref{4.2}, }and use a partition of unity argument in the spatial variables, and to this end we find $\{x_k\} \subset \rr^d$,
$\zeta^k \geq 0$, $\zeta^k \in C_0^\infty (\rr^d)$, $\supp (\zeta^k) \subset B_{\xi R_0/\sqrt{2}} (x_k)$ and
$$
1 \le \sum_{k=1}^\infty |\zeta^k|^p \le N(p), \quad \sum_{k=1}^\infty |D_x \zeta^k |^p\le NR_0^{-p}, \quad \sum_{k=1}^\infty |D^2_x \zeta^k |^p \leq NR_0^{-2p}
$$
for some $N = N (d,\delta, \alpha, p, K_1)$.
For $u_k(t,x) := u(t,x) \zeta^k (x)$, we have
\begin{align*}
\partial_t^\alpha u_k - a^{ij}D_{ij} u_k  &= (\partial_t^\alpha u) \zeta^k - (a^{ij}D_{ij}u) \zeta^k - 2a^{ij}D_i u D_j \zeta^k - a^{ij} u D_{ij}\zeta^k \\
&= f \zeta^k -2a^{ij}D_i u D_j \zeta^k - a^{ij} u D_{ij}\zeta^k.
\end{align*}
Therefore, by Proposition \ref{4.2},
\begin{align*}
& \norm{\partial_t^\alpha u_k}_{p,w}
+ \norm{D^2 u_k}_{p,w}\\
&\le N
\norm{f \zeta^k}_{p,w} + N\sum_{i,j=1}^d\norm{D_i \zeta^k D_j u }_{p,w}
+ N\sum_{i,j=1}^d \norm{u D_{ij}\zeta^k}_{p,w},
\end{align*}
where $N = N(d, \delta, \alpha, p,K_1)$. By raising to the $p$-th power and summing in $k$, we get
\begin{align*}
\norm{\partial_t ^\alpha u}^p_{p,w} + \norm{D^2u}^p_{p,w} \leq
N_0 \norm{f}^p_{p,w} + {N_0R_0^{-2p}} \norm{u}^p_{p,w}
+ {N_0R_0^{-p}} \norm{Du}^p_{p,w},
\end{align*}
where $N_0 = N_0 (d, \delta, \alpha, p, K_1) $. Therefore, (\ref{wlower}) follows from the weighted interpolation inequality {for} the spatial variables (see for instance, \cite[Lemma {3.8} (iii)]{K&D}).

{\em Case 2.} Then, we consider the case when  $p=q$ without assuming $b^i, c = 0$. In this case, we have
\[
\partial_t^\alpha u
- a^{ij} D_{ij} u
= f + b^i D_i u + cu.
\]
Thus, the result follows from Case 1 and the interpolation inequality {for} the spatial variables.

{\em Case 3.} Finally, for the general case when $p\neq q$, we use the extrapolation theorem in \cite[Theorem 2.5]{Ap}. See also \cite[Theorem 1.4]{extrapolation}. The Corollary is proved.
\end{proof}

To prove \eqref{99999}, we need to get rid of the $u$ term on the right-hand side of \eqref{wlower}. We start with the following Lemma.
\begin{lemma}\label{4.4}
Let $\alpha \in [1,2)$, $T \in (0, \infty)$, $p,q \in (1, \infty)$, $K_1 \in [1, \infty)$, and $[w]_{p,q} \le K_1$.
For any $u \in \hh_{p,q,w,0}^{\alpha, 2} ( \rr^d_T )$, we have
$$\norm{u}_{p,q,w} \le N T^\alpha \norm{\partial_t^\alpha u}_{p,q,w},  $$
where $N = N(\alpha,p,q,K_1)$ and $\norm{\cdot}_{p,q,w} = \norm{\cdot}_{L_{p,q,w} ( \rr^d_T )} $.
\end{lemma}

\begin{proof} Note that
$u = I^\alpha \partial_t^\alpha u.$
Therefore, the result follows {from} \cite[Lemma 5.5]{dong20}.
\end{proof}

We are ready to prove Theorem \ref{main}. In the following proof, we use the notation $\norm{\cdot}_{p,q,w} = \norm{\cdot}_{L_{p,q,w} ( \rr^d_T )}$ and $\norm{\cdot}_{p,q,w,(\tau_1,\tau_2)}= \norm{\cdot}_{L_{p,q,w} ( (\tau_1,\tau_2) \times \rr^d )}$.
\begin{proof}[Proof of Theorem \ref{main} ]
We first prove the a priori estimate \eqref{99999}. By Corollary \ref{withl}, it suffices to prove
\[
\norm{u}_{p,q,w} \leq N \norm{f}_{p,q,w}. \stepcounter{equation}\tag{\theequation}\label{absorb}
\]
By extending $u$ and $f$ to be zero for $t <0$, we have
\[
\partial_t^\alpha u - a^{ij} D_{ij} u- b^i D_i u - cu = f \quad \text{in}\quad  (S,T) \times \rr^d
\]
for any $S \leq 0$.

Take a positive integer $m$ to be specified below. For $j = -1,0,\ldots, m-1$, set $s_j = {jT}/{m}$, and for $j\ge 0$, let $\eta_j=\eta_j(t)$ be smooth functions such that
\begin{align*}
\eta_j(t) &=
\begin{cases}
1 \quad \quad \text{when}\quad  t \geq s_j,\\
0 \quad \quad \text{when}\quad  t \leq s_{j-1},
\end{cases}
\quad |\eta_j'| \leq \frac{2m}{T},
\quad |\eta_j''| \leq 4 \frac{m^2}{T^2}. \stepcounter{equation}\tag{\theequation}\label{dd1}
\end{align*}
It follows {from} Lemma \ref{3.1}, $u \eta_j \in \hh_{p,q,w,0}^{\alpha, 2}
( (s_{j-1}, s_{j+1} ) \times \rr^d  )$. Furthermore, similar to the decomposition right below \eqref{01045}, we have
\begin{align*}
\partial_t^\alpha (u \eta_j)
&- a^{ij} D_{ij} (u \eta_j)
- b^i D_i (u \eta_j)
- c(u \eta_j) = f\eta_j + h_j
\end{align*}
in $ (s_{j-1}, s_{j+1} ) \times \rr^d $,
where for $t \in (s_{j-1}, s_{j+1} )$,
\begin{align*}
h_j (t,x)
&= \frac{\alpha (\alpha -1)}
{ \Gamma (2 -\alpha)}
\int_{0}^t
(t-s)^{-\alpha - 1}
[ \eta_j (s) - \eta_j (t) - \eta_j'(t)(s-t)]
u(s,x) \, ds \\
&\quad+ \frac{\alpha}{ \Gamma (2 - \alpha)}
\eta_j'(t)
\partial_t\int_{0}^t
(t-s)^{-\alpha + 1}
u(s,x) \, ds \\
&= \frac{\alpha (\alpha -1)}
{ \Gamma (2 - \alpha)}
\int_{{0}}^t
(t-s)^{-\alpha - 1}
[ \eta_j (s) - \eta_j (t) - \eta_j'(t)(s-t)]
u(s,x)  \chi_{s \leq s_j} \, ds \\
&\quad+ \frac{\alpha}{ \Gamma (2 - \alpha)}
\eta_j'(t)
\partial_t\int_{{0}}^t
(t-s)^{-\alpha + 1}
u(s,x) \, ds.
\end{align*}
For $j=0$, since
$\eta_0' (t) = 0$ for $t\in (s_0, s_1) $, we have $h_0 = 0$.
For $j = 1, 2,\ldots, m-1$, similar to the computation in Proposition \ref{3.3}, we have
\begin{align*}
h_j(t,x)
&= N_1(\alpha)
\int_{-\infty}^t
(t-s)^{-\alpha - 1}
[ \eta_j (s) - \eta_j (t) - \eta_j'(t)(s-t)]
u(s,x)  \chi_{s \leq s_j} \, ds \\
&\quad+ N_1(\alpha)\frac{\eta_j' (t)}
{ \eta_{j-1} (t) }
\int_{-\infty}^t
(t-s)^{-\alpha }
[  \eta_{j-1} (s) - \eta_{j-1}(t)] u(s,x) \chi_{s \leq s_j}  \, ds \\
&\quad+ N_2(\alpha)\frac{\eta_j'(t)}{\eta_{j-1}(t)}
\partial_t^\beta (\eta_{j-1} u )(t,x),
\end{align*}
where $\beta = \alpha -1$.
Thus, by the Minkowski inequality and Lemma \ref{4.4} with the observation that $ \partial_t \partial_t^\beta (\eta_{j-1}u) = \partial_t^\alpha (\eta_{j-1}u)$ in $(s_{j-2}, s_{j})$,
\begin{align*}
&\norm{h_j}_{p,q,w,(s_{j-1}, s_{j+1})}\\
&\leq
N_3 \norm{\eta_j''}_{L_\infty}
\norm{I^{2 - \alpha}
[|u (\cdot, x)| \chi_{\cdot\leq s_j} ] }_{p,q,w,(s_{j-1}, s_{j+1})} \\
&\quad +
N_3 \norm{\eta_j'}_{L_\infty}
\norm{\eta_{j-1}'}_{L_\infty}
\norm{I^{2 - \alpha}
[|u (\cdot, x)| \chi_{\cdot\leq s_j} ] }_{p,q,w,(s_{j-1}, s_{j+1})} \\
&\quad +
N_3 \norm{\eta_j'}_{L_\infty}
\norm{\partial_t^\beta (\eta_{j-1} u) }_{p,q,w,(s_{j-1}, s_{j})} \\
&\leq N_3 m^2 T^{-\alpha} \norm{u }_{p,q,w,(0, s_{j})} + N_3 \norm{\partial_t^{\alpha} (\eta_{j-1}u) }_{p,q,w,(s_{j-2}, s_{j})},
\stepcounter{equation}\tag{\theequation}\label{qqq1}\end{align*}
where $N_3$ is independent of $j$.
By (\ref{wlower}), for $j = 0,1,\ldots, m-1$,
\begin{align*}
&\norm{\partial_t^\alpha (u \eta_j)}_{p,q,w, (s_{j-1}, s_{j+1})}\leq
N_0 \norm{f \eta_j}_{p,q,w,(s_{j-1}, s_{j+1})} \\
&\qquad
+ {N_0R_0^{-2}} \norm{u \eta_j}_{p,q,w,(s_{j-1}, s_{j+1})}
+ N_0 \norm{h_j}_{p,q,w,(s_{j-1}, s_{j+1})},
\end{align*}
where $N_0 = N_0(d, \delta, \alpha, p, q, K_1 )$. In particular, for $j=0$,
\begin{align*}
\norm{\partial_t^\alpha (u \eta_0)}_{p,q,w,(s_{-1}, s_{1})} &\leq
N_0 \norm{f \eta_0}_{p,q,w,(s_{-1}, s_{1})}
+ {N_0R_0^{-2}} \norm{u \eta_0}_{p,q,w,(s_{-1}, s_{1})}. \stepcounter{equation}\tag{\theequation}\label{qqq2}
\end{align*}
For $j=1$, {by \eqref{qqq1} and \eqref{qqq2},}
\begin{align*}
&\norm{\partial_t^\alpha (u \eta_1)}_{p,q,w, (s_{0}, s_{2})} \\
&\leq
N_0 \norm{f \eta_1}_{p,q,w,(s_{0}, s_{2})}
+ {N_0R_0^{-2}} \norm{u \eta_1}_{p,q,w,(s_{0}, s_{2})}
+ N_0 \norm{h_1}_{p,q,w,(s_{-1}, s_{1})} \\
&\leq
N_0 \norm{f \eta_1}_{p,q,w,(s_{0}, s_{2})}
+ {N_0R_0^{-2}} \norm{u \eta_1}_{p,q,w,(s_{0}, s_{2})}
+ N_0
\big[N_3 m^2 T^{-\alpha}  \norm{u}_{{p,q,w,}(0, s_1)}  \\
&\quad + N_3 \norm{\partial_t^\alpha ( u \eta_{0})}_{ (s_{-1}, s_1) }  \big] \\
&\leq
N_0 \norm{f \eta_1}_{p,q,w,(s_{0}, s_{2})}
+ {N_0R_0^{-2}} \norm{u \eta_1}_{p,q,w,(s_{0}, s_{2})}
+ N_0N_3 m^2 T^{-\alpha}  \norm{u}_{p,q,w,(0, s_1)} \\
& \quad + N_0^2N_3  \norm{f \eta_0}_{p,q,w,(s_{-1}, s_{1})}
+ N_0^{2}N_3{R_0^{-2}} \norm{u \eta_0}_{p,q,w,(s_{-1}, s_1 )} \\
&\leq N \norm{f}_{p,q,w,(0, s_{2})} + N \norm{u}_{p,q,w,(0, s_{1})} + N_1 \norm{u}_{p,q,w,(s_{1}, s_{2})},
\end{align*}
where $N_1$ is independent of $m$, and $N$ depends on $m$, but both of them are independent of $j$.
Repeating this process for $j = 2,3,\ldots$, we conclude
\begin{align*}
& \norm{\partial_t^\alpha (u \eta_j) }_{p,q,w,(s_{j-1}, s_{j+1})} \\
&\leq N \norm{f}_{p,q,w,(0, s_{j+1})} + N \norm{u}_{p,q,w,(0, s_{j})} + N_1 \norm{u}_{p,q,w,(s_{j}, s_{j+1})}, \stepcounter{equation}\tag{\theequation}\label{thises}
\end{align*}
where $N_1= N_1(d,\delta,\alpha, p,q,K_1,R_0)$ and $N= N(d,\delta,\alpha, p,q,K_1,R_0,T,m)$.
Using Lemma \ref{4.4} and \eqref{thises}, {we obtain}
\begin{align*}
&\norm{u}_{p,q,w,(s_{j}, s_{j+1})}
\le \norm{u\eta_j}_{p,q,w,(s_{j-1}, s_{j+1})}\\
&\le N
({T}/{m})^{\alpha} \norm{\partial _t^\alpha (u\eta_j)}_{p,q,w,(s_{j-1}, s_{j+1})}\\
&\le N\norm{f}_{p,q,w,(0, s_{j+1})}+ N \norm{u}_{p,q,w,(0, s_{j})}
 + N_1(T/m)^{\alpha}  \norm{u}_{p,q,w,(s_{j-1}, s_{j+1})},
\end{align*}
where $N = N (d,\delta,\alpha, p,q,K_1,R_0,T,m)$.
By taking a sufficiently large $m$ satisfying
$$N_1(T/m)^{\alpha} < 1/2,$$
we arrive at
$$\norm{u}_{p,q,w,(s_{j}, s_{j+1})} \le N\norm{f}_{p,q,w,(0, s_{j+1})} +
N\norm{u}_{p,q,w,(0, s_{j})},$$
where $N = N(d,\delta,\alpha, p,q,K_1,R_0,T)$.
Therefore, (\ref{absorb}) follows from induction. Thus, the a priori estimate \eqref{99999} is proved. The existence of solutions follows the solvability of a simple equation with constant coefficients presented in \cite[Theorem 2.9]{P1} or \cite[Theorem 2.8]{simple} together with the method of continuity.
\end{proof}

\section{Proofs of Theorems \ref{main2} and \ref{nondivdomain}}\label{5}
In this section, we first prove Theorem \ref{main2} for the half space by taking extensions. Then, we prove Theorem \ref{nondivdomain} for smooth domains by using S. Agmon's idea and a partition of unity argument.

\subsection{The half space case}
In this subsection, we consider the case when $\Omega = \rr^d_+$.
Let
\[
(\widetilde{\css \cmm} f ) (t,x)= \sup_{\substack{{Q_{r_1, r_2} (s,y) \cap \Omega_{T} \ni (t,x)}\\ {(s,}y)\in \Omega_{T}}}
\fint_{Q_{r_1, r_2} (s,y) \cap \Omega_{T}}
|f(r,z) |  \, dz \, dr \stepcounter{equation} \tag{\theequation}\label{SM}
\]
and $B_r^+(x_0) := B_r(x_0) \cap \Omega $. Moreover, for $w(t,x) = w_1(t) w_2(x)$, where $w_1 \in A_p(\rr)$ and $w_2 \in A_q(\Omega)$, we take the even extension of $w_2$ to the whole space by setting $\overline{w_2}(x) = w_2( |x^1|, x')$. Furthermore, we denote $\overline{w}(t,x) = w_1(t) \overline{w_2}(x)$.
\begin{remark}
Note that if $[w_2]_{A_q(\Omega)} \le K_1$, then $[\overline{w_2}]_{A_q(\rr^d)} \le 2^q K_1$. Indeed, for any $x \in \rr^d, x^1 \ge0$ and $r>0$, we have
\begin{align*}
&\biggl(
\fint_{B_r(x_0) } w(x) \, dx
\biggr) \biggl(
\fint_{B_r(x_0)} (w(x))^{\frac{-1}{q-1}} \, dx
\biggr)^{q-1} \\
&\le \biggl(
2\fint_{B_r(x_0) \cap \Omega} w(x) \, dx
\biggr) \biggl(
2\fint_{B_r(x_0) \cap \Omega} (w(x))^{\frac{-1}{q-1}} \, dx
\biggr)^{q-1}
\le 2^q K_1.
\end{align*}
The same inequality holds for  $x^1 < 0$ by symmetry.
\end{remark}
\begin{lemma}\label{ext}
For any
$u \in \mathring{\hh}^{\alpha, 2}_{p,q,w,0} (\Omega_T )$,
let $\overline{u}(t,x) := u(t, |x^1|, x') \mathrm{{sgn}} \, x^1$.
Then, $
\overline{u}
\in
 \hh^{\alpha, 2}_{p,q,\overline{w},0} (\rr^d_T )$ and
$$ \norm{u}_{\hh^{\alpha, 2}_{p,q,w} (\Omega_T)} \leq \norm{\overline{u}}_{\hh^{\alpha, 2}_{p,q,\overline{w}} (\rr^d_T)} \leq 2 \norm{u}_{\hh^{\alpha, 2}_{p,q,w} (\Omega_T)}. $$
\end{lemma}
\begin{proof}
When $w=1$, the proof follows from \cite[Lemma 8.2.1]{k}. For general $w$, the estimates follows from the same proof by using the definitions of $\overline{u}$ and $\overline{w}$.
\end{proof}

\begin{proposition}\label{5.2}
Let $\alpha \in (1, 2)$, $p_0 \in (1, 2)$, $T \in (0, \infty)$, $a^{ij}$ be constants, and  $u \in \mathring{\hh}_{p_0, 0, \mathrm{loc}}^{\alpha, 2} ( \Omega_T )$ satisfy
\[
\partial_t^\alpha u - a^{ij} D_{ij} u = f
\quad \text{in}\quad  \Omega_T.
\]
Then, for any $(t_0, x_0) \in (0,T] \times \Omega$, $r \in (0, \infty)$, and $\kappa \in (0, 1/4)$, we have
\begin{align*}
& ( |D^2 u - (D^2 u )_{(t_0-(\kappa r)^{2/\alpha},t_0) \times B^+_{\delta\kappa r}(x_0))} | )_{
(t_0-(\kappa r)^{2/\alpha},t_0) \times B^+_{\delta \kappa r}(x_0)
} \\
&\quad \leq N \kappa^\sigma  \sum_{k=0}^\infty {2^{-k\alpha}}( |D^2 u |^{p_0} )^{1 / p_0}_{(t_0-2^{k}(r/2)^{2/\alpha}r,t_0)\times B^+_{\delta^{-1}r/2}(x_0)}\\
&\quad \quad+ N \kappa^{- (d+2 /\alpha)/p_0}
\sum_{k=0}^\infty 2^{-\alpha k} (|f|^{p_0})^{1/p_0}_{
( t_0 - 2^k r^{2/\alpha}, t_0 )
\times B^+_{  \delta^{-1} \sqrt{d} r}(x_0)
}\stepcounter{equation} \tag{\theequation}\label{712},
\end{align*}
where $N = N(d, \delta, \alpha, p_0)$ and $ \sigma = \sigma (d, \alpha, p_0)$.
\end{proposition}

\begin{proof}
{\em Case 1.} We first consider the case when $a^{ij} = \delta^{ij}$.
By taking the odd extensions of $u$ and $f$ in the $x^1$ direction, denoted by $\overline{u},\overline{f}$, we see that $\overline{u}$ satisfies the equation
$$ \partial_t^\alpha \overline{u} - \Delta \overline{u} = \overline{f}.$$
Using Lemma \ref{ext} and Proposition \ref{3.9}, we have
\begin{align*}
& ( |D^2 \overline{u} - (D^2 \overline{u} )_{(t_0-(\kappa r)^{2/\alpha},t_0) \times B_{\kappa r}(x_0))} | )_{
(t_0-(\kappa r)^{2/\alpha},t_0) \times B_{ \kappa r}(x_0)
} \\
&\quad \leq N \kappa^\sigma  \sum_{k=0}^\infty {2^{-k\alpha}}( |D^2 \overline{u} |^{p_0} )^{1 / p_0}_{(t_0-2^{k}(r/2)^{2/\alpha}r,t_0)\times B_{\delta^{-1}r/2}(x_0)} \\
&\quad \quad+ N \kappa^{- (d+2 /\alpha)/p_0}
\sum_{k=0}^\infty 2^{-\alpha k} (|\overline{f}|^{p_0})^{1/p_0}_{
( t_0 - 2^k r^{2/\alpha}, t_0 )
\times B_{  \sqrt{d} r}(x_0)
}.
\end{align*}
Then, note that
\begin{align*}
& ( |D^2 u - (D^2 u )_{(t_0-(\kappa r)^{2/\alpha},t_0) \times B^+_{\kappa r}(x_0))} | )_{
(t_0-(\kappa r)^{2/\alpha},t_0) \times B^+_{ \kappa r}(x_0)
} \\
&= ( |D^2 \overline{u} - (D^2 \overline{u} )_{(t_0-(\kappa r)^{2/\alpha},t_0) \times B^+_{\kappa r}(x_0))} | )_{
(t_0-(\kappa r)^{2/\alpha},t_0) \times B^+_{ \kappa r}(x_0)
} \\
&\le 4( |D^2 \overline{u} - (D^2 \overline{u} )_{(t_0-(\kappa r)^{2/\alpha},t_0) \times B_{\kappa r}(x_0))} | )_{
(t_0-(\kappa r)^{2/\alpha},t_0) \times B_{ \kappa r}(x_0)
},
\end{align*}
and
\begin{align*}
\fint_{B_r(x_0)} |\overline{f}|
\leq \frac{1}{|B_r^+|} \int_{B_r(x_0)} |\overline{f}|
\leq  \frac{2}{|B_r^+|}  \int_{B^+_r(x_0)} |f|
\end{align*}
for any $r>0$ and $x_0 \in \Omega$.
Therefore, (\ref{712}) follows when $a^{ij}=\delta^{ij}$.

{\em Case 2.} For general $a^{ij}$'s, without loss of generality, we assume that $A = (a^{ij})$ is symmetric, and let $\widetilde{u} (t,x) = u(t, A^{1/2} x)$. Furthermore, we take an orthogonal matrix $O$ such that $O A^{-1/2}(\Omega) = \Omega$, and let $$\widehat{u}(t,x)= \widetilde{u}(t, O^{-1}x) = u(t,A^{1/2}O^{-1}x) \quad \text{for}\quad  x \in \Omega.$$
It follows that
\[
\partial_t^\alpha \widehat{u}
- \Delta \widehat{u}
=\widehat{f}
\quad \text{in}\quad  \Omega_T.
\]
For $ y_0 := OA^{-1/2} (x_0) $, we have
\[ B_{\delta r}^+ (x_0)
\subset A^{1/2}O^{-1} (B_r^+(y_0))
\subset B_{\delta^{-1}r}^+ (x_0) \]
for any $r> 0$. Therefore, (\ref{712}) for general $a^{ij}$'s follows by applying Case 1 to $\widehat{u}$.
\end{proof}

\begin{proof}[Proof of Theorem \ref{main2}]
Lemma \ref{ext} and Theorem \ref{main} yield the existence and uniqueness of solutions to the simple equation $\partial_t^\alpha u - \Delta u = f $ {in the half space}. Then, with Proposition \ref{5.2} and Remark \ref{ass2.3}, we follow the process in Section 4 with minor modifications by applying the Hardy-Littlewood and Fefferman-Stein theorems in the half space. We omit the details.
\end{proof}

\subsection{The domain case}
In this section, we denote the operator
$$
L = {\partial_t^{\alpha} -a^{ij}D_{ij} - b^iD_i -c}.
$$
Most of this subsection is devoted to {the study of} $L+\lambda$ for $\lambda$ sufficiently large.

\begin{lemma}\label{5.3}
Let $G = \rr_+^d$ or $\rr^d$, $\alpha \in (1,2)$, $T \in (0, \infty)$, $p,q \in (1, \infty)$, $ K_1 \in [1, \infty)$, and $[w]_{p,q,\Omega} \le K_1$.
There exist $\gamma_0 = \gamma_0 (d, \delta, \alpha, p, q, K_1 ) \in (0,1)$ and $\lambda_0 = \lambda_0(d, \delta, \alpha, p, q, K_1, R_0) \geq 1$ such that under Assumption \ref{ass2} $(\gamma_0)$, the following holds.
For any $u \in \mathring{\hh}_{p,q,w,0}^{\alpha, 2} ( (0,T)\times \rr^d_+ )$ or $\hh_{p,q,w,0}^{\alpha, 2} ( \rr^d_T )$ and
any $\lambda \geq \lambda_0$, we have
$$
\norm{u}_{\hh_{p,q,w}^{\alpha,2} ((0,T) \times G )}
+ \lambda \norm{u}_{L_{p,q,w}}
\leq N \norm{Lu + \lambda u}_{L_{p,q,w}},
$$
where $\norm{\cdot}_{p,q,w}
= \norm{\cdot}_{L_{p,q,w} ((0,T) \times G) }$ and $N = N(d, \delta, \alpha, p, q, K_1)$.
\end{lemma}
\begin{proof}
We use an idea originally due to S. Agmon. Denote
$$
(0,T) \times G \times \rr
=\bigl\{ (t, x, y): t \in (0,T), x \in G, y \in \rr \bigr\},$$
and let
\[
v(t,x,y) = u(t,x) \zeta(y) \cos(\sqrt{\lambda}y),
\]
where $\zeta \in C_0^\infty (\rr)$ and $ \zeta {\not\equiv} 0$.
Then, let
$$
\widehat{f}
= \partial_t^\alpha v - \sum_{i,j=1}^d a^{ij} D_{ij} v - {D_y^2 v} -\sum_{i=1}^d b^i D_i v - cv,
$$
and note that
\begin{align*}
{D_y^2 v}
= \big( \zeta'' (y) \cos(\sqrt{\lambda}y) - 2\sqrt{\lambda}\zeta' (y) \sin(\sqrt{\lambda}y) - \lambda\zeta (y)\cos(\sqrt{\lambda}y) \big) u
\stepcounter{equation}\tag{\theequation}\label{quack}.
\end{align*}
Thus,
$$ \widehat{f} = \zeta (y) \cos (\sqrt{\lambda} y) (Lu + \lambda u)+2\sqrt{\lambda}\zeta' (y) \sin(\sqrt{\lambda}y) u-\zeta'' (y) \cos(\sqrt{\lambda}y)u.
$$
By applying Corollary \ref{withl} or a version of it in the half space to $v$ in $(0,T) \times G \times \rr$ with the weight $w(t,x,y): = w(t,x)$ for all $y\in \rr$, there exists $\gamma_0 = \gamma_0 (d, \delta, \alpha, p, q, K_1 ) \in (0,1)$ such that, if $a^{ij}$'s satisfy Assumption \ref{ass2} $(\gamma_0)$, then
\begin{align*}
&\norm{D_y^2 v}_{L_{p,q,w} ((0,T) \times G \times \rr) } \leq \norm{v}_{\hh_{p,q,w}^{\alpha,2} ((0,T) \times G \times \rr)}\\
&\leq N_0 \norm{\widehat{f}}_{L_{p,q,w} ((0,T) \times G \times \rr) } + N_1 \norm{v}_{L_{p,q,w} ((0,T) \times G \times \rr) }\\
&\leq N_0 \{  \norm{\zeta (y) \cos (\sqrt{\lambda} y) (Lu + \lambda u)}_{L_{p,q,w} ((0,T) \times G \times \rr) } \\
&\quad+ \sqrt{\lambda}\norm{\zeta' (y) \sin (\sqrt{\lambda} y)u}_{L_{p,q,w} ((0,T) \times G \times \rr) } \\
&\quad+ \norm{\zeta'' (y) \cos(\sqrt{\lambda}y)u})_{L_{p,q,w} ((0,T) \times G \times \rr) } \} + N_1 \norm{v}_{L_{p,q,w} ((0,T) \times G \times \rr) }\\
&\leq N_0 \norm{Lu + \lambda u}_{p,q,w}
+ N_0(1 + \sqrt{\lambda}) \norm{u}_{p,q,w} + N_1 \norm{u}_{p,q,w}, \stepcounter{equation}\tag{\theequation}\label{9111}
\end{align*}
where $N_0 = N_0(d, \delta, \alpha, p, q, K_1)$ and $N_1 = N_1(d, \delta, \alpha, p, q, K_1, R_0)$.
Moreover, by (\ref{quack}) and the fact that $\int_{\rr} |\zeta(y) \cos(\sqrt{\lambda}y)|^p dy$ is bounded away from $0$ independently of $\lambda$,
\begin{align*}
\norm{\lambda u}_{p,q,w}
&\leq \norm{D_y^2 v}_{L_{p,q,w} ((0,T) \times G \times \rr) } + N(1 + \sqrt{\lambda}) \norm{u}_{p,q,w} \\
&\leq N_0 \norm{Lu + \lambda u}_{p,q,w}
+ N_0(1 + \sqrt{\lambda}) \norm{u}_{p,q,w} + N_1 \norm{u}_{p,q,w}. \stepcounter{equation}\tag{\theequation}\label{9112}
\end{align*}
At this point, we take a sufficiently large $\lambda_0(d, \delta, \alpha, p, q, K_1, R_0) \ge 1$ such that
$$ N_0(1 + \sqrt{\lambda_0}) + N_1< \lambda_0/2.$$
Then, by (\ref{9111}) and (\ref{9112}), for any $\lambda \ge \lambda_0$, we have
$$
\norm{\lambda u}_{p,q,w} \leq N_0 \norm{Lu + \lambda u}_{p,q,w},
$$
which implies
\begin{align*}
&\norm{u}_{\hh_{p,q,w}^{\alpha,2} ((0,T) \times G )}
+ \lambda \norm{u}_{p,q,w} \le N\norm{v}_{\hh_{p,q,w}^{\alpha,2} ((0,T) \times G \times \rr)} + \lambda \norm{u}_{p,q,w}\\
&\le N_0 \norm{Lu + \lambda u}_{p,q,w}
+ N_0(1 + \sqrt{\lambda}) \norm{u}_{p,q,w} + N_1 \norm{u}_{p,q,w} + \lambda \norm{u}_{p,q,w}\\
&\le N_0 \norm{Lu + \lambda u}_{p,q,w} +  \lambda \norm{u}_{p,q,w} \le N_0 \norm{Lu + \lambda u}_{p,q,w}.
\end{align*}
The lemma is proved.
\end{proof}
\begin{proposition}\label{5.4}
Let $\Omega $ {be a domain}, $\alpha \in (1,2)$, $T \in (0, \infty)$, $p,q \in (1, \infty)$, $K_1 \in [1, \infty)$ and $[w]_{p,q,{\Omega}} \le K_1$.
There exist $\gamma_0 = \gamma_0 (d, \delta, \alpha, p, q, K_1 )\in(0,1)$ and $\theta = \theta(d, \delta, \alpha, p, q, K_1) \in (0,{1/3})$  such that under Assumption  \ref{ass2} $(\gamma_0)$ and Assumption {\ref{ass4}} $(\theta)$, the following holds.
There exists $\lambda_1 = \lambda_1(d, \delta, \alpha, p, q, K_1, R_0,R_2) \ge 1$ such that for any $\lambda \ge \lambda_1$ and $g \in L_{p,q,w} (\Omega_T)$, there exists a unique solution $u \in \mathring{\hh}_{p,q,w,0}^{\alpha, 2} ( \Omega_T)$ to
\[
Lu + \lambda u= g  \quad \text{in}\quad  \Omega_T, \stepcounter{equation}\tag{\theequation}\label{725}
\]
and $u$ satisfies
$$\norm{u}_{\hh_{p,q,w}^{\alpha,2} (\Omega_T )} \le N\norm{g}_{L_{p,q,w}(\Omega_T)}, $$
{where $N = N(d, \delta, \alpha, p, q, K_1, R_0,R_2)$.}
Moreover, for any constant $\lambda \ge 0$ and $u \in \mathring{\hh}_{p,q,w,0}^{\alpha, 2} ( \Omega_T)$ satisfying (\ref{725}), we have
\begin{align*}
\norm{u}_{\hh_{p,q,w}^{\alpha,2} (\Omega_T )} \le N(\norm{Lu + \lambda u}_{L_{p,q,w}(\Omega_T)} +
\norm{u}_{L_{p,q,w}(\Omega_T)}), \stepcounter{equation}\tag{\theequation}\label{729}
\end{align*}
where $N = N(d, \delta, \alpha, p, q, K_1, R_0,R_2)$. In particular when $\lambda = 0$, we have
\begin{align*}
\norm{u}_{\hh_{p,q,w}^{\alpha,2} (\Omega_T )} \le N(\norm{Lu}_{L_{p,q,w}(\Omega_T)} +
\norm{u}_{L_{p,q,w}(\Omega_T)}). \stepcounter{equation}\tag{\theequation}\label{gg}
\end{align*}
\end{proposition}
\begin{proof}
The proof is based on a partition of unity argument in the spatial variables together with the results for equations in the whole space and the half space. When $p=q$, we construct a weak solution to (\ref{725}) in $\mathring{\hh}_{p,w,0}^{0, 1} ( \Omega_T)$, and prove that the weak solution is a strong solution in  $\mathring{\hh}_{p,q,w,0}^{\alpha, 2} ( \Omega_T)$ when  $\lambda$ is sufficiently large.
See Lemma \ref{A3}-\ref{A5}  for details.

When $p \neq q$, we use the extrapolation theorem in \cite[Theorem 2.5]{Ap}.
\end{proof}

\begin{proof}[Proof of Theorem \ref{nondivdomain}]
By Proposition \ref{5.4} together with the method of continuity, it suffices to prove the a priori estimate \eqref{jj1}. {Thus,} it remains to get rid of the $u$ term on the right-hand side of (\ref{gg}). We follow the proof of (\ref{4.4}) to conclude $$\norm{u}_{L_{p,q,w}(\Omega_T)} \le N \norm{Lu}_{L_{p,q,w}(\Omega_T)},$$
{which leads to}
$$\norm{u}_{\hh_{p,q,w}^{\alpha,2} (\Omega_T )} \le N \norm{Lu}_{L_{p,q,w}(\Omega_T)},$$
{where $N = N(d,\delta,\alpha, p,q,K_1,R_0,R_2,T)$.}
The theorem is proved.
\end{proof}

\section{Equations in divergence form}\label{6}
In this section, we are going to prove Theorem{s} \ref{6main}, \ref{main4}, and \ref{divdomain} using the results for non-divergence form equations.

We pick $\zeta \in C_0^\infty(\rr^d)$ satisfy{ing} $\supp(\zeta)\subset B_1$ and $\int_{\rr^d}\zeta =1$, and let $\zeta^\ep(\cdot) = \ep^{-d}\zeta(\cdot/\ep)$. Furthermore, we denote  $g = (g^1,\ldots, g^d)$, where $g^i$ are real valued functions for all $i$, and denote the mollifier of $u = u(t,x)$ in the spatial variables by
$$u^\ep(t,x) = u \ast_x \zeta^\ep = \int_{\rr^d} u (t,y) \zeta^\ep (x-y) \, dy. $$

\subsection{The whole space case} In this subsection, we prove Theorem \ref{6main} by following a similar procedure in Section \ref{4}, where we first consider the case when the leading coefficients are constants and then the general case using a partition of unity argument. Moreover, we denote $\norm{\,\cdot\,}_{p,q,w} = \norm{\,\cdot\,}_{L_{p,q,w} ( \rr^d_T )} $.
\begin{proposition}\label{8.29}
Let $\alpha \in (1,2)$, $T \in (0, \infty)$, $p,q \in (1, \infty)$, $K_1 \in [1, \infty)$, and $[w]_{p,q} \le K_1$.
For any $g_i \in L_{p,q,w} ( \rr^d_T )$, there exists a unique $u \in \chh_{p,q,w,0}^{\alpha, 1} ( \rr^d_T )$ to
\begin{align*}
\partial_t^\alpha u - \Delta u = D_i g_i \quad \text{in}\quad  \rr^d_T  \stepcounter{equation}\tag{\theequation}\label{68.12},
\end{align*}
and $u$ satisfies
\begin{align*}
    \norm{\partial_t^\alpha u}_{\hh^{-1}_{p,q,w}(\rr^d_T)}
+ \norm{u}_{p,q,w}
+ \norm{Du}_{p,q,w}
\leq N \norm{g}_{p,q,w}, \stepcounter{equation}\tag{\theequation}\label{68.13}
\end{align*}
where $N = N (d,\delta,\alpha,p,q,K_1,T)$.
\end{proposition}
\begin{proof}
By Theorem \ref{main}, for $i = 1, \ldots,d$, there exist $v_i \in  \hh_{p,q,w,0}^{\alpha, 2} ( \rr^d_T )$ satisfying $$ \partial_t^\alpha v_i - \Delta v_i = g_i,$$
and $$\norm{Dv_i}_{p,q,w} + \norm{D^2v_i}_{p,q,w} \le N \norm{g_i}_{p,q,w}, $$
where $N = N (d,\delta,\alpha,p,q,K_1,T)$.
It follows that $$u := \sum_{i=1}^d D_iv_i \in \chh_{p,q,w,0}^{\alpha, 1} ( \rr^d_T )$$ is a solution of (\ref{68.12}). Moreover,
$$\norm{u}_{p,q,w} + \norm{Du}_{p,q,w} \le  \sum_{i=1}^d \norm{Dv_i}_{p,q,w} +  \sum_{i=1}^d  \norm{D^2 v_i}_{p,q,w} \le N\norm{g}_{p,q,w},$$
which together with the equation \eqref{68.12} implies (\ref{68.13}).

To prove the uniqueness of the solution, we consider $u \in \chh_{p,q,w,0}^{\alpha, 1} (\rr^d_T )$
satisfying
$$\partial_t^\alpha u - \Delta u = 0.$$
It suffices to prove $u=0$. We have $u^\ep =  u \ast_x \zeta^\ep \in \hh_{p,q,w,0}^{\alpha, 2}( \rr^d_T) $ and
\[\partial_t^\alpha u^\ep - \Delta u^\ep = 0.\stepcounter{equation}\tag{\theequation}\label{jj2}\]
Indeed, by the Fubini theorem,
\begin{align*}
I^{2 - \alpha} u^\ep (t,x) &= \frac{1}{\Gamma (2 -\alpha) } \int_0^t (t-s)^{1- \alpha}
u^\ep (s,x) \, ds = \int_{\rr^d} I^{2- \alpha} u (t,y) \zeta^{\ep} (x-y) \, dy.
\end{align*}
Thus, for any $\phi \in C_0^\infty ([0,T) \times \rr^d)$ and
$ \widetilde{\phi} (t, x) := \int_{\rr^d} \phi (t,y) \zeta^\ep (y - x) \, dy$, we have
\begin{align*}
&\int_0^T\int_{\rr^d}I^{2-\alpha} u^\ep (t,x) \phi_{tt} (t,x) + \int_0^T\int_{\rr^d}D_i u^\ep D_i \phi \\
&= \int_0^T\int_{\rr^d}I^{2-\alpha} u(t,y) \left( \int_{\rr^d} \phi_{tt} \zeta^\ep (x-y) \, dx \right) \, dy dt \\
&\quad \quad + \int_0^T\int_{\rr^d}\int_{\rr^d} D_i u(t,y) \zeta^\ep (x-y) \, dy D_i \phi(t,x) \, dx dt \\
&= \int_0^T\int_{\rr^d}I^{2-\alpha} u(t,y) \left( \int_{\rr^d} \phi_{tt} \zeta^\ep (x-y) \, dx \right) \, dy dt\\
&\quad \quad + \int_0^T\int_{\rr^d}D_i u(t,y) \int_{\rr^d} D_i \phi(t,x) \zeta^\ep (x-y) \, dx dy dt \\
&= \int_0^T\int_{\rr^d}I^{2-\alpha} u \widetilde{\phi}_{tt} + \int_0^T\int_{\rr^d}D_i u D_i \widetilde{\phi}= 0.
\end{align*}
We arrive at \eqref{jj2}. Therefore, $u^\ep=0$ by \eqref{jj2} and Theorem \ref{main}. Finally, by sending $\ep \to 0$, we arrive at $u=0$. The proposition is proved.
\end{proof}

The mean oscillation estimate follows immediately from the construction of the solution in Proposition \ref{8.29}.
\begin{proposition}\label{63.9}
Let $\alpha \in (1,2)$, $p_0 \in (1, 2)$, $T \in (0, \infty)$, $a^{ij}$ be constant and $u \in \chh_{p_0, 0, \mathrm{loc}}^{\alpha, 1} ( \rr^d_T )$ satisfy
\[
 \partial_t^\alpha u - D_i(a^{ij}D_j u) = D_i g_i
\quad \text{in}\quad  \rr^d_T.
\]
Then, for any $(t_0, x_0) \in (0,T] \times \rr^d,  r \in (0, \infty)$, and $\kappa \in (0, 1/4)$, we have
\begin{align*}
& ( |D u - (D u )_{(t_0-(\kappa r)^{2/\alpha},t_0) \times B_{\delta\kappa r}(x_0))} | )_{
(t_0-(\kappa r)^{2/\alpha},t_0) \times B_{\delta \kappa r}(x_0)
} \\
&\quad \leq N \kappa^\sigma  \sum_{k=0}^\infty {2^{-k\alpha}}( |D u |^{p_0} )^{1 / p_0}_{(t_0-2^{k}(r/2)^{2/\alpha}r,t_0)\times B_{\delta^{-1}r/2}(x_0)}\\
&\quad \quad+ N \kappa^{- (d+2 /\alpha)/p_0}
\sum_{k=0}^\infty 2^{-\alpha k} (|g|^{p_0})^{1/p_0}_{
( t_0 - 2^k r^{2/\alpha}, t_0 )
\times B_{  \delta^{-1} \sqrt{d} r}(x_0),
}
\end{align*}
where $N = N(d, \delta, \alpha, p_0)$ and $\sigma = \sigma (d, \alpha, p_0)$.
\end{proposition}
\begin{proof}
When $a^{ij} = \delta^{ij}$, the result follows from the uniqueness of solution in Propositions \ref{8.29} and \ref{3.9}. For general $a^{ij}$'s, we make a change of variables {as in the proof of Proposition \ref{3.9}}.
\end{proof}

\begin{proposition}\label{12291}
Let $\alpha \in (1,2)$, $T \in (0, \infty)$, $p,q \in (1, \infty)$, $K_1 \in [1, \infty)$, and $[w]_{p,q} \le K_1$.
There exist $p_0 = p_0 (d, p, q, K_1) \in (1, 2)$ and $\mu = \mu (d, p, q, K_1) \in (1, \infty)$ such that
\[
p_0 < p_0 \mu < \min \{ p, q\},
\]
and the following holds.
If $u \in \chh_{p,q,w,0}^{\alpha, 1} ( \rr^d_T )$ is supported on  $(0,T) \times B_{\xi R_0}(x_1)$ for some $x_1\in \rr^d$, $\xi >0 $, and satisfies
\[
\partial_t^\alpha u - D_i(a^{ij} (t,x) D_j u) = D_i g_i \quad \text{in}\quad  \rr^d_T,\stepcounter{equation}\tag{\theequation}\label{6nol}
\]
where the coefficients $a^{ij}$ satisfy Assumption \ref{ass2} $(\gamma_0)$, then for any $  r\in (0, \infty)$, $\kappa \in (0, 1/4)$, and $(t_0, x_0) \in (0,T] \times \rr^d$, we have
\begin{align*}
& ( |D u - (D u )_{(t_0-(\kappa r)^{2/\alpha},t_0) \times B_{\delta \kappa r}(x_0)} | )_{
(t_0-(\kappa r)^{2/\alpha},t_0) \times B_{\delta \kappa r}(x_0)
} \\
&\quad \leq N \xi^{d/q_0} \kappa^{-d/q_0} ( | D u |^{p_0} )^{1/p_0}_{(t_0-(\kappa r)^{2/\alpha},t_0) \times B_{\delta \kappa r}(x_0)} \\
&\quad \quad + N \kappa^\sigma \sum_{k=0}^\infty {2^{-k\alpha}}( |D u |^{p_0} )^{1 / p_0}_{(t_0-2^{k}(r/2)^{2/\alpha}r,t_0)\times B_{\delta^{-1}r/2}(x_0)}\\
&\quad \quad + N \kappa^{- (d+2 /\alpha)/p_0} \gamma_0^{1/(\nu p_0)}
\sum_{k=0}^\infty  2^{-\alpha k} 2^{k/(\nu p_0)}(|D u |^{\mu p_0})^{1/(\mu p_0)}_{
( t_0 - 2^k r^{2/\alpha}, t_0 )
\times B_{  \delta^{-1} \sqrt{d} r}(x_0) }\\
&\quad \quad+ N \kappa^{- (d+2 /\alpha)/p_0}
 \sum_{k=0}^\infty  2^{-\alpha k} (| g  |^{p_0})^{1/p_0}_{
( t_0 - {2^k} r^{2/\alpha}, t_0 ) \times B_{  \delta^{-1} \sqrt{d} r}(x_0) }, \stepcounter{equation}\tag{\theequation}\label{6meanoc}
\end{align*}
where $\nu = \mu/(\mu - 1)$, $q_0 = p_0/(p_0 - 1)$, and $N = N(d, \delta, \alpha, p, q, K_1)$. The functions $u$ and $g$ are defined to be zero whenever $t \leq 0$.
\end{proposition}
\begin{proof}
Note that this proposition is similar to Proposition \ref{4.1}. Indeed, there exist $p_0 = p_0 (d, p, q, K_1) \in (1, 2)$ and $\mu = \mu (d, p, q, K_1) \in (1, \infty)$ such that
\[
p_0 < p_0 \mu < \min \{ p, q\}
\] and  $u \in \chh_{p,q,w,0}^{\alpha, 1} ( \rr^d_T ) \subset \chh_{p_0\mu, 0, \mathrm{loc}}^{\alpha, 1} ( \rr^d_T )$.

If $r < {R_1\delta}/\sqrt{d}$, where $R_1 = 2^{-{\alpha}/{2}}R_0$, let
\[
\overline{a}^{ij}
= \fint_{Q_{\delta^{-1} 2 r} (t_0, x_0) }  a^{ij} \, dt dy
\]
and $\overline{g}_i = g_i + (a^{ij}- \overline{a}^{ij})D_ju$. By using Lemma \ref{2.1} and Proposition \ref{63.9}, we have
\begin{align*}
& ( |D u - (Du )_{(t_0-(\kappa r)^{2/\alpha},t_0) \times B_{\delta \kappa r}(x_0)} | )_{
(t_0-(\kappa r)^{2/\alpha},t_0) \times B_{\delta \kappa r}(x_0)
} \\
& \leq N \kappa^\sigma  {\sum_{k=0}^\infty {2^{-k\alpha}}( |D u |^{p_0} )^{1 / p_0}_{(t_0-2^{k}(r/2)^{2/\alpha}r,t_0)\times B_{\delta^{-1}r/2}(x_0)}} \\
&\quad+ N \kappa^{- (d+2 /\alpha)/p_0}
\sum_{k=0}^\infty 2^{-\alpha k} (| \overline{g}  |^{p_0})^{1/p_0}_{
( t_0 - {2^k} r^{2/\alpha}, t_0 )
\times B_{  \delta^{-1} \sqrt{d} r}(x_0) }\\
&\le  N \kappa^\sigma \sum_{k=0}^\infty {2^{-k\alpha}}( |D u |^{p_0} )^{1 / p_0}_{(t_0-2^{k}(r/2)^{2/\alpha}r,t_0)\times B_{\delta^{-1}r/2}(x_0)}\\
& \quad + N \kappa^{- (d+2 /\alpha)/p_0} \gamma_0^{1/(\nu p_0)}
\sum_{k=0}^\infty  2^{-\alpha k} 2^{k/(\nu p_0)}(|D u |^{\mu p_0})^{1/(\mu p_0)}_{
( t_0 - 2^k r^{2/\alpha}, t_0 )
\times B_{  \delta^{-1} \sqrt{d} r}(x_0) }\\
& \quad+ N \kappa^{- (d+2 /\alpha)/p_0}
 \sum_{k=0}^\infty  2^{-\alpha k} (| g  |^{p_0})^{1/p_0}_{
( t_0 - {2^k} r^{2/\alpha}, t_0 ) \times B_{  \delta^{-1} \sqrt{d} r}(x_0) },.
\end{align*}

If $r \geq R_1 \delta/\sqrt{d}$, then by H\"older's inequality,
\begin{align*}
& (|D u - (D u )_{(t_0-(\kappa r)^{2/\alpha},t_0) \times B_{\delta \kappa r}(x_0)} | )_{
(t_0-(\kappa r)^{2/\alpha},t_0) \times B_{\delta \kappa r}(x_0)
} \\
&\le 2 (|D u| )_{(t_0-(\kappa r)^{2/\alpha},t_0) \times B_{\delta \kappa r}(x_0)}\\
&=  2 (|D u| \chi_{(0,T) \times B_{\xi R_0}(x_1)})_{(t_0-(\kappa r)^{2/\alpha},t_0) \times B_{\delta \kappa r}(x_0)}\\
&\le 2 ( | D u |^{p_0} )^{1/p_0}_{(t_0-(\kappa r)^{2/\alpha},t_0) \times B_{\delta \kappa r}(x_0)}\Big(\frac{|B_{\xi R_0}|}{|B_{\delta \kappa r}|}\Big)^{1/q_0}\\
&\le N(d,\delta,p,q)\xi^{d/q_0} \kappa^{-d/q_0} ( | D u |^{p_0} )^{1/p_0}_{(t_0-(\kappa r)^{2/\alpha},t_0) \times B_{\delta \kappa r}(x_0)}.
\end{align*}
Therefore, by combining both cases, we have the estimate (\ref{6meanoc}), and the proposition is proved.
\end{proof}

\begin{proposition}\label{64.2}
Let $\alpha \in (1,2)$, $T \in (0, \infty)$, $p,q \in (1, \infty)$, $K_1 \in [1, \infty)$, and $[w]_{p,q} \le K_1$. There exist $\gamma_0 = \gamma_0 (d, \delta, \alpha, p, q, K_1) > 0$  and $\xi(d, \delta, \alpha, p, q, K_1) >0$ such that under Assumption \ref{ass2} $(\gamma_0)$, for any $u \in \chh_{p,q,w,0}^{\alpha, 1} ( \rr^d_T )$ supported on $(0,T) \times B_{\xi R_0}(x_1)$ for some $x_1\in \rr^d$ and satisfying (\ref{6nol}) in $\rr^d_T$, we have
\[ \norm{D u}_{L_{p,q,w}}
\leq
N \norm{g}_{L_{p,q,w}},
\]
where $N = N(d, \delta, \alpha, p, q, K_1 )$.
\end{proposition}
\begin{proof}
With (\ref{6meanoc}), the proof is similar to that of Proposition \ref{4.2} using the sharp function theorem with respect to the same dyadic cubes defined in \eqref{cubes} and the weight sharp and maximal function theorems.
\end{proof}

Different from the non-divergence form case, we introduce the parameter $\lambda$ in order to extend the result to
more general functions.
\begin{corollary}\label{6small}
Let $\alpha \in (1,2)$, $T \in (0, \infty)$, $p,q \in (1, \infty)$, $K_1 \in [1, \infty)$, and $[w]_{p,q} \le K_1$.
There exist $\gamma_0 = \gamma_0 (d, \delta, \alpha, p, q, K_1) > 0$ and $\xi(d, \delta, \alpha, p, q, K_1) >0$ such that under Assumption \ref{ass2} $(\gamma_0)$, the following holds. There exists $\lambda_0 = \lambda_0(d, \delta, \alpha, p, q, K_1,R_0)$ such that for any $\lambda \geq \lambda_0$ and $u \in \chh_{p,q,w,0}^{\alpha, 1} ( \rr^d_T )$ supported on $(0,T) \times B_{\xi R_0/\sqrt{2}}(x_1)$ for some $x_1\in \rr^d$ {and} satisfying
$$
\partial_t^\alpha u
 {-} D_i (a^{ij} D_j u ) {+} \lambda u = D_i g_i + f \quad \text{in}\quad  \rr^d_T,
$$
where $f,g_i \in L_{p,q,w}( \rr^d_T )$, we have
\begin{align*}
\sqrt{\lambda} \norm{ u}_{L_{p,q,w}} +
\norm{D u}_{L_{p,q,w}}  \leq
N \norm{g}_{L_{p,q,w}}
+ \frac{N}{\sqrt{\lambda} } \norm{f}_{L_{p,q,w}},
\end{align*}
where $N = N(d, \delta, \alpha, p, q, K_1 )$.
\end{corollary}
\begin{proof}
With Proposition \ref{64.2}, the proof is similar to that of Proposition \ref{5.3} using S. Agmon's idea. See \cite[Lemma 5.5]{k2} for details.
\end{proof}

\begin{corollary}\label{6par}
Let $\alpha \in (1,2)$, $T \in (0, \infty)$, $p,q \in (1, \infty)$, $K_1 \in [1, \infty)$, and $[w]_{p,q} \le K_1$.
There exists $\gamma_0 = \gamma_0 (d, \delta, \alpha, p, q, K_1) > 0$ such that under Assumption \ref{ass2} $(\gamma_0)$, the following holds. There exists $\lambda_0 = \lambda_0(d, \delta, \alpha, p, q, K_1,R_0)$ such that for any $\lambda \geq \lambda_0$ and $u \in \chh_{p,q,w,0}^{\alpha, 1} ( \rr^d_T )$ satisfying
$$
\partial_t^\alpha u
- D_i (a^{ij} D_j u + a^i u ) - b^i D_i u  - cu + \lambda u = D_i g_i + f \quad \text{in}\quad  \rr^d_T,
$$
where $f,g_i \in L_{p,q,w}( \rr^d_T )$, we have
\begin{align*}
&\sqrt{\lambda} \norm{ u}_{L_{p,q,w}} +
\norm{D u}_{L_{p,q,w}}\leq
N \norm{g}_{L_{p,q,w}}
+ \frac{N}{\sqrt{\lambda} } \norm{f}_{L_{p,q,w}},
\end{align*}
where $N = N(d, \delta, \alpha, p, q, K_1 )$.
\end{corollary}

\begin{proof}
{\em Case 1.} We first consider the case when $p=q$ and $a^i=b^i=c=0$. We take $\gamma_0$, $\xi$, and $\lambda_0$ from Corollary \ref{6small}, and use a partition of unity argument in the spatial variables by picking $\{x_k\} \subset \rr^d$,
$\zeta^k \geq 0$, $\zeta^k \in C_0^\infty (\rr^d)$, $\supp (\zeta^k) \subset B_{\xi R_0/\sqrt{2}} (x_k)$ and
$$1 \le \sum_{k=1}^\infty |\zeta^k|^p \le N(p), \quad \sum_{k=1}^\infty |D_x \zeta^k |^p \leq N(d,\alpha, p, K_1, R_0).$$
For $u_k(t,x): = u(t,x) \zeta^k (x)$, we have
{\begin{align*}
\partial_t^\alpha u_k &- D_i (a^{ij} D_j u_k ) + \lambda u_k \\
&= (\partial_t^\alpha u) \zeta^k - D_i (a^{ij} \zeta^k D_j u ) + \lambda u \zeta^k
- D_i (a^{ij} u D_j \zeta^k) \\
&= D_i (g_i \zeta^k - a^{ij} u D_j \zeta^k)
+ f \zeta^k -a^{ij} D_i \zeta^k D_j u
- g_i D_i \zeta^k.
\end{align*}}
Therefore, by Corollary \ref{6small} for $\lambda \geq \lambda_0$,
\begin{align*}
&\sqrt{\lambda} \norm{u_k}_{p,w}
+ \norm{D u_k}_{p,w}\\
&\leq N
\big[\norm{\zeta^k g}_{p,w} + \norm{u D_j \zeta^k}_{p,w} \big]
+ \frac{N}{\sqrt{\lambda}}
\big[
\norm{f \zeta^k}_{p,w} + \norm{D_i \zeta^k D_j u }_{p,w}
+ \norm{D_i \zeta^k g_i}_{p,w}
\big],
\end{align*}
where $N = N(d, \delta, \alpha, p,K_1)$. By raising to the $p$-th power and summing in $k$, we get
\begin{align*}
&{\lambda}^{p/2} \norm{u}^p_{p,w} + \norm{Du}^p_{p,w}\\
&\leq
N_0 \norm{g}^p_{p,w} + N_1 \norm{u}^p_{p,w}
+ \frac{N_0}{{\lambda}^{p/2}} \norm{f}^p_{p,w}
+ \frac{N_1}{{\lambda}^{p/2}} \norm{Du}^p_{p,w}
+ \frac{N_1}{{\lambda}^{p/2}} \norm{g}^p_{p,w},
\end{align*}
where $N_0 (d, \delta, \alpha, p, K_1) $ and $N_1 = (d, \delta, \alpha, p,K_1, R_0)$.
We further pick $\lambda_0$ sufficiently large such that for all $\lambda \geq \lambda_0 (d, \delta, \alpha, p, K_1, R_0)$,
$${\lambda}^{p/2} - N_1 > \frac{{\lambda}^{p/2}}{2} \quad \text{and}\quad  1 - \frac{N_1}{{\lambda}^{p/2}} > \frac{1}{2},$$
which implies
\[
\sqrt{\lambda} \norm{u}_{p,w} + \norm{Du}_{p,w}
\leq N \norm{g}_{p,w} + \frac{N}{\sqrt{\lambda}} \norm{f}_{p,w},
\]
where $N = N(d, \delta, \alpha, p,K_1)$.

{\em Case 2.} Then, we consider the case when  $p=q$ without assuming $a^i, b^i, c = 0$. In this case, we have
\[
\partial_t^\alpha u
- D_i (a^{ij} D_j u)
+\lambda u
= D_i (g_i + a^i u)
+ f + b^i D_i u + cu.
\]
Therefore, the result follows from Case 1 by picking a possibly larger $\lambda_0$ depending on the upper bound of $a^i, b^i$, and $c$.

{\em Case 3.} Finally, for the general case when $p\neq q$, we use the extrapolation theorem in \cite[Theorem 2.5]{Ap}. The corollary is proved.
\end{proof}

Before we prove Theorem \ref{6main}, we introduce the following lemma, which implies that if $u\in \chh_{p,q,w,0}^{\alpha, 1} ( \rr^d_T )$, then
$\partial_t^{\alpha-1} u \in {\hh^{-1}_{p,q,w}( \rr^d_T)}$, and its norm can be controlled by the norm of $\partial_t^{\alpha} u$.
\begin{lemma}\label{6.8}
For $[w]_{p,q} \le K_1$, if there exist $f,g \in L_{p,q,w}  (\Omega_T)$ such that
\[
\int_{\Omega_T} u \phi_{tt}
= - \int_{\Omega_T} g_i D_i \phi
+ \int_{\Omega_T} f \phi
\]
for any $\phi \in C_0^\infty ([0,T) \times \Omega)$,
then there exist $\widetilde{g}, \widetilde{f} \in L_{p,q,w} ( \Omega_T)$
such that
\[
\int_{\Omega_T} u \phi_t
= \int_{\Omega_T} \widetilde{g}_i D_i \phi
- \int_{\Omega_T} f \phi
\]
for any $\phi \in C_0^\infty ([0,T) \times \Omega)$.
Furthermore, \[\norm{\widetilde{g}}_{L_{p,q,w}} \le NT\norm{g}_{L_{p,q,w}} \quad \text{and}\quad  \quad \norm{\widetilde{f}}_{L_{p,q,w}} \le NT\norm{f}_{L_{p,q,w}},\stepcounter{equation}\tag{\theequation}\label{jj4}\]
where $N = N(p,K_1)$.
\end{lemma}
\begin{proof}
Let
\[
\widetilde{g}_i (t,x) = \int_0^t g_i (s,x) \, ds
\quad \text{and}\quad
\widetilde{f} (t,x) = \int_0^t f(s,x) \, ds.
\]
For any $\phi \in C_0^\infty ([0,T) \times \Omega)$, we check that
\[
\int_{\Omega_T} u \phi_t
= \int_{\Omega_T} \widetilde{g}_i D_i \phi
- \int_{\Omega_T} \widetilde{f} \phi. \stepcounter{equation}\tag{\theequation}\label{jj3}
\]
Note that the right-hand side of \eqref{jj3} equals
\begin{align*}
& \int_{\Omega_T} \int_0^t g_i (s,x) \, ds \, D_i \phi (t,x) \, dt - \int_{\Omega_T}  \int_0^t f(s,x) \, ds \, \phi(t,x) \, dt   \\
&=  \int_{\Omega_T} \int_t^T  D_i \phi (s,x)\, ds \,  g_i (t,x) \, dt - \int_{\Omega_T} \int_t^T  \phi (s,x)\, ds \,  f (t,x) \, dt    \\
&=-\int_{\Omega_T} u \Phi_{tt} = \int_{\Omega_T} u \phi_t,
\end{align*}
where
\[
\Phi(t,x) = \int_t^T \phi(s,x) \, ds \quad \text{and}\quad  \Phi \in C_0^\infty ([0,T) \times \Omega).
\]
To prove \eqref{jj4}, note that
\[
\int_0^t \norm{g_i (s, \cdot) }_{L_{q,w_2}} \, ds
\leq t \fint_0^t \norm{g_i (s, \cdot)}_{L_{q,w_2}} \, ds
\leq T \cmm_t \left( \norm{g_i (\cdot, \cdot)}_{L_{q,w_2}} \right) (t),
\]
where $\cmm{_t}$ denotes the maximal function in $t$.
Thus, by the Minkowski inequality and the weighted maximal function theorem, we have
\begin{align*}
\norm{\widetilde{g_i} }_{p,q,w}
\leq \big\|
\int_0^t \norm{g_i (s, \cdot)}_{L_{q,w_2}} \, ds
\big\|_{L_{p,w}(0,T)}
\leq N T  \norm{g_i}_{p,q,w},
\end{align*}
where $N = N(p,K_1)$. The lemma is proved.
\end{proof}

To prove the a priori estimate \eqref{6gg}, we are going to take cutoff functions in time (cf. the proof of \eqref{99999}). To that end, we compute the fractional derivative of the product in the following lemma (cf. Lemma \ref{3.1}).

\begin{lemma}\label{6.666}
Let $p \in [1, \infty)$, $\alpha \in (1,2)$, $k \in \{ 1,2, \ldots \}$,
$-\infty < S < t_0 < T < \infty$, and $v \in \chh_{p,q,w,0}^{\alpha, 1} ( (S,T) \times \Omega )$. Then, for any infinitely differentiable function $\eta$ defined on $\rr$ such that $\eta(t) = 0$ for $t \leq t_0$, we have $\eta v\in \chh_{p,q,w,0}^{\alpha,1} ( (t_0, T) \times \Omega )$ and
\[
\partial_t^\alpha (\eta v)
= \partial^2_t I_{t_0}^{2 - \alpha} (\eta v)
= \eta \partial^2_t I_{S}^{2 - \alpha } v  + h,
\]
where
\begin{align*}
h(t,x) &=  \frac{ \alpha (\alpha - 1)}{\Gamma (2 - \alpha)}  \int_S^t (t-s)^{-\alpha - 1}
[ \eta(s) - \eta(t) - \eta'(t) (s-t) ] v (s,x) \, ds  \\
&\quad+ \frac{ \alpha}{\Gamma (2 - \alpha)}  \eta'(t)  \partial_t I_{t_0}^{2 - \alpha}  v
\end{align*}
in the sense of distribution.
\end{lemma}
\begin{proof}
Let $f,g\in L_{p,q,w}( (S,T) \times \Omega ) $ such that
$$ \partial_t^\alpha v = \partial_t^2
I_S^{2 - \alpha} v
= \mathrm{div} \, g + f.$$
By taking the mollification in the spatial variables and using the Fubini theorem, we have
$$
\partial_t^\alpha v^\ep = \partial_t^2
I_S^{2 - \alpha} v^\ep
= \mathrm{div} \, g^\ep + f^\ep\in L_{p,q,w}( (S,T) \times \Omega ).
$$
Therefore, $v^\ep \in \hh_{p,q,w,0}^{\alpha,2} ( (t_0, T) \times \Omega )$, and by Lemma \ref{3.1},
\begin{align*}
    \partial_t^\alpha (\eta v^\ep) &= \eta \partial^2_t I_{S}^{2 - \alpha } v^\ep \\
   &\quad+ \frac{ \alpha (\alpha - 1)}{\Gamma (2 - \alpha)}  \int_S^t (t-s)^{-\alpha - 1}
[ \eta(s) - \eta(t) - \eta'(t) (s-t) ] v^\ep (s,x) \, ds  \\
&\quad+ \frac{\alpha}{\Gamma (2 - \alpha)}  \eta'(t)  \partial_t I_{t_0}^{2 - \alpha}  v^\ep
\end{align*}
in the strong sense. Then, by sending $\ep \to 0$, and using the dominated convergence theorem and Lemma \ref{6.8}, we conclude
$$ \int_S^t (t-s)^{-\alpha - 1}
[ \eta(s) - \eta(t) - \eta'(t) (s-t) ] v^\ep (s,x) \, ds $$
converges to
$$\int_S^t (t-s)^{-\alpha - 1}
[ \eta(s) - \eta(t) - \eta'(t) (s-t) ] v(s,x) \, ds \quad \text{in}\quad  L_{p,q,w}((S,T) \times \Omega), $$
$$
\partial^2_t I_{S}^{2 - \alpha } v^\ep \to \partial^2_t I_{S}^{2 - \alpha } v \quad \text{in}\quad  \hh^{-1}_{p,q,w}((S,T) \times \Omega),
$$
and
$$
\partial_t I_{t_0}^{2 - \alpha } v^\ep \to \partial_t I_{t_0}^{2 - \alpha } v \quad \text{in}\quad  \hh^{-1}_{p,q,w}((t_0,T) \times \Omega).
$$
Therefore, the lemma is proved.
\end{proof}

We are ready to prove Theorem \ref{6main}. Recall that $\norm{\cdot}_{p,q,w} = \norm{\cdot}_{L_{p,q,w} ( \rr^d_T )}$, and $\norm{\cdot}_{p,q,w,(\tau_1,\tau_2)}:= \norm{\cdot}_{L_{p,q,w} ( (\tau_1,\tau_2) \times \rr^d )}$.
\begin{proof}[Proof of Theorem \ref{6main}]
We first prove the a priori estimate \eqref{6gg}. By adding $\lambda u $ to both sides of (\ref{6main estimate}), we have
\[
\partial_t^\alpha u
- D_i (a^{ij} D_j u + a^i u) -b^i D_i u - cu + \lambda u = D_i g_i + f +  \lambda u.
\]
By Corollary \ref{6par}, there exists $\gamma_0 = \gamma_0 (d, \delta, \alpha, p, q, K_1) > 0$ such that under Assumption \ref{ass2} $(\gamma_0)$, the following holds. There exists $\lambda_0 = \lambda_0(d, \delta, \alpha, p, K_1,R_0)$ such that for all $\lambda \geq \lambda_0$,
\begin{align*}
\sqrt{\lambda}
\norm{u}_{p,q,w}
+ \norm{Du}_{p,q,w}
\leq N \norm{g}_{p,q,w} + \frac{N}{\sqrt{\lambda}} \norm{f}_{p,q,w}
+ N \sqrt{\lambda} \norm{u}_{p,q,w},
\end{align*} and by picking $\lambda = \max (1,\lambda_0)$, we have
\begin{align*}
\norm{u}_{p,q,w}
+ \norm{Du}_{p,q,w}
\leq N \norm{g}_{p,q,w} + N \norm{f}_{p,q,w}
+ N_1 \norm{u}_{p,q,w}, \stepcounter{equation}\tag{\theequation}\label{6hh}
\end{align*}
where $N = N(d, \delta, \alpha, p, K_1)$ and $N_1 = N_1 (d, \delta, \alpha, p, K_1,R_0)$.
Therefore, in order to prove (\ref{6gg}), it is sufficient to prove
\[
\norm{u}_{p,q,w} \leq N\norm{g}_{p,q,w} + N \norm{f}_{p,q,w} \stepcounter{equation}\tag{\theequation}\label{9152}
\]
for some $N = N(d, \delta, \alpha, p, K_1,R_0,T)$.

{\em Step 1.} Take a positive integer $m$ to be specified below. Set $s_k = {kT}/{m}$ for $k =  -1, 0,1,\ldots,  m-1$, and define $\eta_k = \eta_k(t)$ as in \eqref{dd1} for $k\ge 0$.
Then, by Lemma \ref{6.666}, $u \eta_k \in \chh_{p,q,w,0}^{\alpha, 1}
( (s_{k-1}, s_{k+1} ) \times \rr^d  )$
and
\begin{align*}
& \partial_t^\alpha (u \eta_k)  {-} D_i (a^{ij} D_j (u \eta_k) + a^i (u \eta_k)) {-} b^i D_i (u \eta_k)  {-}c(u \eta_k) \\
& = D_i (g_i \eta_k) + f \eta_k  {-} \eta_k \partial_t^\alpha u  {+}\partial_t^\alpha (u \eta_k)\\
& = D_i (g_i \eta_k) + f \eta_k + \Tilde{N}  \int_{-\infty}^t (t-s)^{-\alpha - 1}
[ \eta_k(s) - \eta_k(t) - \eta_k'(t) (s-t) ] u(s,x) \, ds  \\
&\quad+ N \eta_k'  \partial_t I_{t_0}^{2 - \alpha} u\\
&=  D_i (g_i \eta_k) + N D_i(\widetilde{{g_i}} \eta_k') + f\eta_k + N \widetilde{f} \eta_k'(t)  \\
&\quad \quad + \Tilde{N}  \int_{-\infty}^t (t-s)^{-\alpha - 1}
[ \eta_k(s) - \eta_k(t) - \eta_k'(t) (s-t) ] \chi_{s\le s_k} u(s,x) \, ds, \stepcounter{equation}\tag{\theequation}\label{688}
\end{align*}
where $N$ and $\Tilde{N}$ depend only on $\alpha$, and
\[
\partial_t
I_0^{2 - \alpha} u =: \mathrm{div} \, \widetilde{g} + \widetilde{f}.
\]
By using the equation \eqref{6main estimate}, Lemma \ref{6.8}, and (\ref{6hh}), we have
\begin{align*}
&\norm{\widetilde{g}\eta_k'}_{p,q,w,(s_{k-1},s_{k+1})} \le \norm{\eta_k'}_{L_\infty} \norm{\widetilde{g}}_{p,q,w,(0,s_{k})} \\
&\le N \norm{\eta_k'}_{L_\infty}\frac{kT}{m}\{ \norm{g}_{p,q,w,(0,s_{k})} + \norm{Du}_{p,q,w,(0,s_{k})} +  \norm{u}_{p,q,w,(0,s_{k})} \}\\
&\le  Nk( \norm{g}_{p,q,w,(0,s_{k})} + \norm{Du}_{p,q,w,(0,s_{k})} +  \norm{u}_{p,q,w,(0,s_{k})})\\
&\le Nk( \norm{g}_{p,q,w,(0,s_{k})} + \norm{f}_{p,q,w,(0,s_{k})} +  \norm{u}_{p,q,w,(0,s_{k})}), \stepcounter{equation}\tag{\theequation}\label{9101}
\end{align*}
where $N = N (d,\delta,\alpha,p,q,K_1,R_0)$,
and similarly,
\begin{align*}
&\norm{\widetilde{f}\eta_k'}_{p,q,w,(s_{k-1},s_{k+1})} \le Nk ( \norm{g}_{p,q,w,(0,s_{k})} + \norm{f}_{p,q,w,(0,s_{k})}
+  \norm{u}_{p,q,w,(0,s_{k})}). \stepcounter{equation}\tag{\theequation}\label{9102}
\end{align*}

{\em Step 2.} By the triangle inequality,
\begin{align*}
&\norm{u}_{p,q,w,(s_{k},s_{k+1})} \le \norm{u\eta_k}_{p,q,w,(s_{k-1},s_{k+1})} \\
& \leq \norm{(u\eta_k)^\ep - u\eta_k}_{p,q,w,(s_{k-1},s_{k+1})}  + \norm{(u\eta_k)^\ep}_{p,q,w,(s_{k-1},s_{k+1})}. \stepcounter{equation}\tag{\theequation}\label{9105}
\end{align*}
For the first term in the right-hand side of \eqref{9105}, note for each $t$, by Lemma \ref{l5}, we have
\begin{align*}
\norm{(u\eta_k)^\ep (t, \cdot) - (u\eta_k) (t, \cdot) }_{q, w_2}
&\leq N \ep  \norm{D (u\eta_k) (t, \cdot)}_{q, w_2},
\end{align*}
where $N= N(d,q,K_1)$, which implies
\begin{align*}
\norm{(u\eta_k)^\ep - u\eta_k}_{p,q,w,(s_{k-1},s_{k+1})} \leq  N \ep \norm{D( u\eta_k)}_{p,q,w,(s_{k-1},s_{k+1})}. \stepcounter{equation}\tag{\theequation}\label{9103}
\end{align*}
To estimate $\norm{(u\eta_k)^\ep}_{p,q,w,(s_{k-1},s_{k+1})}$ in \eqref{9105}, we take the mollification in the spatial variables on both sides of (\ref{688}), and then move {all the terms} except $\partial_t^\alpha(u\eta_k)^\ep$ to the right-hand side of the equation. Using Lemma \ref{l6}, (\ref{9101}), and (\ref{9102}), for $\ep < 1$, we have
\begin{align*}
&\norm{\partial_t^\alpha(u\eta_k)^\ep}_{p,q,w,(s_{k-1},s_{k+1})}\\
&\le \frac{N}{\ep} \big\{\norm{g\eta_k}_{p,q,w,(s_{k-1},s_{k+1})}  + \norm{\widetilde{g}\eta_k'}_{p,q,w,(s_{k-1},s_{k+1})} \\
&\quad\quad  + \norm{D(u\eta_k)}_{p,q,w,(s_{k-1},s_{k+1})} + \norm{u\eta_k}_{p,q,w,(s_{k-1},s_{k+1})}  \big\} \\
&\quad + N\big\{\norm{f\eta_k}_{p,q,w,(s_{k-1},s_{k+1})} + \norm{\widetilde{f}\eta_k'}_{p,q,w,(s_{k-1},s_{k+1})}  + m^2T^{-\alpha}\norm{u}_{p,q,w,(0,s_{k})}   \big\}\\
&\leq \frac{N}{\ep} \big\{ m\norm{f}_{p,q,w,(0,s_{k+1})}
+ m\norm{g}_{p,q,w,(0,s_{k+1})} +
m\norm{u}_{p,q,w,(0,s_{k})} \\
&\quad \quad +
\norm{u}_{p,q,w,(s_{k-1},s_{k+1})} \big\} + N \big\{m\norm{f}_{p,q,w,(0,s_{k+1})}
+ m\norm{g}_{p,q,w,(0,s_{k+1})} \\
&\quad\quad + m\norm{u}_{p,q,w,(0,s_{k})} + m^2T^{-\alpha}\norm{u}_{p,q,w,(0,s_{k})}\big\} \\
& \leq \frac{N}{\ep} \big\{ m\norm{f}_{p,q,w,(0,T)}
+ m\norm{g}_{p,q,w,(0,T)} \\
&\quad \quad +
m\norm{u}_{p,q,w,(0,s_{k})}+
\norm{u}_{p,q,w,(s_{k-1},s_{k+1})} \big\}  + N  m^2T^{-\alpha}\norm{u}_{p,q,w,(0,s_{k})}  \stepcounter{equation}\tag{\theequation}\label{9104},
\end{align*}
where $N =N(d, \delta, \alpha, p, K_1,R_0) $. For the first inequality \eqref{9104}, we used
\begin{align*}
&\norm{\int_{0}^\cdot (\cdot-s)^{-\alpha - 1}
[ \eta_k(s) - \eta_k(\cdot) - \eta_k'(\cdot) (s-\cdot) ] \chi_{s\le s_k} u(s,x) \, ds}_{p,q,w,(s_{k-1},s_{k+1})}\\
&\leq
\norm{\eta_k''}_{L_\infty}
\norm{I^{2 - \alpha}
[{|}u (\cdot, x){|} \chi_{\cdot\leq s_k} ] }_{p,q,w,(s_{k-1}, s_{k+1})} \\
&\le N(\alpha,p,q,K_1)m^2T^{-\alpha}\norm{u}_{p,q,w,(0,s_{k})},
\end{align*}
and for the second inequality, we used (\ref{6hh}) to derive
\begin{align*}
&\norm{D(u\eta_k)}_{p,q,w,(s_{k-1},s_{k+1})} \le \norm{Du}_{p,q,w,(0,s_{k+1})} \\
&\le N\big\{ \norm{f}_{p,q,w,(0,T)} + \norm{g}_{p,q,w,(0,T)} + \norm{u}_{p,q,w,(0,s_{k+1})} \big\},
\end{align*}
where $N = N(d, \delta, \alpha, p, K_1,R_0)$.
Combining (\ref{9103}), (\ref{9104}), and (\ref{9105}), we have
\begin{align*}
&\norm{u}_{p,q,w,(s_{k},s_{k+1})}\\
& \leq \norm{(u\eta_k)^\ep - u\eta_k}_{p,q,w,(s_{k-1},s_{k+1})}  + \norm{(u\eta_k)^\ep}_{p,q,w,(s_{k-1},s_{k+1})} \\
&\le N\ep  \norm{D( u\eta_k)}_{p,q,w,(s_{k-1},s_{k+1})}+
N\big(\frac{T}{m}\big)^\alpha\norm{\partial_t^\alpha(u\eta_k)^\ep}_{p,q,w,(s_{k-1},s_{k+1})}\\
&\le N\ep \big(\norm{f}_{p,q,w,(0,T)} +\norm{g}_{p,q,w,(0,T)}+  \norm{u}_{p,q,w,(0,s_k)}+
\norm{u}_{p,q,w,(s_k,s_{k+1})}
\big)\\
&\quad \quad + \big( \frac{T}{m}\big)^\alpha \frac{N}{\ep} \big\{ m\norm{f}_{p,q,w,(0,T)} + m\norm{g}_{p,q,w,(0,T)} +
m\norm{u}_{p,q,w,(0,s_{k})}\\
&\quad \quad + \norm{u}_{p,q,w,(s_{k-1},s_{k+1})} \big\}
+ \big( \frac{T}{m}\big)^\alpha N m^2T^{-\alpha}\norm{u}_{p,q,w,(0,s_{k})}, \stepcounter{equation}\tag{\theequation}\label{1401}
\end{align*}
where $N= N(d,\delta,\alpha,p,q,K_1,R_0)$.
By first picking $\ep$ sufficiently small, and then picking $m$ sufficiently large so that
$\norm{u}_{p,q,w,(s_{k},s_{k+1})}$ can be absorbed to the left-hand side of \eqref{1401}, we arrive at
\begin{align*}
\norm{u}_{p,q,w,(s_k, s_{k+1})}
\leq N \norm{f}_{p,q,w,(0, T)}
+ N\norm{g}_{p,q,w,(0, T)} + N\norm{u}_{p,q,w,(0,s_{k})},
\end{align*}
where $N = N(d,\delta,\alpha,p,q,K_1,R_0,T)$.
Therefore, (\ref{9152}) follows from an induction on $k$, and the a priori estimate (\ref{6gg}) is proved. Finally, the existence of solutions follows from the solvability of the simple equation with leading coefficients $a^{ij} = \delta^{ij}$ given in Proposition \ref{8.29} together with the method of continuity.
\end{proof}

Next, we prove Theorem \ref{main4} by following a similar procedure {as} in Section \ref{5}.
\subsection{The half space case}
When $\Omega = \rr^d_+$, i.e., the half space, we derive the mean oscillation estimate by taking odd and even extensions to the whole space.
Recall the definition of the strong maximal function for the half space in \eqref{SM}.
Moreover, for $w(t,x) = w_1(t) w_2(x)$ where $w_1 \in A_p(\rr)$ and $w_2 \in A_q(\Omega)$, we denote the even extension of $w_2$ to the whole space by $\overline{w_2}(x)$, and let $\overline{w}(t,x) = w_1(t) \overline{w_2}(x)$.
Similar to Lemma \ref{ext}, we have the following lemma.

\begin{lemma}\label{9182}
For any
$u \in \mathring{\chh}^{\alpha, 1}_{p,q,w,0} (\Omega_T )$,
let $\overline{u}(t,x) = u(t, |x^1|, x') \mathrm{sgn} \, x^1$.
Then, $
\overline{u}
\in
 \chh^{\alpha, 1}_{p,q,\overline{w},0} (\rr^d_T )$ and
$$ \norm{u}_{\chh^{\alpha, 1}_{p,q,w} (\Omega_T)} \leq \norm{\overline{u}}_{\chh^{\alpha, 1}_{p,q,\overline{w}} (\rr^d_T)} \leq 2 \norm{u}_{\chh^{\alpha, 1}_{p,q,w} (\Omega_T)}. $$
\end{lemma}

\begin{proposition}\label{9181}
Let $\alpha \in (1,2)$, $p_0 \in (1, 2)$, $T \in (0, \infty)$, $a^{ij}$ be constants, and $u \in \mathring{\chh}_{p_0, 0, \mathrm{loc}}^{\alpha, 1} ( \Omega_T )$ satisfy
\[
 \partial_t^\alpha u - D_i(a^{ij}D_j u) = D_i g_i
\quad \text{in}\quad  \Omega_T \stepcounter{equation}\tag{\theequation}\label{9183}.
\]
Then, for any $(t_0, x_0) \in (0,T] \times \Omega,  r \in (0, \infty)$, and $\kappa \in (0, 1/4)$, we have
\begin{align*}
& ( |D u - (D u )_{(t_0-(\kappa r)^{2/\alpha},t_0) \times B^+_{\delta\kappa r}(x_0))} | )_{
(t_0-(\kappa r)^{2/\alpha},t_0) \times B^+_{\delta \kappa r}(x_0)
} \\
&\quad \leq N \kappa^\sigma   \sum_{k=0}^\infty {2^{-k\alpha}}( |D u |^{p_0} )^{1 / p_0}_{(t_0-2^{k}(r/2)^{2/\alpha}r,t_0)\times B^+_{\delta^{-1}r/2}(x_0)}\\
&\quad \quad+ N \kappa^{- (d+2 /\alpha)/p_0}
\sum_{k=0}^\infty 2^{-\alpha k} (|g|^{p_0})^{1/p_0}_{
( t_0 - 2^k r^{2/\alpha}, t_0 )
\times B^+_{  \delta^{-1} \sqrt{d} r}(x_0)
},\stepcounter{equation}\tag{\theequation}\label{uu1}
\end{align*}
where $N = N(d, \delta, \alpha, p_0)$ and $\sigma = \sigma (d, \alpha, p_0)$.
\end{proposition}
\begin{proof}
The proof is similar to that of Proposition \ref{5.2}. Indeed, when $a_{ij} = \delta_{ij}$, we extend $g$ by letting
$$ g_1 (x) = g_1(|x^1|,x'), \quad g_i (x) = g_i(|x^1|,x') \mathrm{sgn}\, x^1 \quad \text{for}\quad  i= 2,\ldots,d$$ so that the odd extension of $u$ is a solution of (\ref{9183}) in the whole space. Therefore, \eqref{uu1} follows from Proposition \ref{63.9}, Lemma \ref{9182}, and a change of variables.
\end{proof}

To prove Theorem \ref{main4}, we only need the following proposition for the case $G = \rr_+^d $. {However,} for convenience, here we include the result for the whole space, since we will need {both of} them later.
\begin{lemma}\label{91566}
Let $G = \rr_+^d$ or $\rr^d$, $\alpha \in (1,2)$, $T \in (0, \infty)$, $p,q \in (1, \infty)$, $ K_1 \in [1, \infty)$, and $[w]_{p,q} \le K_1$.
There exists $\gamma_0 = \gamma_0 (d, \delta, \alpha, p, q, K_1 ) >0$ such that under Assumption \ref{ass2} $(\gamma_0)$, the following holds. There exists
$
\lambda_0 = \lambda_0(d, \delta, \alpha, p, q, K_1, R_0)
$
such that
for any $\lambda\ge \lambda_0$ and $u \in \mathring{\chh}_{p,q,w,0}^{\alpha, 1} ( (0,T) \times \rr^d_+)$ or $\chh_{p,q,w,0}^{\alpha, 1} ( \rr^d_T )$ satisfying
$$
Lu + \lambda u = D_i g_i + f \quad \text{in}\quad  (0,T) \times G,
$$
we have
\begin{align*}
&\sqrt{\lambda} \norm{ u}_{L_{p,q,w} ( (0,T) \times G )} +
\norm{D u}_{L_{p,q,w} ( (0,T) \times G )}\\
&\leq
N \norm{g}_{L_{p,q,w} ( (0,T) \times G )}
+ \frac{N}{\sqrt{\lambda} } \norm{f}_{L_{p,q,w} ( (0,T) \times G )},
\end{align*}
where $N = N(d, \delta, \alpha, p, q, K_1)$.
\end{lemma}
\begin{proof}
{The case when $G = \rr^d$ was proved in Corollary \ref{6par}.}
{When} $G = \rr_+^d$, we can easily derive results analogous to Propositions \ref{12291}, \ref{64.2}, and Corollary \ref{6small} by using Proposition \ref{9181}, the weighted Hardy-Littlewood and Fefferman-Stein theorems in the half space, and S. Agmon's idea, respectively. Then, {Lemma} \ref{91566} follows from a partition of unity {argument} in the half space.
\end{proof}

In the following proof, for $\Omega = \rr^d_+$, we denote $\norm{\cdot}_{p,q,w} = \norm{\cdot}_{L_{p,q,w} ( \Omega_T )}$ and $\norm{\cdot}_{p,q,w,(\tau_1,\tau_2)}= \norm{\cdot}_{L_{p,q,w} ( (\tau_1,\tau_2) \times \Omega )}$.
\begin{proof}[Proof of Theorem \ref{main4}]
First, note that the solvability of the simple equation $$\partial_t^\alpha u - \Delta u = D_ig_i +f$$ can be proved by extending $f,g$ as
$$
g_1 (x) = g_1(|x^1|,x'), \quad g_i (x) = g_i(|x^1|,x') \mathrm{sgn} \, x^1 \quad \text{for}\quad  i= 2,\ldots,d,
$$
$$f (x) = f (|x^1|,x') \mathrm{sgn} \, x^1,
$$
and then applying Theorem \ref{6main}.
Therefore, by the method of continuity it suffices to prove the a priori estimate (\ref{9184}) under Assumption \ref{ass2} $(\gamma_0)$, where $\gamma_0$ is taken from Lemma \ref{91566}. By adding $\max (1,\lambda_0) u$ to both sides of the equation \eqref{915}, where $\lambda_0 = \lambda_0(d, \delta, \alpha, p, q, K_1, R_0)$ is derived from Lemma \ref{91566} in the half space, we have
\begin{align*}
\norm{u}_{p,q,w}
+ \norm{Du}_{p,q,w}
\leq N \norm{g}_{p,q,w} + N \norm{f}_{p,q,w}
+ N_1 \norm{u}_{p,q,w}, \stepcounter{equation}\tag{\theequation}\label{9281}
\end{align*}
where $N = N(d, \delta, \alpha, p, K_1)$ and $N_1 = N_1 (d, \delta, \alpha, p, K_1,R_0)$.
Therefore, in order to prove (\ref{9184}), it remains to prove
\[
\norm{u}_{p,q,w} \leq N\norm{g}_{p,q,w} + N \norm{f}_{p,q,w} \stepcounter{equation}\tag{\theequation}\label{919}
\]
for some $N = N(d, \delta, \alpha, p, K_1,R_0,T)$. We follow the proof of (\ref{9152}). The difference is that we need to take extensions before taking the mollification. For a positive integer $m$ to be specified below, we take $s_k = kT/m$ for $k = -1,0,1,\ldots, m-1$, and define the cutoff functions $\eta_k$ as in \eqref{dd1} for $k\ge 0$. Then (\ref{688}) is satisfied in the half space, i.e.,
\begin{align*}
 \partial_t^\alpha (u \eta_k) &=  D_i (g_i \eta_k  {+} a^{ij}  D_j (u \eta_k) {+} a^i (u \eta_k))) + N_4 D_i(\widetilde{{g_i}} \eta_k')  {+} c(u \eta_k) \\
 & \quad {+} b^i D_i (u \eta_k) + f\eta_k + N_4 \widetilde{f} \eta_k'(t)  \\
& \quad + N_3  \int_{-\infty}^t (t-s)^{-\alpha - 1}
[ \eta_k(s) - \eta_k(t) - \eta_k'(t) (s-t) ] \chi_{s\le s_k} u(s,x) \, ds,
\end{align*}
where $N_3$ and $N_4$ depend only on $\alpha$. Therefore, by Lemma \ref{l6} and (\ref{9281}), we conclude
\begin{align*}
&\norm{\partial_t^\alpha\overline{u\eta_k}^\ep}_{L_{p,q,\overline{w}} ( (s_{k-1},s_{k+1}) \times \rr^d )}
\\
&\le \frac{N}{\ep} \big\{ m\norm{f}_{p,q,w,(0,T)}
+ m\norm{g}_{p,q,w,(0,T)} + m\norm{u}_{p,q,w,(0,s_{k})} \\
&\quad \quad + \norm{u}_{p,q,w,(s_{k-1},s_{k+1})} \big\} + N  m^2T^{-\alpha}\norm{u}_{p,q,w,(0,s_{k})},\stepcounter{equation}\tag{\theequation}\label{uu2}
\end{align*}
where $N= N(d,\delta,\alpha,p,q,K_1,R_0)$, and $\overline{u\eta_k} \in \chh^{\alpha, 1}_{p,q,\overline{w},0} (\rr^d_T )$ denotes the odd extension of $u\eta_k$ in the $x^1$ direction.
Then, by \eqref{9281} and \eqref{uu2},
\begin{align*}
&\norm{u}_{p,q,w,(s_{k},s_{k+1})} \le \norm{\overline{u\eta_k}}_{L_{p,q,\overline{w}} ( (s_{k-1},s_{k+1}) \times \rr^d )}\\
& \leq \norm{\overline{u\eta_k}^\ep - \overline{u\eta_k}}_{L_{p,q,\overline{w}} ( (s_{k-1},s_{k+1}) \times \rr^d )}  + \norm{\overline{u\eta_k}^\ep}_{L_{p,q,\overline{w}} ( (s_{k-1},s_{k+1}) \times \rr^d )} \\
&\le N_1 \ep  \norm{D(\overline{u\eta_k})}_{L_{p,q,\overline{w}} ( (s_{k-1},s_{k+1}) \times \rr^d )}+
N_2(\frac{T}{m})^\alpha\norm{\partial_t^\alpha\overline{u\eta_k}^\ep}_{L_{p,q,\overline{w}} ( (s_{k-1},s_{k+1}) \times \rr^d )}\\
&\le  N_1 \ep\norm{D(u\eta_k)}_{p,q,w,(s_{k-1},s_{k+1})}\\
&\quad \quad + ( \frac{T}{m})^\alpha \frac{N}{\ep} \big\{ m\norm{f}_{p,q,w,(0,T)} + m\norm{g}_{p,q,w,(0,T)} +
m\norm{u}_{p,q,w,(0,s_{k})}\\
&\quad \quad \quad + \norm{u}_{p,q,w,(s_{k-1},s_{k+1})} \big\}
+ ( \frac{T}{m})^\alpha N m^2T^{-\alpha}\norm{u}_{p,q,w,(0,s_{k})}\\
&\le N \ep  (\norm{g}_{p,q,w,(0,T)}+ \norm{f}_{p,q,w,(0,T)} + \norm{u}_{p,q,w,(0,s_k)}+
\norm{u}_{p,q,w,(s_k,s_{k+1})}
)\\
&\quad \quad + ( \frac{T}{m})^\alpha \frac{N}{\ep}  \big\{ m\norm{f}_{p,q,w,(0,T)} + m\norm{g}_{p,q,w,(0,T)} +
m\norm{u}_{p,q,w,(0,s_{k})}\\
&\quad \quad \quad + \norm{u}_{p,q,w,(s_{k-1},s_{k+1})} \big\}
+ ( \frac{T}{m})^\alpha N m^2T^{-\alpha}\norm{u}_{p,q,w,(0,s_{k})},\stepcounter{equation}\tag{\theequation}\label{1402}
\end{align*}
where $N_1 = N_1(d,q,K_1)$, $N_2 = N_2(\alpha,p,q,K_1)$, $N= N(d,\delta,\alpha,p,q,K_1,R_0)$, and $\overline{w}$ denotes the even extension of $w$. By first picking $\ep$ sufficiently small, and then picking $m$ sufficiently large so that
$\norm{u}_{p,q,w,(s_{k},s_{k+1})}$ can be absorbed to the left-hand side of \eqref{1402}, we conclude
\begin{align*}
\norm{u}_{p,q,w,(s_k, s_{k+1})}
\leq N \norm{f}_{p,q,w,(0, T)}
+ N\norm{g}_{p,q,w,(0, T)} + N\norm{u}_{p,q,w,(0,s_{k})},
\end{align*}
where $N = N(d,\delta,\alpha,p,q,K_1,R_0,T)$.
Therefore, \eqref{919} follows from an induction on $k$, and Theorem \ref{main4} is proved.
\end{proof}

\subsection{The domain case}
In this subsection, we denote $$Lu = \partial_t^\alpha u - D_i (a^{ij} D_ju  + a^i u) -b^i D_iu - cu.$$
To prove Theorem \ref{divdomain}, we first deal with $L + \lambda$ for $\lambda$ sufficiently large (cf. Proposition \ref{5.4}).

\begin{proposition}\label{612}
Let $\Omega$ be a domain, $\alpha \in (1,2)$, $T \in (0, \infty)$, $p,q \in (1, \infty)$, $K_1 \in [1, \infty)$, and $[w]_{p,q,{\Omega}} \le K_1$.
There exist $\gamma_0 = \gamma_0 (d, \delta, \alpha, p, q, K_1 )\in(0,1)$ and $\theta = \theta(d, \delta, \alpha, p, q, K_1) \in (0,{1/3})$ such that under Assumption \ref{ass2} $(\gamma_0)$ and Assumption \ref{ass3} $(\theta)$, the following holds.
There exists $\lambda_1 = \lambda_1(d, \delta, \alpha, p, q, K_1, R_0,R_2) \ge 1$ such that for any $\lambda \ge \lambda_1$ and $f, g_i \in L_{p,q,w} (\Omega_T)$, there exists a unique solution $u \in \mathring{\chh}_{p,q,w,0}^{\alpha, 1} ( \Omega_T)$ to
\[
Lu +\lambda u = D_ig_i + f  \quad \text{in}\quad  \Omega_T, \stepcounter{equation}\tag{\theequation}\label{9158}
\]
and $u$ satisfies
$$
\norm{u}_{\chh_{p,q,w}^{\alpha,1} (\Omega_T )} \le N(\norm{g}_{p,q,w}+ \norm{f}_{p,q,w}),
$$
where $N = N(d,\delta,\alpha, p,q,K_1,R_0,R_2)$ and $\norm{\cdot}_{p,q,w} = \norm{\cdot}_{L_{p,q,w} ( \Omega_T )}$. Moreover, for any $\lambda \ge 0$, $u \in \mathring{\chh}_{p,q,w,0}^{\alpha, 1} ( \Omega_T)$ satisfying (\ref{9158}), we have
\begin{align*}
\norm{u}_{\chh_{p,q,w}^{\alpha,1} (\Omega_T )} \le N( \norm{f}_{p,q,w}+ \norm{g}_{p,q,w}
+ \norm{u}_{p,q,w}), \stepcounter{equation}\tag{\theequation}\label{9291}
\end{align*}
where $N = N(d, \delta, \alpha, p, q, K_1, R_0,R_2)$.
\end{proposition}
\begin{proof}
The proof is similar to that of Proposition \ref{5.4}. When $p=q$ and $\lambda$ is sufficiently large, we construct a weak solution to (\ref{9158}). See Lemma \ref{l1} - \ref{l3}  for details. When $p\neq q$, we use the extrapolation theorem in \cite[Theorem 2.5]{Ap}.
\end{proof}

We denote
$$
\norm{\cdot}_{p,q,w} = \norm{\cdot}_{L_{p,q,w} ( \Omega_T )} \quad \text{and}\quad \norm{\cdot}_{p,q,w,(\tau_1,\tau_2)}= \norm{\cdot}_{L_{p,q,w} ( (\tau_1,\tau_2) \times \Omega )}
$$
in the following proof.

\begin{proof}[Proof of Theorem \ref{main4}]
By the solvability of
$$
Lu +\lambda u = D_ig_i + f
$$
for $\lambda$ sufficiently large given in Proposition \ref{612} together with the method of continuity, it remains to prove the a priori estimate (\ref{9184}) under Assumption \ref{ass2} $(\gamma_0)$ and Assumption \ref{ass3} $(\theta)$, where $\gamma_0$ and $\theta$ are taken from  Proposition \ref{612}. By (\ref{9291}), it suffices to prove
\[
\norm{u}_{p,q,w} \leq N\norm{g}_{p,q,w} + N \norm{f}_{p,q,w}.\stepcounter{equation}\tag{\theequation}\label{9302}
\]
We follow the proof of (\ref{9152}) together with a partition of unity {argument} in the spatial variables. For a positive integer $m$ to be specified below, we take $s_k = kT/m$ for $k = -1,0,1,\ldots, m-1$, and define the cutoff functions $\eta_k$ as in \eqref{dd1} for $k\ge0$. Furthermore, we take
${\{\xi^l\}}$ to be a partition of unity in $\Omega$ defined in (\ref{partition}). Note that $\xi^l$ is either supported in $B_{{r_0}/{40}}(x_l) $ for some $x_l \in  \Omega$ or supported in $B_{{r_0}/{12}}(x_l) $ for some $x_l \in \partial \Omega$, where $r_0 =  \min (R_0, R_2)$.
Then, for all $l$, $\xi^lu$ satisfies $$L(\xi^lu):= D_ig^l_i + f^l \quad \text{in}\quad  \Omega \cap \supp(\xi^l), $$ where
\begin{align*}
    g^l_i =  \xi^l g_i - a^{ij}uD_j\xi^l \quad \text{and}\quad
    f^l = \xi^lf -g_iD_i\xi^l - (a^{ij}D_ju+a^iu)D_i\xi^l - b^iuD_i\xi^l
\end{align*}
{such that}
$$
|g^l| + |f^l| \le N( |\xi^lg| + |u D\xi^l| + |gD\xi^l| + |\xi^lf| + |DuD\xi^l|).
$$
It follows that for {any} fixed $k$ and $l$,
\begin{align*}
&\norm{\eta_k \xi^l u}_{p,w,(s_{k-1},s_{k+1})}^p\\
&\le N \ep \big(\norm{f^l}_{p,w,(0,T)} +\norm{g^l}_{p,w,(0,T)}+  \norm{\xi^l u}_{p,w,(0,s_k)}+
\norm{\xi^l u}_{p,w,(s_k,s_{k+1})}
\big)\\
&\quad \quad + ( \frac{T}{m})^\alpha \frac{N}{\ep} \big\{ m\norm{f^l}_{p,w,(0,T)} + m\norm{g^l}_{p,w,(0,T)} +
m\norm{\xi^l u}_{p,w,(0,s_{k})}\\
&\quad \quad + \norm{\xi^l u}_{p,w,(s_{k-1},s_{k+1})} \big\}
+ ( \frac{T}{m})^\alpha N m^2T^{-\alpha}\norm{\xi^l u}_{p,w,(0,s_{k})},\stepcounter{equation}\tag{\theequation}\label{3333}
\end{align*}
where $N = N(d,\delta,\alpha,p,K_1,R_0,R_2)$.
Indeed, to prove \eqref{3333}, when $x_l\in \Omega$, we take the extension of weights to the whole space, and then apply \eqref{1401}; when  $x_l\in \partial \Omega$, we apply a change of coordinates followed by the extension of weights to the half space, and then apply \eqref{1402}.
By using \eqref{3333}, \eqref{eq8.01} and \eqref{9291},
we have
\begin{align*}
&\norm{u}_{p,w,(s_{k},s_{k+1})}^p \le \norm{u\eta_k}_{p,w,(s_{k-1},s_{k+1})}^p \le \sum_{l} \norm{\eta_k \xi^l u}_{p,w,(s_{k-1},s_{k+1})}^p\\
&\le N\ep^p \Big(\sum_l\norm{|g^l| + |f^l|}_{p,w,(0,s_{k+1})}^p + \norm{\xi^lu}_{p,w,(0,s_k)}^p+
\norm{\xi^lu}_{p,w,(s_k,s_{k+1})}^p
\Big)\\
 &\quad \quad + \Big[( \frac{T}{m})^\alpha \frac{1}{\ep}\Big]^p N\Big\{\sum_l m\norm{|g^l| + |f^l|}_{p,w,(0,s_{k+1})}^p  +
 m\norm{\xi^lu}_{p,w,(0,s_{k})}^p\\
&\quad \quad \quad + \norm{\xi^lu}_{p,w,(s_{k-1},s_{k+1})}^p \Big\}
+ \Big[( \frac{T}{m})^\alpha  m^2T^{-\alpha}\Big]^pN \sum_l \norm{\xi^l u}_{p,w,(0,s_{k})}^p\\
&\le
 N\ep^p \Big(\norm{g}^p_{p,w,(0,s_{k+1})}+ \norm{f}^p_{p,w,(0,s_{k+1})} +  \norm{u}^p_{p,w,(0,s_{k+1})}+
\norm{Du}^p_{p,w,(0,s_{k+1})}
\Big)\\
&\quad \quad + \Big[( \frac{T}{m})^\alpha \frac{1}{\ep}\Big]^p N \Big\{ m\norm{f}^p_{p,w,(0,s_{k+1})} + m\norm{g}^p_{p,w,(0,s_{k+1})} +
m\norm{u}^p_{p,w,(0,s_{k+1})}\\
&\quad \quad \quad +m\norm{Du}^p_{p,w,(0,s_{k+1})} \Big\}
+ \Big[( \frac{T}{m})^\alpha  m^2T^{-\alpha}\Big]^pN\norm{u}^p_{p,w,(0,s_{k})}\\
&\le
 N\ep^p \Big(\norm{g}^p_{p,w,(0,s_{k+1})}+ \norm{f}^p_{p,w,(0,s_{k+1})} + \norm{u}^p_{p,w,(0,s_k)}+
\norm{u}^p_{p,w,(s_k,s_{k+1})}
\Big)\\
&\quad \quad + \Big[( \frac{T}{m})^\alpha \frac{1}{\ep}\Big]^p N \Big\{ m\norm{f}^p_{p,w,(0,s_{k+1})} + m\norm{g}^p_{p,w,(0,s_{k+1})} +
m\norm{u}^p_{p,w,(0,s_{k})}\\
&\quad \quad \quad +m \norm{u}^p_{p,w,(s_k,s_{k+1})} \Big\}
+ \Big[( \frac{T}{m})^\alpha  m^2T^{-\alpha}\Big]^pN\norm{u}^p_{p,w,(0,s_{k})},\stepcounter{equation}\tag{\theequation}\label{jj7}
\end{align*}
where $N = N(d,\delta,\alpha,p,K_1,R_0,R_2)$. Therefore, using (\ref{9291}) and the fact that $\alpha \in (1,2) $ and by first picking $\ep$ sufficiently small, and then picking $m$ sufficiently large so that
we can absorb $\norm{u}_{p,w,(s_{k},s_{k+1})}^p$ to the left{-hand side} of \eqref{jj7}, we have
\begin{align*}
\norm{u}_{p,w,(s_k, s_{k+1})} \leq N \norm{f}_{p,w,(0, T)} + N\norm{g}_{p,w,(0, T)} + N\norm{u}_{p,w,(0,s_{k})},
\end{align*}
where $N = N(d,\delta,\alpha,p,K_1,R_0,R_2,T)$. Thus, when $p=q$, (\ref{9302}) follows from an induction on $k$. When $p \neq q$, we use the extrapolation theorem in \cite[Theorem 2.5]{Ap} to conclude (\ref{9302}). The theorem is proved.
\end{proof}

\appendix
\section{}\label{A}
In the first part of the appendix, we prove some results mentioned in previous sections.
\begin{lemma}\label{A2}
If $p \in [1, \infty)$ and $u\in C^2 [0, \infty)$ satisfies $u(0)=u'(0)=0$, then for any $\ep > 0$ and $T>0$,
\[
\norm{u'}_{L_p (0,T)}
\leq N\ep^{-1} \norm{u}_{L_p (0,T)} +
\ep\norm{u''}_{L_p(0,T)},
\]
where $N$ is independent of $u,p,\ep,T$.
\end{lemma}
\begin{proof}
Take $\xi \in C_0^\infty (\rr)$ such that $\xi(0) = 1, \xi'(0) = 0$, and take the zero extension of $u$ on $\rr$. Then, by the fundamental theorem of calculus,
\begin{align*}
(u'(t+s) \xi(s) )_s
&= u''(t+s) \xi (s)
+ ( u(t+s) \xi'(s) )_s
-u(t+s) \xi''(s),
\end{align*}
and
$$
u'(t)
= \int_{-\infty}^0  ( u'(t+s) \xi(s) )_s \, ds
= \int_{-\infty}^0  u''(t+s) \xi(s) \, ds
- \int_{-\infty}^0 u(t+s) \xi''(s) \, ds.
$$
By the Minkowski inequality and the fact that $u$ is supported in $[0, \infty)$,
\begin{align*}
\norm{u'}_{L_p (0,T)}
&\leq \int_{-\infty}^0
(|\xi(s)| + |\xi''(s)| )
(
\norm{u'' (\cdot + s)}_{L_p (0,T)}
+ \norm{u (\cdot + s)}_{L_p (0,T)}
) \\
&\leq N \big[
\norm{u''}_{L_p (0,T)}
+ \norm{u}_{L_p (0,T)}
\big].\stepcounter{equation}\tag{\theequation}\label{tt1}
\end{align*}
For any $\ep >0$, let $v ( \cdot):= u( \ep \cdot)$. Because $N$ in \eqref{tt1} is independent of $T$, we replace $u$ with $v$ to get
$$
\norm{v'}_{L_p (0,T/\ep)}
\leq N \big[
\norm{v''}_{L_p (0,T/\ep)}
+ \norm{v}_{L_p (0,T/\ep)}
\big],
$$
which implies
$$
\ep \norm{u'}_{L_p (0,T)}\leq
N \big[
\ep^2 \norm{u''}_{L_p (0,T)}
+ \norm{u}_{L_p (0,T)}
\big].
$$
Therefore, the desired result follows.
\end{proof}

In Lemmas \ref{l5} and \ref{l6} below, let $\zeta \in C_0^\infty(\rr^d)$ satisfy $\supp(\zeta)\subset B_1$ and $\int_{\rr^d}\zeta =1${,} and let $\zeta^\ep(\cdot) := \ep^{-d}\zeta(\cdot/\ep)$.
\begin{lemma}\label{l5}
For $q\in [1,\infty)$, $w_2\in A_q$ with $[w_2]_{A_q} \le K_1$, $v \in W_{q,w_2}^1$, and $v^\ep := v \ast \zeta^\ep $, we have
$$
\norm{v^\ep - v}_{q, w_2}
\leq  N \ep \norm{Dv}_{q, w_2},
$$
where $N = N (d,q,K_1)$.
\end{lemma}
\begin{proof}
By the definition of $v^\ep$ and the Fubini theorem,
\begin{align*}
|v^\ep (x) - v(x) |
&= | \int_{B_1} [v(x-\ep y) - v(x)] \zeta(y) \, dy | \\
&= | \int_{B_1} \ep \int_0^1 y \cdot Dv (x - s \ep y) \, ds \, \zeta(y) \, dy| \\
&= \ep | \int_0^1 \int_{B_1} \zeta (y) [ y \cdot Dv (x- s \ep y) ] \,dy\,ds  | \\
&\leq \ep \cdot N \int_0^1 \fint_{B_1} |Dv (x - s \ep y )| \, dy \, ds \\
&\leq  N\ep \cmm Dv (x),
\end{align*}
where $\cmm $ denotes the maximal function.
Therefore, by the weighted maximal function theorem, we have
$$
\norm{v^\ep - v}_{q, w_2}
\leq N\ep  \norm{\cmm Dv (x)}_{q, w_2}
\leq N\ep  \norm{Dv}_{q, w_2}.
$$
The lemma is proved.
\end{proof}

\begin{lemma}\label{l6}
If $[w]_{p,q} \le K_1 $ and $u \in \chh_{p,q,w,0}^{\alpha, 1} (\rr^d_T )$ satisfies
\[
\partial_t^\alpha u = D_i g_i + f \quad \text{in}\quad  \rr^d_T \stepcounter{equation}\tag{\theequation}\label{al66}
\]
for some $f,g_i \in L_{p,q,w}(\Omega_T)$. Then $u^\ep := u \ast_x \zeta^\ep \in \hh_{p,q,w,0}^{\alpha, 2} ( \rr^d_T )$, and for any $\ep <1$, we have
 \begin{align*}
\norm{\partial_t^\alpha u^\ep}_{L_{p,q,w}(\rr^d_T)}
&\le \frac{N }{\ep} \norm{g}_{L_{p,q,w}(\rr^d_T)}  + N \norm{f}_{L_{p,q,w}(\rr^d_T)}, \stepcounter{equation}\tag{\theequation}\label{al6}
\end{align*}
where $N = N(d,\delta,q,K_1)$. Moreover, let $\Omega = \rr_+^d$, $u \in \mathring{\chh}^{\alpha, 1}_{p,q,w,0} (\Omega_T )$, $[w]_{p,q,\Omega} \le K_1 $, $\overline{u}$ be the odd extension of $u$ in the $x^1$ direction, and $\overline{w}$ be the even extension of $w$ in the $x^1$ direction. If \eqref{al66} is satisfied in $\Omega_T$ for some $f,g_i \in L_{p,q,w}(\Omega_T)$, then \begin{align*}
\norm{\partial_t^\alpha \overline{u}^\ep}_{L_{p,q,w}(\rr^d_T)}
&\le \frac{N }{\ep} \norm{g}_{L_{p,q,w}(\Omega_T)}  + N \norm{f}_{L_{p,q,w}(\Omega_T)}, \stepcounter{equation}\tag{\theequation}\label{al666}
\end{align*}
where $N = N(d,\delta,q,K_1)$.
\end{lemma}
\begin{proof}
We first consider the whole space case. Taking the convolution in the spatial variables in both sides of \eqref{al66} yields
\begin{align*}
\partial_t^\alpha u^\ep
&= g_i  \ast_x D_i \zeta^\ep
+ f \ast_x \zeta^\ep. \stepcounter{equation}\tag{\theequation}\label{tt3}
\end{align*}
Note that \begin{align*}
|f \ast \zeta^\ep (t,x) |
&= | \int_{B_1} f(t,x - \ep y)  \zeta (y) \, dy | \\
&\leq \norm{\zeta}_{L^\infty} |B_1| | \fint_{B_1} f(t, x - \ep y) \, dy | \\
&\leq \norm{\zeta}_{L^\infty} |B_1|  \cmm_x f (t,x),
\end{align*}
where $\cmm_x$ denotes the maximal function {in} the spatial variables for fixed $t$.
Thus, by the weighted maximal function theorem,
\[
\norm{f \ast \zeta^\ep(t,\cdot)}_{q, w_2}
\leq \norm{\zeta}_{L^\infty}  N \norm{f(t,\cdot)}_{q, w_2},\stepcounter{equation}\tag{\theequation}\label{tt4}
\]
where $N = N(d,q,K_1) $,
and similarly,
\[
\norm{g_i \ast D_i \zeta^\ep(t,\cdot) }_{q, w_2}
\leq \norm{D_i \zeta}_{L^\infty}
\cdot \frac{N}{\ep}
\norm{g_i (t,\cdot)}_{q, w_2}.\stepcounter{equation}\tag{\theequation}\label{tt5}
\]
Therefore, \eqref{al6} follows from \eqref{tt3}, \eqref{tt4}, and \eqref{tt5}.

To prove \eqref{al666}, we note that $\overline{u}\in \chh_{p,q,\overline{w},0}^{\alpha, 1}(\rr^d_T) $ by Lemma \ref{9182}. By taking proper odd and even extensions of $g_i$ and $f$, the extensions satisfy \eqref{al66} in the whole space, and \eqref{al666} follows from \eqref{al6}.
The lemma is proved.
\end{proof}

\section{}\label{B}
In the second part of the appendix, we prove some results for equations in {a} smooth domain $\Omega$. This section is organized as follows. First, we consider the restrictions and extensions of the weights of the spatial variables. Then, we prove the results for non-divergence form equations in Lemmas \ref{A3}-\ref{A5}. Finally, the results for equations in divergence form are proved in Lemmas \ref{l1}-\ref{l3}.

Let $a^{ij}$'s satisfy Assumption \ref{ass2} $(\gamma_0)$ and let $\Omega$ satisfy Assumption \ref{ass3}  $(\theta)$ {for divergence form equations} {or Assumption \ref{ass4}  $(\theta)$ for non-divergence form equations} with $\gamma_0 \in (0,1)$ and $\theta \in (0,1/3)$ to be specified. We locally flatten the boundary of $\partial \Omega$ {as follows}. For $r_0 = \min (R_0, R_2)$ and any $z\in \partial \Omega$, there exists a coordinate system and a mapping $\phi$ such that
$$
\Omega \cap B_{r_0}(z) = \{x\in B_{r_0}(z) : x^1 > \phi(x')\}.
$$
Let
\begin{align*}
\Phi^1(x) = x^1 - \phi(x'), \quad \Phi^j(x) = x^j \quad \text{for}\quad  j\ge 2, \quad \Psi = \Phi^{-1}  \stepcounter{equation}\tag{\theequation}\label{1030},
\end{align*}
and without loss of generality, assume that $\Phi(z) = 0$. Moreover, let
\begin{align*}
    \eta \in C_{0}^{\infty}(B_{3r_0/32}(z)), \,\, \eta =1 \,\,\text{in}\,\, B_{r_0/12}(z), \,\, \eta^z(\cdot)= \eta(\cdot -z), \,\,\text{and}\,\, \hat{\eta}^z(\cdot) = \eta^z(\Psi(\cdot)) . \stepcounter{equation}\tag{\theequation}\label{10031}
\end{align*}
Since $\theta \le 1/3$, $\hat{\eta}^z \in C_{0}^{\infty} (B_{r_0/8})$ and $\hat{\eta} =1$ in $B_{r_0/16}$,
where we used
$$B_{3r_0/32}(z)\subset B_{\frac{r_0}{8(1+ \theta)}}(z) \subset \Psi(B_{r_0/8}) , \quad \Psi(B_{r_0/16}) \subset B_{r_0(1+ \theta)/16}(z) \subset B_{r_0/12}(z).$$
Furthermore,
 {for} $r_1 := r_0/2 > 0 ${, we have}
$$
B_{r_1} \subset \Phi(B_{r_0}(z))\quad \text{and}\quad  B_{r_1}^+ \subset \Phi(\Omega \cap B_{r_0}(z)).
$$
Indeed, for $(y^1,y') \in B_{r_1}$, we can easily check that $(y^1+\phi(y'),y') \in B_{r_0}(z)$ and $\Phi(y^1+\phi(y'),y') = (y^1,y') $.

\begin{lemma}\label{exten}
(1) If $w\in A_p(\Omega)$, then $w\in A_p(\Omega \cap  B_{r_0}(z))$ and $w(\Psi) \in A_p(\Phi(\Omega \cap B_{r_0}(z)))$.

(2) Let $\delta>0$ and $C^+_\delta= [0,\delta] \times [-\delta/2,\delta/2]^{d-1}$. If $w\in A_p(C^+_\delta)$, then there exists $\overline{w} \in A_p(\rr^d_+)$ such that $\overline{w}= w $ in $C^+_\delta$.
\end{lemma}
\begin{proof}
For (1), for any $x \in \Omega \cap  B_{r_0}(z), r > 0$, we have
\begin{align*}
&\left[
\fint_{B_r(x) \cap \Omega \cap  B_{r_0}(z) } w
\right]
\left[
\fint_{B_r(x) \cap \Omega \cap  B_{r_0}(z) } w^{-\frac{1}{p-1}}
\right]^{p-1} \\
&\leq \left[ \frac{|B_r(x) \cap \Omega|}{|B_r(x) \cap \Omega \cap  B_{r_0}(z)|}\fint_{B_r(x)\cap \Omega } w \right] \left[\frac{|B_r(x) \cap \Omega|}{|B_r(x) \cap \Omega \cap  B_{r_0}(z)|}\fint_{B_r(x)\cap \Omega } w^{-\frac{1}{p-1}}  \right]^{p-1}\\
&\le N(r_0,{d}, p)[w]_{p,\Omega},
\end{align*}
where when $r$ is small, we used the smoothness of $\Omega$, and when $r> 2r_0$, we used the fact that $B_r(x) \cap \Omega \cap  B_{r_0}(z)=\Omega \cap  B_{r_0}(z) $. Thus, $w\in A_p(\Omega \cap  B_{r_0}(z))$.

Next, for any $\overline{y} \in \Phi(\Omega \cap B_{r_0}(z))$ and $r > 0$, {since} $\det D\Psi = 1$, we have
\begin{align*}
&\left[
\fint_{B_{r( \overline{y} )} \cap \Phi(\Omega \cap B_{r_0}(z)) } w (\Psi (y)) \, dy
\right]
\left[
\fint_{B_{r( \overline{y} )} \cap \Phi(\Omega \cap B_{r_0}(z)) }  [w (\Psi (y))]^{-\frac{1}{p-1}}  \, dy
\right]^{p-1} \\
&=
\left[
\fint_{ \Psi ( \broy) \cap \Omega \cap B_{r_0}(z) }
w (x)  \, dx
\right]
\left[
\fint_{ \Psi ( \broy) \cap \Omega \cap B_{r_0}(z)  }
w(x)^{-\frac{1}{p-1}} \cdot  \, dx
\right]^{p-1} \\
&\leq
\left( \frac{| B_{2r} (\Psi (\overline{y})) \cap \Omega \cap B_{r_0}(z) |}{ |\Psi ( \broy) \cap \Omega \cap B_{r_0}(z)| } \right)^{p}
\left(
\fint_{ B_{2r} (\Psi (\overline{y})) \cap \Omega \cap B_{r_0}(z) } w \, dx
\right) \\
& \quad \cdot
\left(
\fint_{ B_{2r} (\Psi (\overline{y})) \cap \Omega \cap B_{r_0}(z) } w^{ -\frac{1}{p-1}}   \, dx
\right)^{p-1}\\
&\leq
\left( \frac{| B_{2r} (\Psi (\overline{y})) \cap \Omega \cap B_{r_0}(z) |}{ |B_{r/2} (\Psi (\overline{y}))\cap \Omega \cap B_{r_0}(z)| } \right)^{p}
\left(
\fint_{ B_{2r} (\Psi (\overline{y})) \cap \Omega \cap B_{r_0}(z) } w \, dx
\right) \\
& \quad \cdot
\left(
\fint_{ B_{2r} (\Psi (\overline{y})) \cap \Omega \cap B_{r_0}(z) } w^{ -\frac{1}{p-1}}   \, dx
\right)^{p-1} \\
&\le N({r_0},d,p) [w]_{p,\Omega \cap B_{r_0}(z)}.
\end{align*}
Thus, $w(\Psi) \in A_p(\Phi(\Omega \cap B_{r_0}(z)))$, and $(1)$ is proved.

For (2), we define $\overline{w}$ by taking the even extension of $w$ along each side as follows. First, for $x^1 \in (\delta,2\delta]$ and $x' \in [-\delta/2,\delta/2]^{d-1}$, let
$$\overline{w}(x^1,x') = w(2\delta -x^1,x'),$$ and after defining $\overline{w}$ on $(0,2\delta] \times [-\delta/2,\delta/2]^{d-1}$, we take {the} periodic extension, along the $x^1$ direction, of $\overline{w}$ to $(0,\infty) \times [-\delta/2,\delta/2]^{d-1} $ {with} period $2\delta$. Next, for any $x^1\in (0,\infty)$, $x^2 \in (\delta/2,3\delta/2]$, and $x'' = (x^3,\ldots, x^d)\in [-\delta/2,\delta/2]^{d-2}$, let
$$\overline{w}(x^1,x^2, x') = \overline{w}(x^2,\delta - x^2,x'),$$ and extend $\overline{w}$, along the $x^2$ direction, to $(0,\infty) \times \rr \times [-\delta/2,\delta/2]^{d-2} $ with period $2\delta$. By repeating this process, we extend $\overline{w}$ to $\rr^d_+$. For any $x=(x^1,\ldots,x^d) \in \rr^d_+$,  we find its ``preimage'' $x_0 \in C^+_\delta$, i.e., the point {at which} $w(x)$ is defined by. {More precisely,} if $x_1 - \lfloor x_1/2\delta \rfloor 2\delta >\delta$, we take $x_0^1 = (\lfloor x_1/2\delta \rfloor + 1 )2\delta - x_1 $ and if $x_1 - \lfloor x_1/2\delta \rfloor 2\delta \le \delta$, we take $x_0^1 = x_1 - \lfloor x_1/2\delta \rfloor 2\delta$. Similarly, we find $(x_0^2,\ldots,x_0^d)$. When $r < \delta/2$, $B_r(x)$ intersects $l$ cubes $C_1, \ldots, C_l$ for some $l \le 2^d$. Thus, by the symmetry of $\overline{w}$,
$$
\int_{ C^+_\delta \cap B_r(x_0)} w^\beta = \int_{C(x) \cap B_r(x)} \overline{w}^\beta \ge \int_{C_i\cap B_r(x)} \overline{w}^\beta, \quad i= 1,\ldots, l,
$$
where $\beta = 1$ or $ -1/(p-1)$, and $C(x)$ is the cube containing $x$. Therefore, when $r < \delta/2$,
\begin{align*}
&\left[
\fint_{B_r(x) \cap \rr^d_+ } \overline{w}
\right]
\left[
\fint_{B_r(x) \cap \rr^d_+ } \overline{w}^{-\frac{1}{p-1}}
\right]^{p-1}\\
&\le \left[ \frac{|C^+_\delta \cap B_r(x_0)|}{|B_r(x) \cap \rr^d_+|} 2^d \fint_{C^+_\delta \cap B_r(x_0)}w\right]
\left[ \frac{|C^+_\delta \cap B_r(x_0)|}{|B_r(x) \cap \rr^d_+|} 2^d \fint_{C^+_\delta \cap B_r(x_0)}w^{-\frac{1}{p-1}}\right]^{p-1}\\
&\le N(d,p) [w]_{p,C^+_\delta}.
\end{align*}
{On the other hand, w}hen $r \ge \delta/2$, $B_r(x)\cap \rr^d_+ $ can be covered by $N(r,d,\delta)$ cubes, and $N(r,d,\delta)|C^+_\delta|/|B_r(x)\cap \rr^d_+| $
{is} bounded above by some constant independent of $r$ {and $\delta$}.
Therefore, we still have
\begin{align*}
&\left[
\fint_{B_r(x) \cap \rr^d_+ } \overline{w}
\right]
\left[
\fint_{B_r(x) \cap \rr^d_+ } \overline{w}^{-\frac{1}{p-1}}
\right]^{p-1}\\
&\le \left[ \frac{N(r,d,\delta)|C^+_\delta |}{|B_r(x)\cap \rr^d_+  |}  \fint_{C^+_\delta }w\right]
\left[ \frac{N(r,d,\delta)|C^+_\delta |}{|B_r(x)\cap \rr^d_+  |}  \fint_{C^+_\delta }w^{-\frac{1}{p-1}}\right]^{p-1}\\
&\le N(d,p) [w]_{p,C^+_\delta}.
\end{align*}
The lemma is proved.
\end{proof}

\subsection{Results for non-divergence form equations}
For $z\in \partial \Omega$, we denote
$$
L =   \partial_t^{\alpha}   -a^{ij}D_{ij}    {-} b^iD_i   {-} c \quad \text{and}\quad \hat{L}^z =   \partial_t^{\alpha}    {-} \hat{a}^{ij}D_{ij}    {-}\hat{b}^iD_i    {-} \hat{c},
$$
where
\begin{align*}
    &\hat{a}^{ij}(t,y) = a^{rl}(t,\Psi(y))\Phi^i_{x^r}(\Psi(y))\Phi^j_{x^l}(\Psi(y)),\quad \hat{c}(t,y) = c(t, \Psi(y)),\\
&\hat{b}^{i}(t,y) = a^{rl}(t,\Psi(y))\Phi^i_{x^rx^l}(\Psi(y)) + b^r(t,\Psi(y))\Phi^i_{x^r}(\Psi(y))
.
\end{align*}
For $u \in \mathring{\hh}_{p,q,w,0}^{\alpha, 2} ( \Omega_T )$ satisfying $Lu = f$, we have
$$\hat{L}^z \hat{u} =  \hat{f} \quad \text{in}\quad  (0,T) \times B_{r_1}^+,$$
where $\hat{u}(t, \cdot)  = u(t, \Psi(\cdot))$ vanishing on $B_{r_1} \cap \partial \rr^d_+$ and $\hat{f}(t,y) = f(t, \Psi(y))$.
Note that the coefficients $\hat{b}^{i}$'s contain the second-order derivatives of $\Phi$. Therefore, we need a stronger assumption on the smoothness of $\Omega$ for non-divergence form equations.
Next, recall the definition{s} of $\eta$ and $\hat{\eta}^z$ in (\ref{10031}). We extend $\hat{L}^z$ to the half space by letting
$$ \widetilde{L}^z := \hat{L}^z{\eta}^z + \Delta(1-\hat{\eta}^z)  $$
with the leading coefficients
$$\widetilde{a}^{ij} := \Hat{a}^{ij}\hat{\eta}^z + \delta^{ij}(1-\hat{\eta}^z).$$
For $r_2 = \min (1/4,\theta + \gamma_0) r_1$, any $r\le r_2$, and $(t_0,x_0) \in \rr^d_T$, we have
 \[
    \sup_{i,j} \fint_{Q_r (t_0, x_0)} |\widetilde{a}^{ij} - (\widetilde{a}^{ij})_ {Q_r (t_0, x_0)} | \leq N_0(\theta + \gamma_0), \stepcounter{equation}\tag{\theequation}\label{12281}
\]
where $N_0$ is independent of {$\theta,\gamma_0$ and $z$}. Therefore, by picking sufficiently small $\theta$ and $\gamma_0$ depending on $d, \delta, \alpha, p, q, K_1$ such that $N_0(\theta + \gamma_0)$ is less than the $\gamma_0$ in Lemma \ref{5.3}, we can apply the results in the half space.

Then, we take a partition of unity on $\overline{\Omega}$ as {follows}. Let
$$
\Omega^{r_0/20}= \{x\in \Omega, \mathrm{dist}\,(x, \partial\Omega) \geq r_0/20 \},
$$
and take
$$
\{x_k\}_{k=1}^\infty \subset \Omega^{r_0/20} \cup \partial \Omega,\quad
\xi^k \geq 0,\, \quad \xi^k \in C_0^\infty (\rr^d)
$$
such that
\begin{align*}
\supp (\xi^k) \subset
\begin{cases}
B_{{r_0}/{40}}(x_k) \quad   x_k \in \Omega^{r_0/20}, \\
B_{{r_0}/{12}}(x_k) \quad   x_k \in \partial \Omega,
\end{cases}\quad {\sum_k \xi^k=1\quad\text{in}\,\,\Omega},
\stepcounter{equation}\tag{\theequation}\label{partition}
\end{align*}
and let
$$
\zeta^k = \frac{\xi^k}{(\sum_k (\xi^k)^{2})^{1/2}},
$$
such that there exists an upper bound on the number of non-vanishing $\zeta^k$'s at each $x\in \Omega$.
It follows by H\"older's inequality that
\begin{align*}
    \ep(d,p) \le \sum_{k=1}^\infty |\zeta^k|^p, \sum_{k=1}^\infty |\xi^k|^p \le 1, \quad \sum_{k=1}^\infty |D_x \xi^k |^p, \sum_{k=1}^\infty |D_x \zeta^{k} |^p\leq N(d,p,r_0).\stepcounter{equation}\tag{\theequation}\label{eq8.01}
\end{align*}
Furthermore, by \eqref{12281}, we pick $\theta$ and $\gamma_0$ sufficiently small such that the assumptions of Lemma \ref{5.3} are satisfied. Then, for $\lambda_0$ defined in Lemma \ref{5.3} and any $\lambda \ge \lambda_0$,  let $$R_\lambda^k(\cdot) =
\begin{cases}
\Phi ( \widetilde{L}^{x_k}+ \lambda)^{-1} \Phi^{-1}[\eta^{x_k} \cdot] \quad \text{for}\quad x_k \in \partial \Omega,\\
 (L+ \lambda)^{-1 }(\cdot) \quad \text{for}\quad x_k \in \Omega^{r_0/20},
\end{cases}
$$ which is essentially the inverse of operator $ L + \lambda$.
See \cite[Section 8.3]{k} for more details. Note that in \cite[Section 8.3]{k}, the weight is taken to be $1$. For general weights, in order to use the results in the whole space and the half space, we need to take restrictions and extensions of the weight given in Lemma \ref{exten}.

\begin{lemma}\label{A3}
For $\lambda \ge \lambda_0$, $u \in \mathring{\hh}^{\alpha,2}_{p,w,0}(\Omega_T)$, and $f = Lu + \lambda u $, we have
$$
u = \sum_{k} \zeta^k R_\lambda^k (\zeta^k f - L^k u ),
$$
where
$$
L^k u := {  \zeta^k(Lu) - L(\zeta^{k}u) } = u(a^{ij}D_{ij}\zeta^k+b^iD_i\zeta^k) + 2a^{ij}D_i{\zeta^k}D_ju.
$$
\end{lemma}
\begin{proof}
Similar to \cite[Theorem 8.3.7]{k}, we have $\zeta^k u = R_\lambda^k((L+\lambda)(\zeta^k u))$ for any $k$. Thus,
$$u = \sum_{k} \zeta^k \zeta^k u =  \sum_{k} \zeta^k R_\lambda^k((L+\lambda)(\zeta^k u))= \sum_{k} \zeta^k R_\lambda^k (\zeta^k f - L^k u).$$
{We omit the details.}
\end{proof}

\begin{lemma}\label{A4}
Let $\lambda \ge \lambda_0$, $f\in L_{p,w}(\Omega_T)$, and $u\in\mathring{\hh}^{0,1}_{p,w,0}(\Omega_T)$ such that
$$
u = \sum_{k} \zeta^k R_\lambda^k (\zeta^k f - L^k u ).
$$
Then, $u \in \mathring{\hh}^{\alpha,2}_{p,w,0}(\Omega_T)$, and
there exists $\lambda_1(d, \delta, \alpha, p, q, K_1, R_0, R_2)\ge 1$ such that if $\lambda \ge \lambda_1$, then
$$ Lu + \lambda u = f.
$$
\end{lemma}
\begin{proof}
The first assertion follows from the definition of $ R_\lambda^k$. For the second assertion, we use the fact that the mapping
$$\sum_{k} L^k R_\lambda^k (\zeta^k\cdot): L_{p,w}(\Omega_T) \longrightarrow L_{p,w}(\Omega_T)$$ is a contraction when $\lambda$ is sufficiently large, which follows from the a priori estimate in Lemma \ref{5.3} and the observation that $L^k$ contains only the zeroth and the first order terms of $u$. See \cite[Lemma 8.5.1]{k} for details. A similar proof for equations in divergence form is given {below} in Lemma \ref{l2}.
\end{proof}
\begin{lemma}\label{A5}
There exists a constant $\lambda_1(d, \delta, \alpha, p, q, K_1, R_0, R_2)\ge 1$ such that for any $\lambda \ge \lambda_1$ and $f\in L_{p,w}(\Omega_T) $, there exists a unique $u\in \mathring{\hh}^{0,1}_{p,w,0}(\Omega_T)$ to
$$
u = \sum_{k} \zeta^k R_\lambda^k (\zeta^k f - L^k u ),
$$
and $u$ satisfies
\begin{align*}
\lambda \norm{u }_{L_{p,w}(\Omega_T)} + \lambda^{1/2}\norm{Du }_{L_{p,w}(\Omega_T)} \le  N\norm{f }_{L_{p,w}(\Omega_T)}. \stepcounter{equation}\tag{\theequation}\label{oo}
\end{align*}
\end{lemma}
\begin{proof}
It suffices to prove that $$\sum_{k} \zeta^k R_\lambda^k (\zeta^k f - L^k \cdot) : \mathring{\hh}^{0,1}_{p,w,0}(\Omega_T)
\longrightarrow \mathring{\hh}^{0,1}_{p,w,0}(\Omega_T)$$ is a contraction. For this, we prove
\begin{align*}
\lambda \norm{u }_{L_{p,w}(\Omega_T)} + \lambda^{1/2}\norm{Du }_{L_{p,w}(\Omega_T)}\le N(  \norm{f }_{L_{p,w}(\Omega_T)} + \norm{v }_{\hh^{0,1}_{p,w}(\Omega_T)}),\stepcounter{equation}\tag{\theequation}\label{qq1}
\end{align*}
where $$ u = \sum_{k} \zeta^k R_\lambda^k (\zeta^k f - L^k v).$$
Indeed, by the definition of $R_\lambda^k$ and Lemma \ref{5.3}, we have
\begin{align*}
\lambda^p \norm{u }_{L_{p,w}(\Omega_T)}^p &\le N \sum_{k} \lambda^p \norm{\zeta^k R_\lambda^k (\zeta^k f - L^k v)}_{L_{p,w}(\Omega_T)}^p\\
&\le N \sum_{k} \norm{\zeta^k f - L^k v}_{L_{p,w}(\Omega_T)}^p\\
&\le N (\norm{f }_{L_{p,w}(\Omega_T)}^p + \norm{v }_{\hh^{0,1}_{p,w}(\Omega_T)}^p),
\end{align*}
{where we used \eqref{eq8.01} in the last inequality.}
Similarly,
$$ \lambda^{1/2}\norm{Du }_{L_{p,w}(\Omega_T)} \le N (\norm{f }_{L_{p,w}(\Omega_T)} + \norm{v }_{\hh^{0,1}_{p,w}(\Omega_T)}).$$
Thus, \eqref{qq1} follows, and the lemma is proved.
\end{proof}

Combining Lemmas \ref{A4} and \ref{A5}, we conclude that there exists $\lambda_1$ as in Lemma \ref{A5} such that for any $\lambda \ge \lambda_1$ {and} $g \in L_{p,w} (\Omega_T)$, there exists a unique solution
$u \in \mathring{\hh}_{p,w,0}^{\alpha, 2} ( \Omega_T)$ satisfying
\[
\partial_t^\alpha u - a^{ij} D_{ij} u - b^i D_i u - cu + \lambda u  = g.
\]
If $\lambda \ge \lambda_1$, (\ref{729}) follows from (\ref{oo}). If $\lambda \in [0,\lambda_1)$, let $h := Lu + \lambda_1 u $. We have
\begin{align*}
\norm{u}_{\hh_{p,w}^{\alpha,2} (\Omega_T )}^p &\le N \sum_{k} \norm{\zeta^k R_{\lambda_1}^k (\zeta^k {h} - L^k u)}_{\hh_{p,w}^{\alpha,2} (\Omega_T )}^p\\
&\le N \sum_{k} \norm{\zeta^k h - L^k u}_{L_{p,w}(\Omega_T)}^p\\
&\le N (\norm{h}_{L_{p,w}(\Omega_T)}^p + \norm{u}_{L_{p,w}(\Omega_T)}^p + \norm{Du}_{L_{p,w}(\Omega_T)}^p).
\end{align*}
Thus, by {the interpolation inequality},
\begin{align*}
\norm{u}_{\hh_{p,w}^{\alpha,2} (\Omega_T )} &\le N (\norm{Lu + \lambda u  }_{L_{p,w}(\Omega_T)} + |\lambda_1 - \lambda|\norm{u}_{L_{p,w}(\Omega_T)} +\norm{u}_{L_{p,w}(\Omega_T)} )\\
&\le N (\norm{Lu + \lambda u  }_{L_{p,w}(\Omega_T)} + \norm{u}_{L_{p,w}(\Omega_T)}).
\end{align*}

\subsection{Results for divergence form equations}
Next, we use a similar process to prove the results for equations in divergence form. For $z \in \partial \Omega$, we
denote
\begin{align*}
   & Lu = \partial_t^\alpha u - D_i (a^{ij} D_ju  + a^i u) - b^i D_iu - cu, \\
    &\hat{L}^z \hat{u} = \partial_t^\alpha \hat{u} - D_i (\hat{a}^{ij} D_j\hat{u}  - \hat{a}^i \hat{u}) - \hat{b}^i D_i\hat{u} - \hat{c}\hat{u},
\end{align*}
where
\begin{align*}
&\hat{a}^{ij}(t,y) = a^{rl}(t,\Psi(y))\Phi^i_{x^r}(\Psi(y))\Phi^j_{x^l}(\Psi(y)),\quad \hat{c}(t,y) = c(t, \Psi(y)),\\
&\hat{b}^{i}(t,y) = b^r(t,\Psi(y))\Phi^i_{x^r}(\Psi(y)), \quad \hat{a}^{i}(t,y) = a^r(t,\Psi(y))\Phi^i_{x^r}(\Psi(y)),
\end{align*}
where $\Phi$ and $\Psi$ are defined in \eqref{1030}. For $u \in \mathring{\chh}_{p,q,w,0}^{\alpha, 1} ( \Omega_T )$ satisfying $Lu = D_ig_i + f$, we have
$$\hat{L}^z \hat{u} =  D_i\hat{g}_i + \hat{f} \quad \text{in}\quad  (0,T) \times B_{r_1}^+,$$
where $\hat{u}(t, \cdot)  = u(t, \Psi(\cdot))$ vanishing on $B_{r_1} \cap \partial \rr^d_+$,
$$
\hat{g}_i = g^r(t,\Psi(y))\Phi^i_{x^r}(\Psi(y)), \quad \text{and}\quad  \hat{f}(t,y) = f(t, \Psi(y)).
$$

Next, similar to Section B.1, we define $ \widetilde{L}^z$ and $\widetilde{a}^{ij}$ {with the same $r_2$} {so that} \eqref{12281} {still holds}. Furthermore, we pick $\theta$ and $\gamma_0$ sufficiently small such that the assumptions of Lemma \ref{91566} are satisfied. Then, for $\lambda_0$ defined in Lemma \ref{5.3} and any $\lambda \ge \lambda_0$, let $$R_\lambda^k(\cdot) =
\begin{cases}
\Phi ( \widetilde{L}^{x_k} + \lambda)^{-1} \Phi^{-1}[\eta^{x_k} \cdot] \quad \text{for}\quad x_k \in \partial \Omega,\\
(L + \lambda)^{-1 }(\cdot) \quad \text{for}\quad x_k \in \Omega^{r_0/20},
\end{cases}
$$ which is essentially the inverse of operator $L+\lambda $.
\begin{lemma}\label{l1}
For $\lambda \ge \lambda_0$, $u \in \mathring{\chh}^{\alpha,1}_{p,w,0}(\Omega_T)$, and $D_ig_i + f = Lu + \lambda u $, we have
$$u = \sum_{k} \zeta^k R_\lambda^k (D_i(\zeta^kg_i)-g_iD_i\zeta^k  + \zeta^k f - L^k u ),$$
where
\[
L^k u {:=\zeta^k L u- L (\zeta^k u)}= D_i(a^{ij}uD_j\zeta^k)+ (a^{ij}D_ju + a^iu+b^iu)D_i\zeta^k. \stepcounter{equation}\tag{\theequation}\label{ee2}
\]
\end{lemma}
\begin{proof}
Since $\zeta^k u = R_\lambda^k((L+\lambda)(\zeta^k u))$ for any $k$,
\begin{align*}
u &= \sum_{k} \zeta^k \zeta^k u =  \sum_{k} \zeta^k R_\lambda^k((L+\lambda)(\zeta^k u))\\
&= \sum_{k} \zeta^k R_\lambda^k (D_i(\zeta^kg_i)-g_iD_i\zeta^k  + \zeta^k f - L^k u ).
\end{align*}
{The lemma is proved.}
\end{proof}
\begin{lemma}\label{l2}
Let $\lambda \ge \lambda_0$, $f,g_i\in L_{p,w}(\Omega_T)$, and $u\in\mathring{\hh}^{0,1}_{p,w,0}(\Omega_T)$ such that
$$ u = \sum_{k} \zeta^k R_\lambda^k (D_i(\zeta^kg_i)-g_iD_i\zeta^k + \zeta^k f - L^k u ).$$
Then, $u \in \mathring{\chh}^{\alpha,1}_{p,w,0}(\Omega_T)$, and there exists $\lambda_1(d, \delta, \alpha, p, q, K_1, R_0, R_2)$  such that if $\lambda \ge \lambda_1$, then
$$
Lu + \lambda u = D_ig_i + f.
$$
\end{lemma}
\begin{proof}
The first assertion follows from the definition of $ R_\lambda^k$. For the second assertion, let
$$ u := \sum_{k} \zeta^k R_\lambda^k (D_i(\zeta^kg_i)-g_iD_i\zeta^k + \zeta^k f - L^k u )
$$
and
$$ Lu + \lambda u =: D_i \overline{g_i} + \overline{f} .$$
Then, by Lemma \ref{l1},
$$ u = \sum_{k} \zeta^k R_\lambda^k (D_i(\zeta^k\overline{g_i})-\overline{g_i}D_i\zeta^k + \zeta^k \overline{f} - L^k u ).$$ It is sufficient to prove if $\lambda$ is sufficiently large, and
$$ T_\lambda(D_ig_i +f):= \sum_{k} \zeta^k R_\lambda^k (D_i(\zeta^kg_i)-g_iD_i\zeta^k + \zeta^k f) = 0,$$
then ${D_ig_i+}f = 0$. Given $T_\lambda(D_ig_i +f)=0$, we have
\begin{align*}
0 = (\lambda + L ) T_\lambda(D_ig_i +f) = D_ig_i +f - \sum_{k} L^k R_\lambda^k (D_i(\zeta^kg_i)-g_iD_i\zeta^k  + \zeta^k f).
\end{align*}
To finish the proof, it suffices to prove the mapping
\[\sum_{k} L^k R_\lambda^k (\zeta^k\cdot): \hh^{-1}_{p,w}(\Omega_T) \longrightarrow \hh^{-1}_{p,w}(\Omega_T)
\stepcounter{equation}\tag{\theequation}\label{ee1}\]
is a contraction when $\lambda$ is  sufficiently large,
where for $h = D_ig_i + f \in \hh^{-1}_{p,w}(\Omega_T)$ {with} some $f,g_i \in L_{p,w}(\Omega_T)$ in the distribution sense, we write
$$\zeta^k h := D_i(\zeta^kg_i)-g_iD_i\zeta^k + \zeta^k f \in \hh^{-1}_{p,w}(\Omega_T).$$
To this end, we take the following equivalent norm in $\hh^{-1}_{p,w}(\Omega_T)$:
$$ \norm{h}_{\hh^{-1}_{p,w}(\Omega_T)} :=  \inf \Big\{ \sum_{i=1}^d \norm{g_i}_{L_{p,w}} + \frac{1}{\sqrt{\lambda}}\norm{f}_{L_{p,w}}: h =D_ig_i + f \Big\},$$
where $\norm{\cdot}_{p,w} = \norm{\cdot}_{L_{p,w}(\Omega_T)}$.

For any $f,g_i \in L_{p,w}(\Omega_T)$ and $h = D_ig_i + f$, by \eqref{ee2}, we have
$$L^k { U_{k}}  = D_i(a^{ij}{ U_{k}}D_j\zeta^k)+ (a^{ij}D_j{ U_{k}} + a^i{ U_{k}}+b^i{ U_{k}})D_i\zeta^k,$$
where ${ U_{k}} = R_\lambda^k (D_i(\zeta^kg_i)-g_iD_i\zeta^k + \zeta^k f)$. Then, by \eqref{eq8.01} and Lemma \ref{91566},
\begin{align*}
&\sum_{k}  \norm{a^{ij}U_{k}D_j\zeta^k)}^{p}_{{p,w}} + \frac{1}{\lambda^{p/2}} \norm{(a^{ij}D_jU_{k} + a^iU_{k}+b^iU_{k})D_i\zeta^k}^{p}_{p,w} \\
&\le N\sum_{k} \big\{ \norm{R_\lambda^k (D_i(\zeta^kg_i)-g_iD_i\zeta^k + \zeta^k f)}^{p}_{{p,w}}\\
& \quad \quad + \frac{1}{\lambda^{p/2}} \norm{{D}R_\lambda^k (D_i(\zeta^kg_i)-g_iD_i\zeta^k + \zeta^k f)}^{p}_{p,w} \big\}\\
& \le N  \sum_{k}   \frac{1}{\lambda^{p/2}}\norm{\zeta^k g_i}^{p}_{{p,w}} + \frac{1}{\lambda^{p}}\norm{\zeta^k f}^{p}_{{p,w}} + \frac{1}{\lambda^{p}}\norm{g_iD_i\zeta^k}^{p}_{{p,w}}
\\
&\le N  \Big(\frac{1}{\lambda^{p/2}}\norm{ g_i}^{p}_{{p,w}} + \frac{1}{\lambda^{p}}\norm{f}^{p}_{{p,w}}\Big)
= \frac{N}{\lambda^{p/2}} \Big(\norm{ g_i}^{p}_{{p,w}} + \frac{1}{\lambda^{p/2}}\norm{f}^{p}_{{p,w}}\Big),
\end{align*}
where $N = N(d,\delta,p,K_1,R_0,R_2)$.
Thus, by the Minkowski inequality,
$$
\norm{\sum_{k} L^k R_\lambda^k ( \zeta^k {h})}^{p}_{\hh^{-1}_{p,w}(\Omega_T)} \le \frac{N}{\lambda^{p/2}} \norm{{h}}^{p}_{\hh^{-1}_{p,w}(\Omega_T)}.
$$
Therefore, by picking $\lambda$ sufficiently large, this mapping \eqref{ee1} {is} a contraction, and the lemma is proved.
\end{proof}
\begin{lemma}\label{l3}
There exists a constant $\lambda_1(d, \delta, \alpha, p, q, K_1, R_0, R_2)\ge 1$ such that for any $\lambda \ge \lambda_1$ and $f,g_i\in L_{p,w}(\Omega_T) $, there exists a unique $u\in \mathring{\hh}^{0,1}_{p,w,0}(\Omega_T)$ to
$$
u = \sum_{k} \zeta^k R_\lambda^k (D_i(\zeta^kg_i)-g_iD_i\zeta^k + \zeta^k f - L^k u ),
$$
and $u$ satisfies
\begin{align*}
\sqrt{\lambda} \norm{u }_{p,w} +\norm{Du }_{p,w} \le  N \norm{g }_{p,w} + \frac{N}{\sqrt{\lambda}}\norm{f }_{p,w},
\end{align*}
where $\norm{\cdot}_{p,w} = \norm{\cdot}_{L_{p,w}(\Omega_T)} $.
\end{lemma}
\begin{proof}
In order to prove the existence of solutions, it suffices show $$\sum_{k} \zeta^k R_\lambda^k (D_i(\zeta^kg_i)-g_iD_i\zeta^k + \zeta^k f  - L^k \cdot) : \mathring{\hh}^{0,1}_{p,w,0}(\Omega_T)
\longrightarrow \mathring{\hh}^{0,1}_{p,w,0}(\Omega_T)$$ is a contraction when $\lambda$ is  sufficiently large. To this end, we prove
\begin{equation}\label{eq10.17}
\lambda \norm{u }_{p,w} + \sqrt{\lambda}\norm{Du }_{p,w} \le N(   \sqrt{\lambda} \norm{g}_{L_{p,w}} + \norm{f}_{L_{p,w}}+ \norm{v }_{\hh^{0,1}_{p,w}(\Omega_T)}),
\end{equation}
where $$ u = \sum_{k} \zeta^k R_\lambda^k (D_i(\zeta^kg_i)-g_iD_i\zeta^k + \zeta^k f  - L^k v),$$ and $N$ is independent of $\lambda$.
Indeed, by the definition of $R_\lambda^k$, {\eqref{eq8.01},} and Lemma \ref{91566}, we have
\begin{align*}
\lambda^{p/2} \norm{u }_{p,w}^p
&\le N \sum_{k} \lambda^{p/2} \norm{\zeta^k R_\lambda^k (D_i(\zeta^kg_i)-g_iD_i\zeta^k + \zeta^k f  - L^k v)}_{p,w}^p\\
&\le N \sum_{k} \norm{\zeta^k g}_{p,w}^p + \frac{1}{\lambda^{p/2}}\norm{\zeta^k f }_{p,w}^p +  \frac{1}{\lambda^{p/2}}\norm{L^k v}_{p,w}^p \\
&\le N \Big(\norm{g}_{p,w}^p + \frac{1}{\lambda^{p/2}}\norm{f }_{p,w}^p + \frac{1}{\lambda^{p/2}}\norm{v }_{\hh^{0,1}_{p,w}(\Omega_T)}^p\Big),
\end{align*}
and similarly,
\begin{align*}
\norm{Du }_{p,w}^p
&\le N \sum_{k} \norm{D(\zeta^k R_\lambda^k (D_i(\zeta^kg_i)-g_iD_i\zeta^k + \zeta^k f  - L^k v))}_{p,w}^p\\
&\le N \Big(\norm{g}_{p,w}^p + \frac{1}{\lambda^{p/2}}\norm{f }_{p,w}^p + \frac{1}{\lambda^{p/2}}\norm{v }_{\hh^{0,1}_{p,w}(\Omega_T)}^p\Big),
\end{align*}
which imply \eqref{eq10.17}.
The lemma is proved.
\end{proof}

\section{}\label{C}
In this section, we prove Corollary \ref{initial} by applying Theorem \ref{main} to
$$\Tilde{f} := f + (a^{ij} D_{ij} U + b^i D_i U + cU) - D_t^\alpha (U - u_0 - tu_1),$$
where $U$ is given in Lemma \ref{ioi}. Indeed, by Theorem \ref{main} there exists a unique $w \in \hh_{p,q,w,0}^{\alpha, 2} ( {\rr^d_T} ) $ satisfying
\[
\partial_t^\alpha w - (a^{ij}  D_{ij} w + b^i D_i w + cw ) = \Tilde{f} \quad \text{in}\quad  {\rr^d_T}.
\]
Then, it follows that $u := w + U$ is a solution to the equation \eqref{cong}, and the estimates \eqref{88998} follows {from} \eqref{99999} and \eqref{jjj}.

{Similarly}, Corollary \ref{initial2} follows by taking the derivative{s} of $U$ in the spatial variables and repeating the above process together with Theorem \ref{6main}. Indeed, by the following lemma, we have $D_t^\alpha (U - u_0 - tu_1)\in L_{p,q,w} ( {\rr^d_T} )$, which implies $D_t^\alpha (DU - Du_0 - tDu_1)\in \hh^{-1}_{p,q,w} ( {\rr^d_T} )$.

\begin{lemma}\label{ioi}
Under the assumptions of Theorem \ref{initial}, for any $u_0 \in \cxx_0$ and $u_1 \in \cxx_1$, there exists a function $U$ in ${\rr^d_T} $ such that $U(0,\cdot) = u_0 (\cdot)$ and
$U_{t}(0,\cdot) = u_1 (\cdot)$
in the trace sense and
\[
\norm{ \partial_t^\alpha U}_{p,q,w}
+ \norm{U}_{p,q,w}
+\norm{DU}_{p,q,w}
+\norm{D^2 U}_{p,q,w}
\leq N \norm{u_0}_{\cxx_0} + N \norm{u_1}_{\cxx_1}, \stepcounter{equation}\tag{\theequation}\label{jjj}
\]
where $\norm{\cdot}_{p,q,w} = \norm{\cdot}_{L_{p,q,w} ( {\rr^d_T} )} $ and $N = N (d, \alpha, p, q, K_1, T, \mu )$.
\end{lemma}
\begin{proof}
We can assume that $u_0 (x)$ and $u_1 (x)$ are infinitely differentiable with compact support.
Then, let $p = p(t,x)$ be the fundamental solution of $\partial_t^\alpha - \Delta$, and let $\Tilde{p}(t,x) := I_0p(t,x)$. See \cite[Section 6.2]{P1} for the construction of $p$. For the existence of $U$, we define $$U(t, \cdot) := p(t, \cdot) \ast u_0 (\cdot) + \Tilde{p}(t, \cdot) \ast u_1 (\cdot).$$ Then by \cite[Section 1]{Kochubei}, $U$ is a solution to
\begin{align*}
    \begin{cases}
    \partial_t^\alpha U - \Delta U = 0 &\quad \text{in}\quad  {\rr^d_T} \\
    U(0, \cdot) = u_0 (\cdot)&\quad \text{on}\quad  \rr^d\\
    U_t(0, \cdot) = u_1 (\cdot) &\quad \text{on}\quad  \rr^d
    \end{cases} .
\end{align*}

Next, we prove the estimate \eqref{jjj} following the method in \cite[Lemma 5.7]{dong20}. Without loss of generality, we assume that $u_0 =0$, i.e., $ U = \Tilde{p}(t, \cdot) \ast u_1 (\cdot)$. Indeed, in order to estimate $U$, we estimate $p(t, \cdot) \ast u_0 (\cdot)$ and $ \Tilde{p}(t, \cdot) \ast u_1 (\cdot)$ separately. The estimate of the first term is similar to that of the second term, and it is given in \cite[Lemma 5.7]{dong20}.

By \cite[Section 6.2]{P1} and \cite[Section 2.2]{Kochubei}, we have the following normalization property and estimates of $\Tilde{p}$
\[
\Tilde{p}(t,x) = t^{-\alpha d / 2 + 1}  \Tilde{p}(1, t^{-\alpha / 2} x), \stepcounter{equation}\tag{\theequation}\label{normal}
\]
for any $m = 0,1,2, 3$ and $x \neq 0$,
\begin{align*}
&
|D^m \Tilde{p} (1,x) |\leq N 1_{|x| \geq 1} e^{- \sigma |x|^{\frac{2}{2-\alpha}}} \\
&\qquad+N 1_{|x| < 1} |x|^{-d - m}
( |x|^2 + |x|^2 | \log |x| | 1_{d = 2, m=0} + |x| 1_{d=1, m=0} ), \stepcounter{equation}\tag{\theequation}\label{es}
\end{align*}
and
\[
|\Delta \Tilde{p}(1,x) |
\leq N 1_{|x| \geq 1} e^{-\sigma |x|^{\frac{2}{2-\alpha} } }
+ N 1_{|x| < 1} (
|x|^{-d + 2} 1_{d \geq 3} + | \log |x| | 1_{d=2} + 1
) ,
 \stepcounter{equation}\tag{\theequation}\label{lap}
\]
where $N = N(d, \alpha, m)$ and $\sigma = \sigma (d, \alpha, m)>0$.

To prove \eqref{jjj}, we start with the estimate of $U$.
From \eqref{normal} and  \eqref{es} together with a {suitable} dyadic decomposition {of $\rr^d$}, it follows that for any $t > 0$ and $x \in \rr^d$,
\[
|U(t,x)| \leq   t N \cmm_x u_1(x),
\]
where $\cmm_x$ is the maximal operator with respect to $x$ and $N$ is independent of $t$. Therefore, by the Hardy-Littlewood maximal function theorem with $A_p$ weights,
\[
\norm{U (t, \cdot) }_{L_{q, w_2} (\rr^d)}
\leq  tN \norm{u_1}_{L_{q, w_2} (\rr^d)},
\]
where $N = N(d, \alpha, q, K_1 )$, and
\[
\norm{U}_{ L_{p,q,w,} ({\rr^d_T}  )  } \leq   N \norm{u_1}_{L_{q, w_2} (\rr^d)},\stepcounter{equation}\tag{\theequation}\label{ppp}
\]
where $N = N(d, \alpha, q, K_1, T, \mu)$.

Next, we estimate $D^2 U$. By \eqref{normal} and  \eqref{es}, we have $ \Tilde{p}(t, \cdot) , D\Tilde{p}(t, \cdot) \in L^1 (\rr^d)$, which implies
\[
\int_{\rr^d} D\Tilde{p} (t,y) \, dy = 0.
\]
Then, by integration by parts,
\begin{align*}
D^2 U (t,x) &=   \int_{\rr^d} \Tilde{p} (t,y) D^2 u_1 (x-y) \, dy =  \int_{\rr^d} D\Tilde{p}(t,y) D u_1 (x- y) \, dy \\
&=   \int_{\rr^d}  D\Tilde{p}(t,y) (t,y) (Du_1(x-y) - D u_1 (x)) \, dy . \stepcounter{equation}\tag{\theequation}\label{2orderes}
\end{align*}
Moreover, we observe that
\[
\Delta U (t,x) =  \int_{\rr^d}
\Delta \Tilde{p} (t,y) (u_1 (x-y) - u_1 (x)) \, dy. \stepcounter{equation}\tag{\theequation}\label{laprewrite}
\]
Indeed, from \eqref{2orderes} we have
\begin{align*}
\Delta U (1,x) &= - \int_{\rr^d} \nabla \Tilde{p}(1,y) \cdot \nabla_y (u_1 (x-y) - u_1 (x)) \, dy \\
&=  \lim_{\ep \searrow 0} \int_{\rr^d \setminus B_\ep } \nabla \Tilde{p} (1,y) \cdot \nabla_y (u_1 (x) - u_1 (x-y)) \, dy\\
&=: \lim_{\ep \searrow 0}  L_\ep.
\end{align*}
An integration by parts yields
\begin{align*}
L_\ep = \int_{\rr^d \setminus B_\ep}
\Delta \Tilde{p} (1,y) (u_1 (x-y) - u_1(x)) \, dy
- \int_{\partial B_\ep} \frac{\partial \Tilde{p}}{\partial \nu} (1, y) (u_1(x-y) -u_1 (x)) \, dS.
\end{align*}
By the continuity of $u_1$ and the estimates in \eqref{es}, it is seen that the second term in the above expression vanishes
as $\ep \to 0$. Furthermore, \eqref{lap} implies that $ \Delta\Tilde{p}(1,\cdot) \in L_1 (\rr^d)$. Thus, using the dominated convergence theorem, we have
\begin{align*}
\Delta U(1,x) &=   \int_{\rr^d} \Delta \Tilde{p} (1,y) (u_1 (x-y) - u_1 (x)) \, dy. \stepcounter{equation}\tag{\theequation}\label{apple}
\end{align*}
Therefore, \eqref{laprewrite} follows {from} \eqref{normal}.

Then, we consider the following four cases.

{\em Case 1.} $\theta_1 \in (0,1)$. In this case,
\begin{align*}
     \alpha p \in (1 + \mu + p, 2 (1 + \mu + p)). \stepcounter{equation}\tag{\theequation}\label{case1}
\end{align*}
By the weighted Calder\'on-Zygmund estimate and the Minkowski inequality applied to \eqref{laprewrite}, we have
\begin{align*}
\norm{D^2 U (t, \cdot) }_{L_{q, w_2}}
&\leq N \norm{\Delta U (t, \cdot)}_{L_{q, w_2}} \\
&\leq \int_{\rr^d} | \Delta \Tilde{p} (t,y) | \norm{u_1 ( \cdot - y) - u_1 (\cdot) }_{ L_{q, w_2}} \, dy.
\end{align*}
Therefore, by H\"older's inequality, for any $\beta \in (0,1)$, we have
\begin{align*}
&\norm{D^2 U }^p_{L_{p,q, w}}\le \int_0^T \left(
\int_{\rr^d} | \Delta \Tilde{p} (t,y) | \norm{u_1 ( \cdot - y) - u_1 (\cdot) }_{ L_{q, w_2}} \, dy
\right)^p t^\mu \, dt\\
&\leq N \int_0^T \left(
\int_{\rr^d} | \Delta  \Tilde{p}(t,y) |^{p (1 - \beta) }
\norm{u_1 ( \cdot - y) - u_1 (\cdot) }^p_{ L_{q, w_2}} \, dy
\right) \cdot J_1 t^\mu \, dt,\stepcounter{equation}\tag{\theequation}\label{ho}
\end{align*}
where
\begin{align*}
 J_1 := \left(
\int_{\rr^d}  | \Delta \Tilde{p} (t,y) |^{ \beta p/ (p- 1) } \, dy
\right)^{p-1}.
\end{align*}
By taking $\beta < (p-1) / p$, \eqref{normal} and  \eqref{es} {imply} that
\begin{align*}
 J_1 \leq Nt^{-\alpha \beta (d+2) p / 2 + \alpha d (p-1) /2+ \beta p}.
\end{align*}
By \eqref{ho} with the Fubini theorem, we have
\begin{align*}
\norm{D^2 U }^p_{L_{p, q, w}}
&\le N  \int_{\rr^d} J_2  \norm{u_1 ( \cdot - y) - u_1 (\cdot) }^p_{ L_{q, w_2}},
\end{align*}
where
$$J_2 := \int_0^T | \Delta \Tilde{p} (t,y)|^{ p(1- \beta) }
t^{-\alpha \beta (d+2) p / 2 + \alpha d (p-1) /2 + \mu + \beta p} \, dt.
$$
To estimate $J_2$, we first observe that by using the assumption \eqref{case1} and picking $\beta = (p-1)/ p - \ep$ for a sufficiently small $\ep> 0$, the following inequality holds,
\begin{align*}
\mu + 1 + p  + \alpha d ( p (1 - \beta )  - 1) / 2  < \alpha p.
\end{align*}
Then, by \eqref{normal}, \eqref{es}, and splitting the integral $J_2$ into two integrals over $t\le|y|^{2 / \alpha}$ and $t> |y|^{2 / \alpha}$,
we see that
\[
 J_2 \leq N|y|^{-d - p \theta_1},
\]
where $N$ is independent of $T$.
Therefore, it follows that
\begin{align*}
\norm{D^2 U }_{L_{p,q, w}  ( (0,T ) \times \rr^d  }
\leq N \norm{u_1}_{\cxx_1}  \stepcounter{equation}\tag{\theequation}\label{uuu}.
\end{align*}

{\em Case 2.} $ \theta_1 \in (1,2)$. In this case,
\begin{align*}
\alpha p \in (2(1+ \mu+p), \infty). \stepcounter{equation}\tag{\theequation}\label{case22}
\end{align*}
By using \eqref{2orderes} and H\"older's inequality, for any $\beta \in (0,1)$, we have
\begin{align*}
&\norm{D^2 U }^p_{L_{p, q, w}}\\
&\leq \int_0^T \left(
\int_{\rr^d}
|D\Tilde{p} (t,y)|^{p (1 - \beta)}
\norm{ D u_1 ( \cdot - y) - D u_1( \cdot) }^p_{ L_{q, w_2}} \, dy
\right) \cdot J_1 t^\mu \, dt, \stepcounter{equation}\tag{\theequation}\label{case2}
\end{align*}
where
\[
 J_1:= \left( \int_{\rr^d}
|D\Tilde{p} (t,y)|^{\beta p / (p-1)} \, dy
\right)^{p-1}.
\]
By taking $\beta < (p-1) d / (p (d-1))$, from \eqref{normal} and \eqref{es} we get
\[
J_1 \leq N t^{\alpha \beta (d+1) p / 2 + \alpha d (p-1) / 2 + \beta p}.
\]
This together with \eqref{case2} and the Fubini theorem yield{s}
\begin{align*}
    \norm{D^2 U }^p_{L_{p,q, w}}
&\leq
 N
\int_{\rr^d}
J_2
\norm{ D u_1 ( \cdot - y) - D u_1( \cdot) }^p_{ L_{q, w_2}} \, dy,
\end{align*}
where
\[
J_2 := \int_0^T |D\Tilde{p}(t,y)|^{p (1 - \beta)}
t^{-\alpha \beta (d+1) p/2 + \alpha d (p-1)/2 + \mu + \beta p } \, dt.
\]
To estimate $J_2$, we observe that by using the assumption \eqref{case22} and picking $\beta = \min \{ 1, (p-1)d/(p(d-1)) \} - \ep$ for a sufficiently small $\ep> 0$, the following inequality holds
\[
\mu + 1 + p + \alpha (( d-1) p (1-\beta) - d)/2 < \alpha p / 2.
\]
Then, by \eqref{normal}, \eqref{es}, and splitting the integral $J_2$ into two integrals as in Case 1, we see that
\[
 J_2 \leq N |y|^{-d - p(\theta_1 -1)},
\]
which implies \eqref{uuu}.

{\em Case 3.} $\theta_1 < 0$. In this case,
\begin{align*}
 \alpha p \in (0, 1+ \mu + p) \stepcounter{equation}\tag{\theequation}\label{case33}.
\end{align*}
By \eqref{apple}, \eqref{normal}, and \eqref{es}, we have
\[
| \Delta U(t,x) | \leq  N t^{-\alpha+1} \cmm_x u_1 (x) + N t^{-\alpha+1} |u_1 (x) |  .
\]
Then, by applying the weighted Calder\'on-Zygmund estimate and the weighted Hardy-Littlewood maximal function theorem,
\[
\norm{D^2 U (t, \cdot)}_{L_{q, w_2}}
\leq N \norm{\Delta U (t, \cdot)}_{L_{q,w_2}}
\leq  N t^{-\alpha+1} \norm{u_1}_{L_{q, w_2}}.
\]
Therefore, with the assumption \eqref{case33},
\begin{align*}
\norm{D^2 U }^p_{L_{p,q, w}}
&\leq  \int_0^T
\norm{u_1}^p_{L_{q, w_2}}
t^{-\alpha p + p + \mu} \, dt \leq   N(T) \norm{u_1}^p_{L_{q, w_2}},
\end{align*}
which implies \eqref{uuu}.

{\em Case 4.} $\theta_1 = 0$ or $1$. We only consider the case when $\theta_1 = 1$ since the other case is similar.
Observe that in this case $\alpha p = 2(\mu + 1 + p)$. We take $\mu' \in (-1, \mu)$ such that $\alpha p \in ((\mu' + 1 + p),2(\mu' + 1 + p))$, and let $\Tilde{w} = \Tilde{w}_1 (t) w_2 (x)${,} where $\Tilde{w}_1 (t) = t^{\mu'}$.
By Case 2, we have
\[
\norm{D^2 U }_{L_{p, q, \Tilde{w} }({\rr^d_T} )  } \leq N \norm{u_1}_{\cxx_1}.
\]
Since $t^\mu \leq N t^{\mu'}$ for $t \in (0,T)$, \eqref{uuu} follows.

To prove \eqref{jjj}, we estimate $DU$ using the interpolation inequality \cite[Lemma {3.8}]{K&D} together with \eqref{uuu} and \eqref{ppp}. Furthermore, the estimate of $\partial_t^\alpha U$ follows from the equation $\partial_t^\alpha U = \Delta U$. The lemma is proved.
\end{proof}

\bibliographystyle{plain}

\def\cprime{$'$}

\end{document}